\documentclass[3p,10pt,nopreprintline]{elsarticle}

\usepackage[T1]{fontenc}
\usepackage{lmodern}
\usepackage{amsmath,amssymb,amsthm,mathtools}
\usepackage{booktabs,tabularx,array}
\usepackage{microtype}
\usepackage{enumitem}
\usepackage{needspace}
\usepackage{placeins}
\usepackage[hidelinks]{hyperref}
\hypersetup{
  pdftitle={Global strong solutions to the one-dimensional full compressible
  Navier--Stokes equations with temperature-dependent coefficients
  and initial vacuum},
  pdfauthor={Yachun Li, Peng Lu, and Zhaoyang Shang},
  pdfsubject={Global strong solutions for the one-dimensional full
  compressible Navier--Stokes system with initial vacuum},
  pdfkeywords={compressible Navier--Stokes equations;
  temperature-dependent transport coefficients; vacuum; global strong
  solutions; effective viscous flux; exponential relaxation},
  bookmarksnumbered=true
}

\allowdisplaybreaks[1]
\numberwithin{equation}{section}
\biboptions{sort&compress}
\newtheoremstyle{plainjournal}
  {6pt}{6pt}{\itshape}{}{\bfseries}{.}{0.5em}{}
\newtheoremstyle{definitionjournal}
  {6pt}{6pt}{\normalfont}{}{\bfseries}{.}{0.5em}{}

\theoremstyle{plainjournal}
\newtheorem{theorem}{Theorem}[section]
\newtheorem{proposition}[theorem]{Proposition}
\newtheorem{lemma}[theorem]{Lemma}
\newtheorem{corollary}[theorem]{Corollary}
\theoremstyle{definitionjournal}
\newtheorem{definition}[theorem]{Definition}
\newtheorem{remark}[theorem]{Remark}

\newcommand{\I}{(0,1)}

\newcommand{\dd}{\mathop{}\!\mathrm d}
\newcommand{\norm}[1]{\lVert#1\rVert}

\newcommand{\bt}{\bar\theta}
\newcommand{\Dt}{D_t}
\newcommand{\Acal}{\mathcal A}
\newcommand{\Dcal}{\mathcal D}
\newcommand{\Ecal}{\mathcal E}
\newcommand{\Ycal}{\mathcal Y}
\newcommand{\Zcal}{\mathcal Z}
\newcolumntype{L}[1]{>{\raggedright\arraybackslash}p{#1}}
\newcolumntype{Y}{>{\raggedright\arraybackslash}X}

\allowdisplaybreaks[4]

\begin{document}
\begin{frontmatter}

\title{Global strong solutions to the 1D full compressible Navier--Stokes equations with temperature-dependent coefficients
and vacuum}

\author[label1]{Yachun Li}
\address[label1]{School of Mathematical Sciences, MOE--LSC, and SHL--MAC,
Shanghai Jiao Tong University, Shanghai 200240, P.R. China}
\ead{ycli@sjtu.edu.cn}

\author[label2]{Peng Lu}
\address[label2]{Department of Mathematics, Zhejiang Sci-Tech University,
Hangzhou, Zhejiang 310018, P.R. China}
\ead{plu25@zstu.edu.cn}

\author[label3]{Zhaoyang~Shang\corref{cor1}}
\address[label3]{School of Finance, Shanghai Lixin University of Accounting
and Finance, Shanghai 201209, P.R. China}
\cortext[cor1]{Corresponding author.}
\ead{shangzhaoyang@sjtu.edu.cn}

\begin{abstract}
We study the initial--boundary value problem for the one-dimensional full
compressible Navier--Stokes system whose viscosity and heat conductivity are
given by $\mu(\theta)=\mu\theta^\alpha$ and
$\kappa(\theta)=\kappa\theta^\beta$, respectively.  The initial density is
nonnegative and is allowed to vanish.  For $\alpha>2$ and
$\beta>2\alpha+4$, we prove the existence of a global strong solution for
arbitrarily large compatible $H^2$ data, provided that the prescribed
boundary temperature is sufficiently high relative to their size.  No
velocity or temperature perturbation is required to shrink as the boundary
temperature tends to infinity.  A shifted initial-layer weight and an exact
symmetrization of the two material-derivative equations separate the fast
thermal mode from the momentum mode.  The resulting estimates are uniform
with respect to the positive lower bound used in the approximation and
therefore survive the limit to genuine vacuum.  The proof further combines a
relative-entropy identity centered at the conserved mean density, the
physical effective viscous flux, a Zlotnik-type estimate along particle
trajectories, and finite-time transport estimates for the density
derivatives.  The density remains uniformly bounded from above and the
temperature stays uniformly comparable with the prescribed boundary
temperature.  After the initial layer we establish vacuum-compatible
exponential relaxation: the density converges to its conserved mean in every
finite $L^p$, the velocity and temperature perturbation decay in $H^1$, and
the weighted material-derivative energy decays at rate
$c\bar\theta^{1-\alpha}$.  Vacuum points are nevertheless
transported by the bi-Lipschitz flow.  Thus density relaxation cannot occur in
$L^\infty$ or $H^1$; if genuine vacuum is initially present, the $L^2$
relaxation forces the density gradient to grow exponentially.  This
simultaneous low-order equilibration and vacuum-layer sharpening is a
distinctive feature of the problem.
\end{abstract}


\end{frontmatter}

\section{Introduction}\label{sec:introduction}

The motion of a viscous, heat-conducting ideal gas is described by the full
compressible Navier--Stokes equations.  Even in one space dimension, the
simultaneous presence of vacuum and state-dependent transport coefficients
creates a genuinely degenerate coupling: the density multiplying the
material derivatives may vanish, while the highest-order coefficients vary
with the unknown temperature.  The purpose of this paper is to establish a
global strong-solution theory for this regime under a high-temperature
hypothesis without any smallness condition on the compatible initial data.

Let $\rho$, $u$, and $\theta$ denote the density, velocity, and absolute
temperature.  On the interval $\I$ we consider
\begin{equation}\label{eq:NSF}
\left\{
\begin{aligned}
 &\rho_t+(\rho u)_x=0,\\
 &\rho\dot u+(R\rho\theta)_x=(\mu\theta^\alpha u_x)_x,\\
 &\rho\dot\theta+R\rho\theta u_x
   =\mu\theta^\alpha u_x^2+(\kappa\theta^\beta\theta_x)_x,
\end{aligned}\right.
\end{equation}
where $\dot f:=f_t+u f_x$, $R,\mu,\kappa>0$.  The specific heat at constant volume has been
normalized to one.  We impose the no-slip and fixed-temperature conditions
\begin{equation}\label{eq:boundary}
 u(0,t)=u(1,t)=0,
 \qquad \theta(0,t)=\theta(1,t)=\bt>0,
\end{equation}
and the initial condition
\begin{equation}\label{eq:initial}
 (\rho,u,\theta)|_{t=0}=(\rho_0,u_0,\theta_0).
\end{equation}
Below, $\Dt=\partial_t+u\partial_x$ denotes the material derivative.

For constant positive transport coefficients, the one-dimensional theory is
classical.  Global large solutions without vacuum go back to
Kazhikhov and Shelukhin~\cite{KazhikhovShelukhin1977}; see also the
small-perturbation Cauchy theory of Kawashima and
Nishida~\cite{KawashimaNishida1981}.  Kawohl~\cite{Kawohl1985} allowed a
broad class of state-dependent heat conductivities, and subsequent works
developed both strong- and weak-solution theories for temperature-dependent
heat conduction; representative results include
\cite{JenssenKarper2010,PanZhang2015}.  These arguments profit essentially
from either a density lower bound or a viscosity coefficient bounded away
from zero independently of the evolving temperature.

Power-law temperature dependence is motivated by high-temperature gas
models and presents a different analytical problem.  Liu, Yang, Zhao and
Zou~\cite{LiuYangZhaoZou2014} obtained global large classical solutions in a regime
where the adiabatic exponent is close to one.  Wang and
Zhao~\cite{WangZhao2016} treated viscosity and heat conductivity of the form
$h(v)\theta^\alpha$ for sufficiently small $|\alpha|$.  With constant
viscosity and power-law heat conductivity, Huang and Shi~\cite{HuangShi2020}
proved nonlinear exponential stability for arbitrarily large non-vacuum
data.  For bounded intervals, Sun, Zhang and
Zhao~\cite{SunZhangZhao2021} established nonlinear exponential stability
when the temperature exponent in the viscosity is small, and Dong and
Tan~\cite{DongTan2025} subsequently obtained an $H^2$ stability theory for
temperature-dependent transport coefficients.  More recently, Dong and
Guo~\cite{DongGuo2025} studied the Cauchy problem for large data.  In all
these global temperature-dependent-viscosity results, the density remains
strictly positive.  Thus their estimates cannot be applied uniformly to a
sequence whose density lower bound tends to zero.  A theorem-level comparison
with the vacuum results is given in Table~\ref{tab:comparison} below.

The strong-solution theory with vacuum follows a separate line.  Cho and
Kim~\cite{ChoKim2006} established a local theory for full compressible flows
under vacuum-compatible compatibility conditions.  For one-dimensional
fixed-domain systems, Wen and Zhu~\cite{WenZhu2013} constructed global
classical large solutions with vacuum for constant viscosity and
temperature-dependent heat conductivity, whereas Li~\cite{Li2019} proved
global $H^2$ well-posedness for nonnegative density when both transport
coefficients are constant.  Neither result covers a viscosity of the form
$\theta^\alpha$.  Zhang and Fang~\cite{ZhangFang2008} developed a
multidimensional local vacuum theory for state-dependent coefficients, but
their hypotheses and solution framework do not provide the uniform global
estimates required here.

Several global results with variable transport coefficients are closely
related but structurally different.  At the finite-energy level,
Feireisl~\cite{Feireisl2004Motion} established global variational solutions
to the full system with viscosity coefficients depending on the temperature.
Wang and Zuo~\cite{WangZuo2022} treated three-dimensional
global weak solutions with temperature-dependent viscosity and heat
conductivity for strictly positive initial density.  In the strong-solution
class, Yu and Zhang~\cite{YuZhang2020} allowed initial vacuum in a
three-dimensional bounded domain with density--temperature-dependent
viscosities, under a small-gradient condition.  These multidimensional weak
or small-data theories use structures different from the one-dimensional
large-data estimates developed below.

Vacuum free-boundary problems provide another important comparison.  Li, Ma
and Ou~\cite{LiMaOu2019} obtained global strong solutions when the viscosity
depends on the density and the heat conductivity on the temperature.  Chen
and Zhang~\cite{ChenZhang2021} treated Chapman--Enskog temperature power laws
for both coefficients and proved global strong solvability for data near a
background solution.  In these free-boundary settings, however, the density
is positive in the moving fluid interior and vanishes at its boundary;
there is no persistent interior vacuum point in a fixed spatial interval.

For pure temperature powers and vacuum in fixed or unbounded domains, Cao and
Li~\cite{CaoLi2022} proved local strong solvability in a three-dimensional
bounded domain under a small-measure condition on the initial vacuum set.
More recently, Cao and Jiang~\cite{CaoJiang2026} obtained local-in-time
regular solutions for a two-dimensional Cauchy problem with far-field vacuum
when viscosity and heat conductivity obey the same temperature power.  A
related global incompressible theory is available in~\cite{HeLi2024}.
These results leave
open the fixed-domain problem in which the density vanishes inside the fluid
region while both transport coefficients obey possibly different
temperature power laws.

To the best of our knowledge, no global strong-solution result was previously
available for the one-dimensional fixed-boundary problem with genuine
interior vacuum and simultaneous pure temperature laws for viscosity and
heat conductivity.  The contribution of the present paper is to close this
specific gap for arbitrarily large compatible data in a high-temperature
regime.  Here ``arbitrarily large'' has its standard quantified meaning: an
arbitrary data bound is fixed first, and the required temperature threshold
is then allowed to depend on that bound.  Notice that the
temperature remains uniformly positive in our construction; the degeneracy
comes from the vanishing density in front of the material derivatives, not
from a vanishing viscosity along the solution.

The large-time behavior has two complementary aspects.  Let
$X_t(a,t)=u(X(a,t),t)$ and $X(a,0)=a$.  The continuity equation implies
\begin{equation}\label{eq:vacuum-transport-intro}
 \rho(X(a,t),t)X_a(a,t)=\rho_0(a).
\end{equation}
Hence every initial vacuum point persists for all finite times.  Because
$H^1(0,1)$ embeds continuously into $C([0,1])$, convergence of $\rho$ to its
positive conserved mean in $H^1$ would contradict
\eqref{eq:vacuum-transport-intro}.  A vacuum-compatible theorem therefore
cannot contain uniform dissipation of $\norm{\rho_x}_{L^2}^2$ over
$(0,\infty)$, nor exponential decay of the unweighted density gradient.
This does not prevent relaxation below the density-gradient level.  We prove
that $\rho$ converges exponentially to its mean in every finite $L^p$, while
$u$ and $\theta-\bt$ decay in $H^1$ and the weighted material derivatives
decay in $L^2$.  If a vacuum point is present, the same estimate forces
$\norm{\rho_x}_{L^2}$ to grow exponentially.  Thus the solution equilibrates in
low-order norms while an increasingly sharp vacuum layer preserves the exact
zero.  Correspondingly, high Sobolev norms of the density are controlled on
arbitrary finite time intervals rather than uniformly for all time.

The principal analytical difficulties are threefold.  First, division by
$\rho$ is unavailable, so the momentum and temperature equations do not
directly control unweighted time derivatives.  Second, differentiation of
$\theta^\alpha$ and $\theta^\beta$ produces nonlinear commutators whose
powers must be balanced uniformly in the large reference temperature.
Third, the density equation is hyperbolic and has no intrinsic smoothing;
its higher derivatives must be propagated through the velocity Lipschitz
norm without introducing a uniform-in-time density-gradient dissipation.

Our argument resolves these issues as follows.\nopagebreak[4]
\begin{enumerate}[label=\textup{(\roman*)},leftmargin=2.2em]
\item We use relative entropy centered at the actual conserved mean density,
      rather than at a prescribed positive equilibrium.  Temperature
      coercivity is recovered through the Dirichlet boundary condition, so
      no density lower bound enters the basic estimate.
\item The physical effective viscous flux
      $F=\mu\theta^\alpha u_x-R\rho\theta$ satisfies
      $F_x=\rho\dot u$.  A weighted spatial anchor and a Zlotnik estimate
      along trajectories provide a density upper bound uniform in the
      strictly positive approximation parameter.
\item Material-derivative estimates are organized with powers of $\bt$.
      To accommodate the natural initial scales
      $\norm{\sqrt\rho\dot u(0)}_2=O(\bt^\alpha)$ and
      $\norm{\sqrt\rho\dot\theta(0)}_2=O(\bt^\beta)$, we introduce the
      shifted weight
      $\sigma_{\bt}(t)=\min\{1,t+\bt^{1-\beta}\}$.  Multiplying the
      differentiated momentum identity by $\bt$ before adding the thermal
      identity cancels their apparently singular leading coupling.  Under
      $\alpha>2$ and $\beta>2\alpha+4$, all remaining commutators are
      absorbed once $\bt$ is sufficiently large.  This range is
      intentionally conservative and is not claimed to be optimal.
\item We solve positive-density approximations, derive estimates independent
      of their lower bounds, and pass to the vacuum limit by compactness.
      Density derivatives are propagated by finite-time transport estimates,
      while the exact flow formula records persistence of the vacuum set.
\item The effective-flux anchor also yields an instantaneous, vacuum-compatible
      coercivity estimate for the density oscillation.  Adding a suitably
      scaled relative entropy to the differentiated energy produces a strict
      Lyapunov functional and an exponential relaxation rate of order
      $\bt^{1-\alpha}$.  For genuine vacuum, this low-order decay implies
      exponential growth of the density gradient and therefore identifies the
      sharp endpoint of the stability statement.
\end{enumerate}

No smallness is imposed on the density, velocity, or temperature
perturbation.  Their $H^2$ bound may be arbitrarily large.  The compatibility
residuals are bounded only after division by their natural principal-part
scales $\bt^\alpha$ and $\bt^\beta$; these normalized bounds are not
smallness assumptions.  The main theorem is an existence result; uniqueness
in the vacuum class is not claimed.

The remainder of the paper is organized as follows.  Section~\ref{sec:main}
states the strong-solution class and the global existence and relaxation
theorems.
Section~\ref{sec:entropy} develops the relative entropy and the
vacuum-compatible pointwise anchors.  The
uniform first-order and material-derivative estimates are proved in
Section~\ref{sec:uniform}, and the finite-time higher-order estimates in
Section~\ref{sec:higher}.  Section~\ref{sec:construction} constructs the
solution by positive-density approximation and continuation.  Finally,
Section~\ref{sec:persistence} establishes exponential relaxation below the
density-gradient level, proves persistence of vacuum, and quantifies the
resulting sharpening of the vacuum layer.

\section{Main theorem}\label{sec:main}

Throughout the paper, all Lebesgue and Sobolev norms are taken over $(0,1)$
unless another domain is displayed explicitly; in particular,
$\norm{\cdot}_{L^p}=\norm{\cdot}_{L^p(0,1)}$ and
$H^k=H^k(0,1)$.  Constants denoted by $C$ may change from line to line.  They
never depend on a positive lower bound for the density.  Constants depending
on a finite final time are denoted by $C_T$.

\begin{definition}\label{def:strong}
A triple $(\rho,u,\theta)$ is a strong solution on $[0,T]$ if
\begin{align*}
 &\rho\in C([0,T];H^2),\qquad
   \rho_t\in C([0,T];H^1),\\
 &u,\theta-\bt\in C([0,T];H^2\cap H_0^1)
    \cap L^2(0,T;H^3),\\
 &\sqrt\rho\,\dot u,\sqrt\rho\,\dot\theta
      \in L^\infty(0,T;L^2),
   \qquad \dot u,\dot\theta\in L^2(0,T;H_0^1),
\end{align*}
with $\rho\geq0$ and $\theta>0$.  The equations \eqref{eq:NSF} hold almost
everywhere in $(0,1)\times(0,T)$, whereas the boundary and initial conditions
\eqref{eq:boundary}--\eqref{eq:initial} are attained in the trace sense.
A global strong solution is a strong solution on $[0,T]$ for every finite
$T>0$.
\end{definition}

\begin{theorem}[Global strong existence allowing initial vacuum]\label{thm:main}
Let $R,\mu,\kappa>0$ and $K_0\geq1$.  Fix mass bounds
$0<m_*<m^*$, and assume
\begin{equation}\label{eq:exponents}
 \alpha>2,\qquad \beta>2\alpha+4.
\end{equation}
There exists
\[
 \Theta_0=\Theta_0(R,\mu,\kappa,\alpha,\beta,K_0,m_*,m^*)\geq1
\]
with the following property.  For every boundary temperature
$\bt\geq\Theta_0$, prescribe the initial trace
\[
 (\rho,u,\theta)(x,0)=(\rho_0,u_0,\theta_0)(x),
 \qquad x\in(0,1).
\]
The complete assumptions on the initial triple
$(\rho_0,u_0,\theta_0)$ and its compatibility data are the following; no
additional initial-data hypothesis is imposed elsewhere.
\begin{enumerate}[label=\textup{(H\arabic*)},leftmargin=3.2em,itemsep=4pt]
\item \emph{Initial density and total mass.}  The initial density obeys
\begin{equation}\label{eq:data-density}
 \rho_0\in H^2,\qquad \rho_0\geq0,\qquad
 m_*\leq m:=\int_0^1\rho_0\dd x\leq m^*,\qquad
 \norm{\rho_0}_{H^2}\leq K_0.
\end{equation}
No positive lower bound is imposed on $\rho_0$; beyond the compatibility
condition \textup{(H3)}, no smallness or measure restriction is placed on its
zero set.

\item \emph{Velocity and temperature data.}  The remaining data obey
\begin{equation}\label{eq:data-dynamic}
 u_0\in H^2\cap H_0^1,\qquad
 \theta_0-\bt\in H^2\cap H_0^1,\qquad
 \norm{u_0}_{H^2}^2+\norm{\theta_0-\bt}_{H^2}^2\leq K_0,
 \qquad
 \frac34\bt\leq\theta_0\leq\frac54\bt.
\end{equation}
Thus, in particular,
\[
 u_0(0)=u_0(1)=0,\qquad
 \theta_0(0)=\theta_0(1)=\bt.
\]

\item \emph{Compatibility conditions.}  There exist
$g_1,g_2\in L^2(0,1)$, taken to vanish almost everywhere on
$\{\rho_0=0\}$, such that the relations
\begin{equation}\label{eq:compatibility}
\left\{
\begin{aligned}
 &(\mu\theta_0^\alpha u_{0x})_x-(R\rho_0\theta_0)_x
       =\sqrt{\rho_0}\,g_1,\\
 &(\kappa\theta_0^\beta\theta_{0x})_x
       +\mu\theta_0^\alpha u_{0x}^2
       -R\rho_0\theta_0u_{0x}
       =\sqrt{\rho_0}\,g_2
\end{aligned}\right.
\end{equation}
hold almost everywhere.  In particular, the second relation includes the
pressure-work term $-R\rho_0\theta_0u_{0x}$.

\item \emph{Natural compatibility scale.}  The compatibility data satisfy
\begin{equation}\label{eq:natural-data}
 \bt^{-2\alpha}\norm{g_1}_{L^2}^2
 +\bt^{-2\beta}\norm{g_2}_{L^2}^2\leq K_0.
\end{equation}
The number $K_0$ is arbitrary and is not assumed small.  The powers in
\eqref{eq:natural-data} merely normalize the principal parts of
\eqref{eq:compatibility}: for fixed $u_0$ and $\theta_0-\bt$, their natural
sizes are $g_1=O(\bt^\alpha)$ and $g_2=O(\bt^\beta)$.
\end{enumerate}
Then problem \eqref{eq:NSF}--\eqref{eq:initial} admits at least one global
strong solution in the sense of Definition~\ref{def:strong}.
\smallskip

\noindent\emph{Uniform pointwise bounds.}
There are constants $M$ and $C$ depending only on
$R,\mu,\kappa,\alpha,\beta$, $K_0$, $m_*$, and $m^*$.  They are independent
of time, $\bt$, and any positive lower bound for $\rho_0$, and satisfy
\begin{equation}\label{eq:pointwise-main}
 0\leq\rho(x,t)\leq M,\qquad
 \frac12\bt\leq\theta(x,t)\leq\frac32\bt
 \quad (x\in[0,1],\ t\geq0),
\end{equation}
\smallskip

Put
\begin{equation}\label{eq:sigma-main}
 \tau_{\bt}=\bt^{1-\beta},\qquad
 \sigma_{\bt}(t)=\min\{1,t+\tau_{\bt}\}.
\end{equation}

\noindent\emph{Initial-layer and post-layer energy estimates.}
\nopagebreak[4]
\begin{align}\label{eq:uniform-main}
 &\operatorname*{ess\,sup}_{0\leq t\leq1}\Bigl[
  \sigma_{\bt}\bt^{\alpha-1}\norm{u_x(t)}_{L^2}^2
 +\sigma_{\bt}\bt^{\beta-2}\norm{\theta_x(t)}_{L^2}^2
 +\sigma_{\bt}^2\bt^{-1}
       \norm{\sqrt{\rho}\,\dot u(t)}_{L^2}^2
 +\sigma_{\bt}^2\bt^{-2}
       \norm{\sqrt{\rho}\,\dot\theta(t)}_{L^2}^2
 \Bigr]\nonumber\\
 &\quad+\int_0^1\Bigl[
  \bt^{\alpha-1}\norm{u_x}_{L^2}^2
 +\bt^{\beta-2}\norm{\theta_x}_{L^2}^2
 +\sigma_{\bt}\bt^{-1}
       \norm{\sqrt{\rho}\,\dot u}_{L^2}^2
 +\sigma_{\bt}\bt^{-2}
       \norm{\sqrt{\rho}\,\dot\theta}_{L^2}^2\nonumber\\[-1pt]
 &\quad
 +\sigma_{\bt}^2\bt^{\alpha-1}\norm{\dot u_x}_{L^2}^2
 +\sigma_{\bt}^2\bt^{\beta-2}\norm{\dot\theta_x}_{L^2}^2
 \Bigr]\dd t\leq C.
\end{align}
At the endpoint of the initial layer, $\sigma_{\bt}(1)=1$; the post-layer
closure then gives
\begin{equation}\label{eq:post-layer-A}
 \operatorname*{ess\,sup}_{t\geq1}\Acal(t)
 +\int_1^\infty\Dcal(t)\dd t\leq C,
\end{equation}
where $\Acal$ and $\Dcal$ are defined in
\eqref{eq:A}--\eqref{eq:D} below.
\smallskip

\noindent\emph{Finite-time higher-order estimates.}
For each $T<\infty$ there is a constant $C_T$, depending only on
$T$, $\bt$, $R,\mu,\kappa,\alpha,\beta$, $K_0$, $m_*$, and $m^*$.
It may grow with $T$, but is independent of any positive lower bound for
$\rho_0$, and satisfies
\begin{align}\label{eq:finite-high}
 &\sup_{0\leq t\leq T}
   \bigl(\norm{\rho(t)}_{H^2}+\norm{u(t)}_{H^2}
        +\norm{\theta(t)-\bt}_{H^2}\bigr)+\int_0^T\bigl(\norm{u}_{H^3}^2+\norm{\theta-\bt}_{H^3}^2
       +\norm{\dot u}_{H^1}^2+\norm{\dot\theta}_{H^1}^2\bigr)\dd t
 \leq C_T.
\end{align}
\smallskip

\noindent\emph{Persistence of vacuum.}
For every $a\in[0,1]$, let $X(a,t)$ be the particle path defined by
\[
 X_t(a,t)=u(X(a,t),t),\qquad X(a,0)=a.
\]
If $\rho_0(a)=0$, then $\rho(X(a,t),t)=0$ for every finite $t$.
\end{theorem}

\begin{theorem}[Exponential relaxation and vacuum-layer sharpening]
\label{thm:relaxation}
Let $(\rho,u,\theta)$ be the global strong solution furnished by
Theorem~\ref{thm:main}, and let
$m=\int_0^1\rho_0\dd x$.  There exist constants $c_{\rm d},C_{\rm d}>0$,
depending only on $R,\mu,\kappa,\alpha,\beta,K_0,m_*$, and $m^*$, but
independent of $\bt$ and of any positive lower bound for the density, such
that, with
\begin{equation}\label{eq:decay-rate}
 \lambda_{\rm d}=c_{\rm d}\bt^{1-\alpha},
\end{equation}
one has, for almost every $t\geq1$,
\begin{align}
 &\norm{\rho(t)-m}_{L^2}^2
 +\bt^{\alpha-2}\norm{u(t)}_{H^1}^2
 +\bt^{\beta-2}\norm{\theta(t)-\bt}_{H^1}^2
  +\bt^{-2}\left(
   \norm{\sqrt\rho\,\dot u(t)}_{L^2}^2
  +\norm{\sqrt\rho\,\dot\theta(t)}_{L^2}^2\right)
 \leq C_{\rm d}e^{-\lambda_{\rm d}(t-1)}.
 \label{eq:low-order-decay}
\end{align}
The first three terms in \eqref{eq:low-order-decay} satisfy the same estimate
for every $t\geq1$.  Moreover, for every $t\geq1$,
\begin{align}
 &\int_t^\infty\Bigl[
  \bt^{\alpha-1}\norm{u_x}_{L^2}^2
 +\bt^{\beta-2}\norm{\theta_x}_{L^2}^2
 +\bt^{-2}\Bigl(
   \norm{\sqrt{\rho}\,\dot u}_{L^2}^2
  +\norm{\sqrt{\rho}\,\dot\theta}_{L^2}^2\Bigr)
 \nonumber\\
 &\qquad +\bt^{\alpha-2}\norm{\dot u_x}_{L^2}^2
 +\bt^{\beta-2}\norm{\dot\theta_x}_{L^2}^2
 \Bigr]\dd s
 \leq C_{\rm d}e^{-\lambda_{\rm d}(t-1)}.
 \label{eq:decay-tail}
\end{align}
Consequently, for every $1\leq p<\infty$ there are constants $C_p,c_p>0$
with
\begin{equation}\label{eq:density-Lp-decay}
 \norm{\rho(t)-m}_{L^p}
 \leq C_p\exp\!\left(-c_p\bt^{1-\alpha}(t-1)\right)
 \qquad(t\geq1),\qquad
 c_p=\frac{c_{\rm d}}{\max\{2,p\}}.
\end{equation}
If, in addition, $\rho_0(a_0)=0$ for some $a_0\in[0,1]$, set
$c_{\rm v}:=m^2/(2C_{\rm d}^{1/2})$.  Then, for every $t\geq1$,
\begin{equation}\label{eq:gradient-growth}
 \norm{\rho(t)-m}_{L^\infty}\geq m,
 \qquad
 \norm{\rho_x(t)}_{L^2}
 \geq\max\left\{\sqrt{3}\,m,
                 c_{\rm v}e^{\lambda_{\rm d}(t-1)/2}\right\},
\end{equation}
and, for every $T>1$,
\begin{equation}\label{eq:gradient-growth-integrated}
 \int_1^T\norm{\rho_x(t)}_{L^2}^2\dd t
 \geq\max\left\{3m^2(T-1),
       \frac{c_{\rm v}^2}{\lambda_{\rm d}}
       \left(e^{\lambda_{\rm d}(T-1)}-1\right)\right\}.
\end{equation}
In particular,
$\norm{\rho(t)-m}_{H^1}\geq
c_{\rm v}e^{\lambda_{\rm d}(t-1)/2}$ for $t\geq1$.
Thus, in the genuine-vacuum case, density relaxation cannot hold in
$L^\infty$, $H^1$, or any stronger topology that controls point values.
\end{theorem}

\begin{remark}[The vacuum class is nonempty]\label{rem:example}
For example, fix $m\in[m_*,m^*]$ and set
\[
 \rho_0(x)=12m\left(x-\frac12\right)^2,\qquad
 u_0=0,\qquad \theta_0=\bt.
\]
Then $\rho_0(1/2)=0$ and \eqref{eq:compatibility} holds with $g_2=0$ and
\[
 g_1(x)=-4R\bt\sqrt{3m}\,
 \operatorname{sgn}\!\left(x-\frac12\right),
 \qquad \operatorname{sgn}(0):=0.
\]
In particular, $g_1\in L^2$ and
$\bt^{-2\alpha}\norm{g_1}_{L^2}^2$ is uniformly bounded for $\bt\geq1$.
Consequently, provided
that $K_0$ is chosen sufficiently large depending only on $R$ and $m$, these
data satisfy all the hypotheses of Theorem~\ref{thm:main}.
\end{remark}

\begin{remark}[Range and sharpness of the results]\label{rem:scope}
Theorem~\ref{thm:main} permits arbitrarily large compatible data: $K_0$ is
fixed at an arbitrary value before the threshold $\Theta_0(K_0)$ is chosen.
It does not assert a single temperature threshold uniform over data of
unbounded size.  The normalization \eqref{eq:natural-data} is dictated by the
principal differential operators and is not a perturbative smallness
condition.  Substantially enlarging the exponent range would require a
different commutator closure.
Theorem~\ref{thm:relaxation} gives exponential relaxation precisely below
the density-gradient level.  If $\rho_0$ vanishes somewhere, the density
cannot converge to its positive mean in $L^\infty$ or in any topology that
controls its point values or spatial derivative.  The approximation argument
proves existence; uniqueness in the vacuum class is not addressed here.
\end{remark}

\Needspace{10\baselineskip}
\begin{table}[!htbp]
\centering
\caption{Comparison with selected one-dimensional global and large-time
results for the full compressible Navier--Stokes system.  The list
is representative rather than exhaustive.  For works in Lagrangian
coordinates, ``no vacuum'' means that the specific volume stays finite and
the density remains strictly positive.}
\label{tab:comparison}
\footnotesize
\setlength{\tabcolsep}{3.1pt}
\renewcommand{\arraystretch}{1.17}
\begin{tabularx}{\textwidth}
  {@{}L{2.05cm}L{2.55cm}L{2.45cm}L{1.05cm}Y@{}}
\toprule
Reference
& Setting and data regime
& Transport coefficients
& Vacuum
& Main conclusion and distinction \\
\midrule
\multicolumn{5}{@{}l}{\itshape
Global positive-density results with temperature-dependent coefficients}
\\[-1pt]
\addlinespace

Liu et al.~\cite{LiuYangZhaoZou2014}
& Cauchy problem; Nishida--Smoller large-data regime
& General smooth positive functions
  $\mu(\theta)$ and $\kappa(\theta)$
& No
& Unique global smooth solution when $(\gamma-1)$ times the relevant
  $H^3$ norm is small. \\

\addlinespace
Wang--Zhao~\cite{WangZhao2016}
& Cauchy problem; large $H^3$ data and $|\alpha|$ small
& $\mu=\widetilde\mu h(v)\theta^\alpha$,
  $\kappa=\widetilde\kappa h(v)\theta^\alpha$
& No
& Unique global smooth solution and uniform convergence to equilibrium;
  no exponential rate is asserted. \\

\addlinespace
Huang--Shi~\cite{HuangShi2020}
& Bounded-interval problem; large compatible data
& $\mu=\mathrm{const.}$,
  $\kappa(\theta)\sim\theta^\beta$, $\beta>0$
& No
& Global strong solution with nonlinear exponential stability in $H^1$. \\

\addlinespace
Sun--Zhang--Zhao~\cite{SunZhangZhao2021}
& Bounded-interval problem; $|\alpha|$ small
& $\mu(\theta)\sim\theta^\alpha$,
  $\kappa(\theta)\sim\theta^\beta$
& No
& Global strong solution and nonlinear exponential stability in $H^1$;
  the heat-conductivity exponent may be arbitrary. \\

\addlinespace
Dong--Tan~\cite{DongTan2025}
& Bounded-interval problem; $H^2\times H^1\times H^1$ data
& $\mu(\theta)\sim\theta^\alpha$,
  $\kappa(\theta)\sim\theta^\beta$
& No
& Global strong well-posedness and nonlinear exponential stability in
  $H^2$ for a sufficiently small viscosity exponent. \\

\addlinespace
Dong--Guo~\cite{DongGuo2025}
& Cauchy problem; large $H^2\times H^1\times H^1$ data and $|a|$ small 
& $\mu(\theta)=\widetilde\mu\theta^a$,
  $\kappa(\theta)=\widetilde\kappa\theta^b$, $a,b\geq0$
& No
& Global strong solution and asymptotic stability; the two pure temperature
  powers may be different, and no uniform exponential rate is claimed. \\

\midrule
\multicolumn{5}{@{}l}{\itshape
Fixed-domain results allowing initial vacuum}
\\[-1pt]
\addlinespace

Wen--Zhu~\cite{WenZhu2013}
& Bounded-domain problem; large data
& Constant viscosity; temperature-dependent heat conductivity
& Yes
& Unique global classical solution, but the viscosity is independent of
  temperature. \\

\addlinespace
Li~\cite{Li2019}
& Bounded-domain problem; large $H^2$ data
& Constant viscosity and heat conductivity
& Yes
& Unique global strong solution with nonnegative density; neither
  transport coefficient depends on temperature. \\

\addlinespace
\textbf{Present work}
& Bounded-domain problem; large $H^2$ data at high temperature
& $\mu(\theta)=\mu\theta^\alpha$,
  $\kappa(\theta)=\kappa\theta^\beta$,
  $\alpha>2$, $\beta>2\alpha+4$
& \textbf{Yes}
& Global strong existence allowing genuine interior vacuum.  Density decays in
  every finite $L^p$, while $u$ and $\theta-\bt$ decay in $H^1$; a transported
  vacuum point forces exponential growth of $\norm{\rho_x}_{L^2}$. \\
\bottomrule
\end{tabularx}
\end{table}
\FloatBarrier

\begin{remark}[Analytical difficulties and proof mechanisms]
\label{rem:proof-mechanisms}
Table~\ref{tab:comparison} focuses on the one-dimensional global theory.
The multidimensional local-vacuum results and the vacuum free-boundary
results discussed in the introduction are not repeated there, because their
geometry and solution classes are different.  The comparison is intended to
record structural distinctions, not to claim optimality of the exponent
range.  Within the precise regime of
Theorems~\ref{thm:main}--\ref{thm:relaxation}, the main analytical
difficulties and the mechanisms used here to resolve them are as follows.
\begin{enumerate}[label=\textup{(\roman*)},leftmargin=2.2em,itemsep=5pt]
\item \emph{Coercivity compatible with vacuum.}
The density may vanish, so neither the momentum equation nor the temperature
equation can be divided by $\rho$.  The relative entropy
\eqref{eq:relative-energy}, centered at the conserved mean $m$, remains
finite at vacuum.  Its thermal part is density weighted and therefore does
not control the temperature on the vacuum set; the missing unweighted
control is recovered from the fixed boundary temperature and the Dirichlet
Poincar\'e inequality.  Hence the basic energy estimate requires no positive
lower density bound.

\item \emph{Pointwise bounds without division by the density.}
For the physical effective viscous flux
\[
 F=\mu\theta^\alpha u_x-R\rho\theta,
 \qquad F_x=\rho\dot u,
\]
the quantity $\Pi$ in \eqref{eq:Pi} fixes the spatial constant through a
$\theta^{-\alpha}$-weighted mean.  This yields the anchor
\eqref{eq:F-anchor} in terms of $\sqrt\rho\,\dot u$, without invoking
$\rho^{-1}$.  A Zlotnik argument for $\log\rho$ along positive-density
particle paths then gives the density upper bound uniformly in the
approximation parameter.  The heat-flux identity \eqref{eq:heat-flux}
supplies the derivative control needed in the thermal commutators, while the
temperature bootstrap is improved using the scaled $L^2$ estimate of
$\theta_x$ and the fixed boundary trace.

\item \emph{Initial-layer smoothing and post-layer closure.}
The time scale $\tau_{\bt}=\bt^{1-\beta}$ and the shifted weight
$\sigma_{\bt}(t)=\min\{1,t+\tau_{\bt}\}$ capture the rapid thermal initial
layer without requiring $u_0$ or $\theta_0-\bt$ to shrink as
$\bt\to\infty$.  Multiplying the material momentum identity by $\bt$ before
adding the material temperature identity cancels the leading mixed term at
equilibrium.  The resulting weighted estimate places the solution at the
normalized post-layer scale at $t=1$.  From that time onward, the hierarchy
\eqref{eq:A}--\eqref{eq:D} and Lemma~\ref{lem:bookkeeping} close the
variable-coefficient commutators under
$\alpha>2$ and $\beta>2\alpha+4$.  This is a sufficient exponent range, not
a claim of sharpness.

\item \emph{Passage from positive density to genuine vacuum.}
The local problem is solved for $\rho_0^\delta=\rho_0+\delta$.  The modified
compatibility residuals satisfy the uniform naturally scaled bounds
\eqref{eq:gdelta}--\eqref{eq:natural-data-delta}.  The two-stage estimates are
independent of $\inf\rho_0^\delta$, while the density derivatives are
propagated only on finite time intervals by the transport estimates of
Section~\ref{sec:higher}.  Compactness then produces a strong vacuum-limit
solution, and the exact flow identity \eqref{eq:flow-formula} transports the
initial zero set.  No smallness or measure condition is imposed on that set,
although uniqueness of the vacuum-limit solution is not asserted.

\item \emph{Relaxation below the density-gradient level.}
The basic entropy contains no direct density dissipation.
Lemma~\ref{lem:decay-coercivity} recovers the missing $L^2$ control of
$\rho-m$ from the effective flux without using $\rho_x$.  Adding a scaled
relative entropy to the differentiated energy gives the strict Lyapunov
inequality in Proposition~\ref{prop:strict-lyapunov} and the rate
$\lambda_{\rm d}=c_{\rm d}\bt^{1-\alpha}$.  This proves finite-$L^p$
density relaxation, $H^1$ relaxation of $u$ and $\theta-\bt$, and weighted
$L^2$ decay of the material derivatives.  When a vacuum point is present,
the transported zero rules out density convergence in $L^\infty$ and $H^1$;
together with the $L^2$ decay it yields the quantitative vacuum-layer
sharpening in \eqref{eq:gradient-growth}--
\eqref{eq:gradient-growth-integrated}.
\end{enumerate}
Thus the contribution is the conjunction of global vacuum-compatible
large-data existence, an initial-layer estimate uniform in the density lower
bound, and low-order exponential relaxation after that layer.  The exponent
range is not claimed optimal, and uniqueness in the vacuum class remains
open.
\end{remark}

\section{Relative entropy and vacuum-compatible pointwise estimates}
\label{sec:entropy}

Henceforth, the parameters and initial data are fixed as in
Theorem~\ref{thm:main}, and $m$ denotes the conserved total mass defined in
\eqref{eq:data-density}.

Set
\begin{equation}\label{eq:relative-functions}
 h(\rho\mid m)=\rho\log\frac{\rho}{m}-\rho+m,
 \qquad
 \Psi(\theta\mid\bt)=\theta-\bt-\bt\log\frac{\theta}{\bt},
\end{equation}
with $0\log0=0$, and define
\begin{equation}\label{eq:relative-energy}
 \Ecal(t)=\int_0^1\left(
 \frac12\rho u^2+R\bt h(\rho\mid m)+\rho\Psi(\theta\mid\bt)
 \right)\dd x.
\end{equation}

\begin{lemma}[Relative entropy identity]\label{lem:entropy}
Every positive smooth solution of \eqref{eq:NSF}--\eqref{eq:boundary}
satisfies
\begin{align}\label{eq:entropy-identity}
 \frac{\dd}{\dd t}\Ecal(t)
 &+\int_0^1\left(
 \mu\bt\theta^{\alpha-1}u_x^2
 +\kappa\bt\theta^{\beta-2}\theta_x^2
 \right)\dd x=0.
\end{align}
The same identity, with equality replaced by the corresponding integrated
inequality, holds for the vacuum limit constructed below.
\end{lemma}

\begin{proof}
The continuity equation implies
\[
 \frac{\dd}{\dd t}\int_0^1\frac12\rho u^2\dd x
 =\int_0^1\rho\dot u\,u\dd x.
\]
Testing the momentum equation by $u$, using $u=0$ at the endpoints, and
integrating the temperature equation give, respectively,
\begin{align*}
& \frac{\dd}{\dd t}\int_0^1\frac12\rho u^2\dd x
 +\int_0^1\mu\theta^\alpha u_x^2\dd x
 =\int_0^1R\rho\theta u_x\dd x,\\
 &\frac{\dd}{\dd t}\int_0^1\rho\theta\dd x
 +\int_0^1R\rho\theta u_x\dd x
 =\int_0^1\mu\theta^\alpha u_x^2\dd x
   +\left[\kappa\theta^\beta\theta_x\right]_{0}^{1}.
\end{align*}
Their sum is the total-energy balance
\begin{equation}\label{eq:total-energy-balance}
 \frac{\dd}{\dd t}\int_0^1
 \left(\frac12\rho u^2+\rho\theta\right)\dd x
 =\left[\kappa\theta^\beta\theta_x\right]_{0}^{1}.
\end{equation}
On the other hand, division of the temperature equation by $\theta$ and
use of $\Dt\rho=-\rho u_x$ yield the local entropy balance
\begin{equation}\label{eq:local-entropy-balance}
 (\rho\log\theta-R\rho\log\rho)_t
 +(\rho u\log\theta-R\rho u\log\rho)_x
 =\mu\theta^{\alpha-1}u_x^2
  +\kappa\theta^{\beta-2}\theta_x^2
  +(\kappa\theta^{\beta-1}\theta_x)_x.
\end{equation}
Integrate \eqref{eq:local-entropy-balance}, multiply it by $\bt$, and
subtract it from \eqref{eq:total-energy-balance}.  The boundary fluxes
cancel exactly because $\theta=\bt$ on $\partial\I$:
\[
 \kappa\theta^\beta\theta_x
 -\bt\kappa\theta^{\beta-1}\theta_x=0.
\]
The functional obtained by this subtraction is
\begin{equation*}
 \widetilde{\Ecal}(t)=\int_0^1\left[
 \frac12\rho u^2+\rho\theta-\bt\rho\log\theta
 +R\bt\rho\log\rho\right]\dd x.
\end{equation*}
Expanding the two relative functions in \eqref{eq:relative-energy} gives
\begin{align*}
 \Ecal(t)-\widetilde{\Ecal}(t)
 =\int_0^1\bigl[&-R\bt\rho\log m-R\bt\rho+R\bt m
 -\bt\rho+\bt\rho\log\bt\bigr]\dd x.
\end{align*}
Since $\int_0^1\rho\dd x=m$ and $|\I|=1$, the right-hand side equals
\[
 -R\bt m\log m-\bt m+\bt m\log\bt,
\]
which is constant in time.  Therefore $\Ecal'=\widetilde{\Ecal}'$, and the
subtracted balance is exactly \eqref{eq:entropy-identity}.  This calculation
also shows explicitly why centering at the conserved mean $m$ changes only
the normalization and places the minimum at the actual equilibrium density.

For the vacuum limit, integrate the positive-density identity between
$s$ and $t$.  The strong convergence in $C([0,T];H^1)$ established in
Section~\ref{sec:construction} passes every term of $\Ecal^\delta(s)$ and
$\Ecal^\delta(t)$ to the limit.  The temperature bounds make the diffusion
coefficients converge uniformly, while weak lower semicontinuity gives
\begin{align*}
 \int_s^t\!\int_0^1\left(
 \mu\bt\theta^{\alpha-1}u_x^2
 +\kappa\bt\theta^{\beta-2}\theta_x^2\right)\dd x\dd r
\leq\liminf_{\delta\to0}
 \int_s^t\!\int_0^1\left(
 \mu\bt(\theta^\delta)^{\alpha-1}|u_x^\delta|^2
 +\kappa\bt(\theta^\delta)^{\beta-2}|\theta_x^\delta|^2\right)\dd x\dd r.
\end{align*}
Thus the integrated identity passes to the asserted integrated inequality
without using a positive lower bound for the density.
\end{proof}

\subsection{Coercivity and pointwise anchors}

Under a bootstrap bound $0\leq\rho\leq M$ and
$\bt/4\leq\theta\leq2\bt$, convexity gives, for some
$c_{\rm ent}=c_{\rm ent}(M,m_*,m^*,R)>0$,
\begin{equation}\label{eq:entropy-coercivity}
 c_{\rm ent}\left[
 \bt\norm{\rho-m}_{L^2}^2+\norm{\sqrt\rho u}_{L^2}^2\right]
 \leq\Ecal(t)
\end{equation}
and
$\Ecal(0)\leq C_{\rm in}\bt$, where
$C_{\rm in}=C_{\rm in}(K_0,m_*,m^*,R)$.  These constants are independent
of $t$, $\bt$, and a positive lower bound for $\rho$.  We do not place an
unweighted thermal $L^2$ term on the left: vacuum entropy cannot control
temperature where $\rho=0$.  Such a term, when needed, follows instead from
$\theta-\bt\in H_0^1$ and $\norm{\theta-\bt}_{L^2}\leq C\norm{\theta_x}_{L^2}$.

The effective flux used below is the physical one,
\begin{equation}\label{eq:flux}
 F=\mu\theta^\alpha u_x-R\rho\theta,
 \qquad F_x=\rho\dot u.
\end{equation}
Define its time-dependent anchor
\begin{equation}\label{eq:Pi}
 \Pi(t)=R\frac{\displaystyle\int_0^1\rho\theta^{1-\alpha}\dd x}
                 {\displaystyle\int_0^1\theta^{-\alpha}\dd x}.
\end{equation}

\begin{lemma}[Flux anchor]\label{lem:flux-anchor}
Suppose $0\leq\rho\leq M$ and
$\bt/4\leq\theta\leq2\bt$.  Then
\begin{align}
 \norm{F+\Pi}_{L^\infty}
 \leq m^{1/2}\norm{\sqrt\rho\dot u}_{L^2},
 \label{eq:F-anchor}\\
 \norm{u_x}_{L^\infty}
 \leq C_{\rm F}\bt^{1-\alpha}
   +C_{\rm F}\bt^{-\alpha}\norm{\sqrt\rho\dot u}_{L^2}.
 \label{eq:ux-anchor}
\end{align}
Here $C_{\rm F}=C_{\rm F}(M,R,\mu,\alpha)>0$ can be chosen
nondecreasing in $M$ and is independent of $t$, $\bt$, and
$\inf\rho$.  In particular, \eqref{eq:F-anchor} itself has no dependence
on $M$ and neither estimate uses a lower bound for $\rho$.
\end{lemma}

\begin{proof}
Since $\int_0^1u_x\dd x=0$, definitions
\eqref{eq:flux}--\eqref{eq:Pi} give
\[
 \int_0^1\theta^{-\alpha}(F+\Pi)\dd x=0.
\]
Hence $F(\xi,t)+\Pi(t)=0$ at some $\xi\in[0,1]$.  Therefore
\[
 |F(x,t)+\Pi(t)|\leq\norm{F_x}_{L^1}
 \leq m^{1/2}\norm{\sqrt\rho\dot u}_{L^2}.
\]
The temperature comparison also yields
\[
 \int_0^1\rho\theta^{1-\alpha}\dd x
 \leq C_\alpha m\bt^{1-\alpha},\qquad
 \int_0^1\theta^{-\alpha}\dd x
 \geq c_\alpha\bt^{-\alpha}.
\]
Consequently,
\[
 |\Pi|\leq C_{\alpha,R}m\bt.
\]
Solving \eqref{eq:flux} for $u_x$ and adding and subtracting $\Pi$ give
\[
 u_x=\frac{\theta^{-\alpha}}{\mu}
       \bigl((F+\Pi)-\Pi+R\rho\theta\bigr).
\]
Since $\theta^{-\alpha}\leq4^\alpha\bt^{-\alpha}$,
$\rho\theta\leq2M\bt$, and $m=\int_0^1\rho\dd x\leq M$, we obtain
\begin{align*}
 \norm{u_x}_{L^\infty}
 &\leq\frac{4^\alpha}{\mu}\bt^{-\alpha}
 \left(\norm{F+\Pi}_{L^\infty}+|\Pi|
       +R\norm{\rho\theta}_{L^\infty}\right)\\
 &\leq C_{\alpha,\mu}\bt^{-\alpha}
 \left(m^{1/2}\norm{\sqrt\rho\dot u}_{L^2}
       +(C_{\alpha,R}m+2RM)\bt\right)\\
 &\leq C_{\rm F}\bt^{1-\alpha}
       +C_{\rm F}\bt^{-\alpha}
          \norm{\sqrt\rho\dot u}_{L^2},
\end{align*}
which is \eqref{eq:ux-anchor} with the stated dependence of $C_{\rm F}$.
\end{proof}

For the temperature put
\begin{equation}\label{eq:heat-flux}
 G=\kappa\theta^\beta\theta_x,\qquad
 G_x=\rho\dot\theta+R\rho\theta u_x
             -\mu\theta^\alpha u_x^2.
\end{equation}

\begin{lemma}[Thermal anchor]\label{lem:thermal-anchor}
Under the hypotheses of Lemma~\ref{lem:flux-anchor},
\begin{align}
 \norm{\theta_x}_{L^\infty}
 &\leq C_{\rm G}\bt^{-\beta}\left(
  \norm{\sqrt\rho\dot\theta}_{L^2}+\bt\norm{u_x}_{L^2}
  +\bt^\alpha\norm{u_x}_{L^2}^2\right),
 \label{eq:thetax-infty}\\
 \norm{\theta_{xx}}_{L^2}
 &\leq C_{\rm G}\bt^{-\beta}\left(
  \norm{\sqrt\rho\dot\theta}_{L^2}+\bt\norm{u_x}_{L^2}
  +\bt^\alpha\norm{u_x}_{L^\infty}\norm{u_x}_{L^2}\right)
  +C_{\rm G}\bt^{-1}\norm{\theta_x}_{L^\infty}
                         \norm{\theta_x}_{L^2}.
 \label{eq:thetaxx}
\end{align}
Here
$C_{\rm G}=C_{\rm G}(M,R,\mu,\kappa,\alpha,\beta)>0$ can be chosen
nondecreasing in $M$ and is independent of $t$, $\bt$, and $\inf\rho$.
\end{lemma}

\begin{proof}
Because $\theta(0,t)=\theta(1,t)$,
$\int_0^1G\theta^{-\beta}\dd x=\kappa\int_0^1\theta_x\dd x=0$.
Thus $G$ has a zero and $\norm{G}_{L^\infty}\leq\norm{G_x}_{L^1}$.
To estimate the latter norm, equation \eqref{eq:heat-flux},
Cauchy--Schwarz, mass conservation, the density upper bound, and
$|\I|=1$ give
\begin{align*}
 \norm{G_x}_{L^1}
 &\leq\norm{\rho\dot\theta}_{L^1}
      +R\norm{\rho\theta u_x}_{L^1}
      +\mu\norm{\theta^\alpha u_x^2}_{L^1}\\
 &\leq m^{1/2}\norm{\sqrt\rho\dot\theta}_{L^2}
      +2RM\bt\norm{u_x}_{L^1}
      +2^\alpha\mu\bt^\alpha\norm{u_x}_{L^2}^2\\
 &\leq m^{1/2}\norm{\sqrt\rho\dot\theta}_{L^2}
      +2RM\bt\norm{u_x}_{L^2}
      +2^\alpha\mu\bt^\alpha\norm{u_x}_{L^2}^2.
\end{align*}
Because $G=\kappa\theta^\beta\theta_x$ and
$\theta^{-\beta}\leq4^\beta\bt^{-\beta}$,
\[
 \norm{\theta_x}_{L^\infty}
 \leq\frac{4^\beta}{\kappa}\bt^{-\beta}
       \norm{G}_{L^\infty}
 \leq\frac{4^\beta}{\kappa}\bt^{-\beta}
       \norm{G_x}_{L^1}.
\]
Using $m\leq M$ proves \eqref{eq:thetax-infty} with the stated constant.

For the second derivative, the three terms in
\eqref{eq:heat-flux} satisfy
\begin{align*}
& \norm{\rho\dot\theta}_{L^2}^2
 =\int_0^1\rho^2|\dot\theta|^2\dd x
 \leq M\norm{\sqrt\rho\dot\theta}_{L^2}^2,\\
 &R\norm{\rho\theta u_x}_{L^2}
 \leq2RM\bt\norm{u_x}_{L^2},\\
& \mu\norm{\theta^\alpha u_x^2}_{L^2}
 \leq2^\alpha\mu\bt^\alpha
       \norm{u_x^2}_{L^2}
 \leq2^\alpha\mu\bt^\alpha
       \norm{u_x}_{L^\infty}\norm{u_x}_{L^2}.
\end{align*}
Consequently,
\[
 \norm{G_x}_{L^2}
 \leq M^{1/2}\norm{\sqrt\rho\dot\theta}_{L^2}
 +2RM\bt\norm{u_x}_{L^2}
 +2^\alpha\mu\bt^\alpha
       \norm{u_x}_{L^\infty}\norm{u_x}_{L^2}.
\]
Finally, expanding $G_x$ gives the pointwise identity
\[
 \theta_{xx}=\frac{\theta^{-\beta}}{\kappa}G_x
              -\beta\theta^{-1}\theta_x^2.
\]
The bootstrap lower bound for $\theta$ therefore implies
\begin{align*}
 \norm{\theta_{xx}}_{L^2}
 &\leq\frac{4^\beta}{\kappa}\bt^{-\beta}
          \norm{G_x}_{L^2}
      +4|\beta|\bt^{-1}\norm{\theta_x^2}_{L^2}\leq\frac{4^\beta}{\kappa}\bt^{-\beta}
          \norm{G_x}_{L^2}
      +4|\beta|\bt^{-1}\norm{\theta_x}_{L^\infty}
                              \norm{\theta_x}_{L^2}.
\end{align*}
Substitution of the preceding $L^2$ bound for $G_x$ proves
\eqref{eq:thetaxx} and displays explicitly every dependence entering
$C_{\rm G}$.
\end{proof}

We shall also use the following elementary form of Zlotnik's lemma.

\begin{lemma}[Zlotnik]\label{lem:zlotnik}
Let $y$ and $b$ be absolutely continuous and satisfy
$y'\leq g(y)+b'$ on $[0,T]$, where $g(s)\to-\infty$ as $s\to\infty$.
If
\[
 b(t_2)-b(t_1)\leq N_0+N_1(t_2-t_1)
 \quad(0\leq t_1<t_2\leq T)
\]
and $g(s)\leq-N_1$ for $s\geq\zeta$, then
$y(t)\leq\max\{y(0),\zeta\}+N_0$.
\end{lemma}

\begin{proof}
Set $M_0=\max\{y(0),\zeta\}$.  Suppose, to the contrary, that there is a
time $t_2\in(0,T]$ such that $y(t_2)>M_0+N_0$.  By continuity of the
absolutely continuous function $y$, there is a last time
\[
 t_1:=\max\{s\in[0,t_2]:y(s)=M_0\}.
\]
This set is nonempty: if $M_0=y(0)$ it contains $0$, whereas if
$M_0=\zeta>y(0)$ the intermediate-value theorem supplies a crossing before
$t_2$.  The choice of $t_1$ implies
$y(s)>M_0\geq\zeta$ for $t_1<s\leq t_2$, and therefore
$g(y(s))\leq-N_1$ almost everywhere on this interval.  Integrating the
differential inequality gives
\begin{align*}
 y(t_2)-M_0
 =y(t_2)-y(t_1)
 \leq\int_{t_1}^{t_2}g(y(s))\dd s+b(t_2)-b(t_1)
 &\leq-N_1(t_2-t_1)+N_0+N_1(t_2-t_1)=N_0,
\end{align*}
contradicting $y(t_2)>M_0+N_0$.  Hence
$y(t)\leq M_0+N_0$ on $[0,T]$.
\end{proof}

\section{Uniform estimates independent of the density lower bound}
\label{sec:uniform}

This section establishes the global first-order and material-derivative
estimates without introducing an unweighted density-gradient bootstrap.
All calculations are first made for a smooth solution with positive density.
The constants are uniform as that lower bound tends to zero.

The normalized energy and dissipation scales are
\begin{align}
 \Acal(t)
 &=\bt^{\alpha-2}\norm{u_x}_{L^2}^2+\bt^{\beta-2}\norm{\theta_x}_{L^2}^2
   +\bt^{-2}(\norm{\sqrt\rho\dot u}_{L^2}^2+\norm{\sqrt\rho\dot\theta}_{L^2}^2),\label{eq:A}\\
 \Dcal(t)
 &=\bt^{\alpha-1}\norm{u_x}_{L^2}^2+\bt^{\beta-2}\norm{\theta_x}_{L^2}^2
   +\bt^{-2}(\norm{\sqrt\rho\dot u}_{L^2}^2+\norm{\sqrt\rho\dot\theta}_{L^2}^2)
   +\bt^{\alpha-2}\norm{\dot u_x}_{L^2}^2+\bt^{\beta-2}\norm{\dot\theta_x}_{L^2}^2.
\label{eq:D}
\end{align}

\subsection{Exact differentiated identities}

\begin{lemma}[First material identities]\label{lem:first-identities}
Every smooth positive solution satisfies
\begin{align}
& \frac\mu2\frac{\dd}{\dd t}\int_0^1\theta^\alpha u_x^2\dd x
 -R\frac{\dd}{\dd t}\int_0^1\rho\theta u_x\dd x
 +\norm{\sqrt\rho\dot u}_{L^2}^2 \nonumber\\
 \quad
 &=-R\int_0^1\rho\dot\theta u_x\dd x
  +R\int_0^1\rho\theta u_x^2\dd x
  -\frac\mu2\int_0^1\theta^\alpha u_x^3\dd x
  +\frac{\mu\alpha}{2}\int_0^1\theta^{\alpha-1}\dot\theta u_x^2\dd x.
\label{eq:first-u}
\end{align}
\begin{equation}
 \frac\kappa2\frac{\dd}{\dd t}\int_0^1\theta^\beta\theta_x^2\dd x
 +\norm{\sqrt\rho\dot\theta}_{L^2}^2
 =\int_0^1(\mu\theta^\alpha u_x^2-R\rho\theta u_x)\dot\theta\dd x
  -\frac\kappa2\int_0^1\theta^\beta u_x\theta_x^2\dd x
  +\frac{\kappa\beta}{2}\int_0^1\theta^{\beta-1}\dot\theta\theta_x^2\dd x .
\label{eq:first-theta}
\end{equation}
\end{lemma}

\begin{proof}
Multiply $\rho\dot u+(R\rho\theta)_x=(\mu\theta^\alpha u_x)_x$
by $\dot u$ and integrate.
The two integration-by-parts formulas
\begin{align*}
 -\mu\int_0^1\theta^\alpha u_x\dot u_x\dd x
 &=-\frac\mu2\frac{\dd}{\dd t}\int_0^1\theta^\alpha u_x^2\dd x
   +\frac{\mu\alpha}{2}\int_0^1\theta^{\alpha-1}\dot\theta u_x^2\dd x
   -\frac\mu2\int_0^1\theta^\alpha u_x^3\dd x,\\
 R\int_0^1\rho\theta\dot u_x\dd x
 &=R\frac{\dd}{\dd t}\int_0^1\rho\theta u_x\dd x
   -R\int_0^1(\rho\dot\theta-\rho\theta u_x)u_x\dd x
\end{align*}
are valid because $u=0$ at the boundary.  Since
$\Dt(R\rho\theta)=R\rho\dot\theta-R\rho\theta u_x$,
\eqref{eq:first-u} follows.
Multiplication of the third equation in \eqref{eq:NSF} by $\dot\theta$
and the identity
\[
 -\kappa\int_0^1\theta^\beta\theta_x\dot\theta_x\dd x
 =-\frac\kappa2\frac{\dd}{\dd t}\int_0^1\theta^\beta\theta_x^2\dd x
   +\frac{\kappa\beta}{2}\int_0^1\theta^{\beta-1}\dot\theta\theta_x^2\dd x
   -\frac\kappa2\int_0^1\theta^\beta u_x\theta_x^2\dd x
\]
give \eqref{eq:first-theta}.  In particular, the coefficient of
$\frac{\dd}{\dd t}\int_0^1\theta^\beta\theta_x^2\dd x$ is $\kappa/2$, not
$\kappa\beta/2$.
\end{proof}

The material derivatives vanish at the boundary:
\begin{equation}\label{eq:dot-boundary}
 \dot u|_{\partial\I}=0,\qquad \dot\theta|_{\partial\I}=0.
\end{equation}
Consequently,
\begin{equation}\label{eq:dot-poincare}
 \norm{\dot u}_{L^\infty}+\norm{\dot u}_{L^2}\leq C \norm{\dot u_x}_{L^2},\qquad
 \norm{\dot\theta}_{L^\infty}+\norm{\dot\theta}_{L^2}\leq C \norm{\dot\theta_x}_{L^2},
\end{equation}
and, using $\rho\leq M$,
\begin{equation}\label{eq:weighted-by-derivative}
 \norm{\sqrt\rho\dot u}_{L^2}^2\leq CM \norm{\dot u_x}_{L^2}^2,\qquad \norm{\sqrt\rho\dot\theta}_{L^2}^2\leq CM \norm{\dot\theta_x}_{L^2}^2.
\end{equation}

\begin{lemma}[Material-derivative inequalities]\label{lem:material}
Assume $0\leq\rho\leq M$ and $\bt/4\leq\theta\leq2\bt$.  Then
\begin{align}
 \frac12\frac{\dd}{\dd t}\norm{\sqrt\rho\dot u}_{L^2}^2+c\bt^\alpha \norm{\dot u_x}_{L^2}^2
 &\leq C\bt^{-\alpha}\norm{\sqrt\rho\dot\theta}_{L^2}^2+C\bt^{2-\alpha}\norm{u_x}_{L^2}^2
   +C\bt^\alpha\norm{u_x}_{L^\infty}^2\norm{u_x}_{L^2}^2
   +C\bt^{\alpha-2}\norm{u_x}_{L^2}^2\norm{\dot\theta_x}_{L^2}^2,
\label{eq:material-u}\\
 \frac12\frac{\dd}{\dd t}\norm{\sqrt\rho\dot\theta}_{L^2}^2+c\bt^\beta \norm{\dot\theta_x}_{L^2}^2
 &\leq C\norm{u_x}_{L^\infty} \norm{\sqrt\rho\dot\theta}_{L^2}^2
 +C\bt \norm{\sqrt\rho\dot\theta}_{L^2}\bigl(\norm{\dot u_x}_{L^2}+\norm{u_x}_{L^\infty} \norm{u_x}_{L^2}\bigr)
 \nonumber\\
 &\quad+C\bt^\alpha \norm{\dot\theta_x}_{L^2}\norm{u_x}_{L^\infty} \norm{u_x}_{L^2}^2
 +C\bt^{\alpha-1}\norm{\dot\theta_x}_{L^2}^2\norm{u_x}_{L^2}^2
 \nonumber\\
 &\quad
 +C\bt^\alpha \norm{\dot\theta_x}_{L^2}\norm{u_x}_{L^2}\norm{\dot u_x}_{L^2}
 +C\bt^{\beta-1}\norm{\theta_x}_{L^2}\norm{\dot\theta_x}_{L^2}^2
 \nonumber\\
 &\quad+C\bt^\beta\norm{u_x}_{L^\infty} \norm{\theta_x}_{L^2} \norm{\dot\theta_x}_{L^2} .
\label{eq:material-theta}
\end{align}
In \eqref{eq:material-u}--\eqref{eq:material-theta},
$c=c(\mu,\kappa,\alpha,\beta)>0$ denotes a coercivity constant, whereas
$C=C(M,R,\mu,\kappa,\alpha,\beta)>0$ denotes a generic upper-bound
constant that may change from line to line.  Both are independent of
$t$, $\bt$, and $\inf\rho$.
\end{lemma}

\begin{proof}
For a scalar function $f$, define in this proof the conservative material
operator
\[
 \mathcal C f=f_t+(uf)_x=\Dt f+u_xf.
\]
The continuity equation and direct differentiation give
\begin{equation}\label{eq:material-commutators}
 \mathcal C(\rho f)=\rho\Dt f,\qquad
 \mathcal C(q_x)=(\Dt q)_x,
 \qquad \Dt u_x=\dot u_x-u_x^2,
 \qquad \Dt\theta_x=\dot\theta_x-u_x\theta_x .
\end{equation}
The constitutive laws also give
\[
 \Dt(R\rho\theta)=R\rho\dot\theta-R\rho\theta u_x,\qquad
 \Dt(\mu\theta^\alpha)=\mu\alpha\theta^{\alpha-1}\dot\theta,\qquad
 \Dt(\kappa\theta^\beta)=\kappa\beta\theta^{\beta-1}\dot\theta.
\]
Applying $\mathcal C$ to the momentum equation and using
\eqref{eq:material-commutators}, we obtain the exact equation
\begin{equation}\label{eq:exact-material-momentum}
 \rho\Dt\dot u+\bigl(R\rho\dot\theta-R\rho\theta u_x\bigr)_x
 =\bigl(\mu\alpha\theta^{\alpha-1}\dot\theta u_x
       +\mu\theta^\alpha\dot u_x-\mu\theta^\alpha u_x^2\bigr)_x .
\end{equation}
Since $\dot u=0$ on $\partial\I$, testing by $\dot u$ yields
\begin{align}
 &\frac12\frac{\dd}{\dd t}\norm{\sqrt\rho\dot u}_{L^2}^2
 +\mu\int_0^1\theta^\alpha|\dot u_x|^2\dd x\nonumber\\
 &
=R\int_0^1\rho\dot\theta\dot u_x\dd x
  -R\int_0^1\rho\theta u_x\dot u_x\dd x  -\mu\alpha\int_0^1\theta^{\alpha-1}\dot\theta u_x\dot u_x\dd x
  +\mu\int_0^1\theta^\alpha u_x^2\dot u_x\dd x .
 \label{eq:exact-material-u-energy}
\end{align}
Here we used
$\frac{\dd}{\dd t}\int_0^1\rho f^2\dd x
=2\int_0^1\rho f\Dt f\dd x$, which follows
from the continuity equation and $u|_{\partial\I}=0$.  For every
$\varepsilon>0$, the four terms on the right satisfy
\begin{align*}
& R\left|\int_0^1\rho\dot\theta\dot u_x\dd x\right|
 \leq\varepsilon\bt^\alpha \norm{\dot u_x}_{L^2}^2
      +C_\varepsilon\bt^{-\alpha}\norm{\sqrt\rho\dot\theta}_{L^2}^2,\\
 &R\left|\int_0^1\rho\theta u_x\dot u_x\dd x\right|
 \leq\varepsilon\bt^\alpha \norm{\dot u_x}_{L^2}^2
      +C_\varepsilon\bt^{2-\alpha}\norm{u_x}_{L^2}^2,\\
& \mu\alpha\left|\int_0^1\theta^{\alpha-1}\dot\theta u_x\dot u_x\dd x\right|
 \leq C\bt^{\alpha-1}\norm{\dot\theta_x}_{L^2}\norm{u_x}_{L^2}\norm{\dot u_x}_{L^2}
  \leq\varepsilon\bt^\alpha \norm{\dot u_x}_{L^2}^2
      +C_\varepsilon\bt^{\alpha-2}\norm{u_x}_{L^2}^2\norm{\dot\theta_x}_{L^2}^2,\\
& \mu\left|\int_0^1\theta^\alpha u_x^2\dot u_x\dd x\right|
 \leq C\bt^\alpha\norm{u_x}_{L^\infty} \norm{u_x}_{L^2}\norm{\dot u_x}_{L^2}
  \leq\varepsilon\bt^\alpha \norm{\dot u_x}_{L^2}^2
      +C_\varepsilon\bt^\alpha\norm{u_x}_{L^\infty}^2\norm{u_x}_{L^2}^2.
\end{align*}
Choosing $\varepsilon$ sufficiently small proves
\eqref{eq:material-u}.

We next apply $\mathcal C$ to the temperature equation.  Keeping every
commutator term gives
\begin{align}
 \rho\Dt\dot\theta-(\kappa\theta^\beta\dot\theta_x)_x
 &=-R\rho\dot\theta u_x-R\rho\theta\dot u_x+R\rho\theta u_x^2
   -\mu\theta^\alpha u_x^3\nonumber\\
 &\quad+\mu\alpha\theta^{\alpha-1}\dot\theta u_x^2
   +2\mu\theta^\alpha u_x\dot u_x
   +\partial_x\bigl(\kappa\beta\theta^{\beta-1}\dot\theta\theta_x
                     -\kappa\theta^\beta u_x\theta_x\bigr).
 \label{eq:exact-material-temperature}
\end{align}
Testing by $\dot\theta$ and using
$\dot\theta|_{\partial\I}=0$ gives
\begin{align}
 &\frac12\frac{\dd}{\dd t}\norm{\sqrt\rho\dot\theta}_{L^2}^2
 +\kappa\int_0^1\theta^\beta|\dot\theta_x|^2\dd x\nonumber\\
 &=-R\int_0^1\rho u_x|\dot\theta|^2\dd x
 -R\int_0^1\rho\theta\dot u_x\dot\theta\dd x +R\int_0^1\rho\theta u_x^2\dot\theta\dd x
 -\mu\int_0^1\theta^\alpha u_x^3\dot\theta\dd x+\mu\alpha\int_0^1\theta^{\alpha-1}|\dot\theta|^2u_x^2\dd x\nonumber\\
 &\quad
 +2\mu\int_0^1\theta^\alpha u_x\dot u_x\dot\theta\dd x -\kappa\int_0^1\bigl(\beta\theta^{\beta-1}\dot\theta\theta_x
                        -\theta^\beta u_x\theta_x\bigr)
       \dot\theta_x\dd x\nonumber\\
 &\leq C\norm{u_x}_{L^\infty}\norm{\sqrt\rho\dot\theta}_{L^2}^2
 +C\bt \norm{\sqrt\rho\dot\theta}_{L^2}\norm{\dot u_x}_{L^2}
 +C\bt \norm{\sqrt\rho\dot\theta}_{L^2}\norm{u_x}_{L^\infty}\norm{u_x}_{L^2}+C\bt^{\alpha-1}\norm{\dot\theta_x}_{L^2}^2\norm{u_x}_{L^2}^2\nonumber\\
 &\quad 
 +C\bt^\alpha\norm{\dot\theta_x}_{L^2}\norm{u_x}_{L^2}\norm{\dot u_x}_{L^2}
 +C\bt^{\beta-1}\norm{\theta_x}_{L^2}\norm{\dot\theta_x}_{L^2}^2
 +C\bt^\beta\norm{u_x}_{L^\infty}\norm{\theta_x}_{L^2}\norm{\dot\theta_x}_{L^2}.
 \label{eq:exact-material-temperature-energy}
\end{align}
The inequality in \eqref{eq:exact-material-temperature-energy} follows
from \eqref{eq:dot-poincare}: the seven groups on the right-hand side of
the equality are bounded, in their displayed order, by the seven
corresponding groups following the inequality sign.  These are precisely
the terms on the right of \eqref{eq:material-theta}.  In particular, only
the boundary Poincar\'e
estimate for $\dot u$ and $\dot\theta$ is used; at no point is an unweighted
material derivative estimated by a density-weighted norm.
\end{proof}

\subsection{A shifted initial layer for large data}

The natural compatibility scales in Theorem~\ref{thm:main} are much larger
than the post-layer scales in \eqref{eq:A}.  This mismatch cannot be handled
by inserting the natural initial size into the later bootstrap constant.  We
instead use a shifted time weight.  Throughout this subsection put
\begin{equation}\label{eq:layer-weight}
 \tau=\bt^{1-\beta},\qquad
 \sigma(t)=\min\{1,t+\tau\}.
\end{equation}
Thus $\sigma(0)>0$, which is essential both for the initial compatibility
traces and for the effective-flux estimate at the time endpoint.
The function $\sigma$ is Lipschitz.  All differentiated identities below
are understood almost everywhere on the two smooth pieces of $\sigma$;
equivalently, one may use smooth approximations of $\sigma$ and pass to the
limit.

\begin{proposition}[Vacuum-uniform initial-layer estimate]
\label{prop:initial-layer}
There exist $M_{\rm in},K_{\rm in},\Theta_{\rm in}\geq1$, depending only on
the physical parameters, $K_0,m_*$, and $m^*$, with the following property.
Let a smooth positive-density solution with data satisfying
\textup{(H1)}--\textup{(H4)} of Theorem~\ref{thm:main} exist on $[0,S]$,
where $0<S\leq1$, and suppose, as bootstrap assumptions, that
\begin{equation}\label{eq:layer-bootstrap}
 \rho\leq2M_{\rm in},\qquad
 \frac14\bt\leq\theta\leq2\bt,
 \qquad \mathcal B_\sigma(t)\leq2K_{\rm in},
\end{equation}
where
\begin{align}\label{eq:layer-B}
 \mathcal B_\sigma={}&
 \sigma\bt^{\alpha-1}\norm{u_x}_{L^2}^2+\sigma\bt^{\beta-2}\norm{\theta_x}_{L^2}^2
 +\sigma^2\bt^{-1}\norm{\sqrt\rho\dot u}_{L^2}^2+\sigma^2\bt^{-2}\norm{\sqrt\rho\dot\theta}_{L^2}^2.
\end{align}
If $\bt\geq\Theta_{\rm in}$, then on $[0,S]$
\begin{align}\label{eq:layer-improvement}
 \rho\leq M_{\rm in},\qquad
   \frac12\bt\leq\theta\leq\frac32\bt,
   \qquad \sup_{0\leq t\leq S}\mathcal B_\sigma(t)
       \leq K_{\rm in},     
       \end{align}
       \begin{align} 
       \label{eq:layer-dissipation}
        \int_0^S\left[
   \sigma\bt^{-1}\norm{\sqrt\rho\dot u}_{L^2}^2+\sigma\bt^{-2}\norm{\sqrt\rho\dot\theta}_{L^2}^2
  +\sigma^2\bt^{\alpha-1}\norm{\dot u_x}_{L^2}^2
  +\sigma^2\bt^{\beta-2}\norm{\dot\theta_x}_{L^2}^2\right]\dd t\leq K_{\rm in}.
 \end{align}
All constants are independent of the positive lower bound of the density.
At $S=1$, \eqref{eq:layer-improvement} implies
\begin{equation}\label{eq:layer-entry}
 \Acal(1)\leq C K_{\rm in}.
\end{equation}
\end{proposition}

\begin{proof}
Let $C_S$ be the norm of the embedding $H^2(0,1)\hookrightarrow L^\infty(0,1)$
and let $C_Z$ be the fixed threshold constant used in the Zlotnik estimate
for the logarithmic density inequality derived below.  Before estimating the
energy, fix
\begin{equation}\label{eq:Min-choice}
 M_{\rm in}>2e\max\{C_SK_0+1,C_Zm^*,1\}.
\end{equation}
All constants below are now computed under the density bootstrap
$\rho\leq2M_{\rm in}$.  The energy constant $K_{\rm in}$ is fixed next,
and the temperature threshold is chosen only after both constants have been
fixed.

The relative-entropy identity and the pointwise bootstrap give
\begin{equation}\label{eq:layer-basic}
 \sup_{0\leq t\leq S}\Ecal(t)\leq C\bt,
 \qquad
 \int_0^S\left(\bt^{\alpha-1}\norm{u_x}_{L^2}^2
                   +\bt^{\beta-2}\norm{\theta_x}_{L^2}^2\right)\dd t\leq C.
\end{equation}
Set
\begin{align*}
 \mathcal H_u&=\frac\mu2\int_0^1\theta^\alpha u_x^2\dd x
   -R\int_0^1(\rho\theta-m\bt)u_x\dd x,\qquad
 \mathcal H_\theta=\frac\kappa2\int_0^1\theta^\beta\theta_x^2\dd x.
\end{align*}
For parameters $0<\eta<1<L$ to be fixed below, define
\begin{equation}\label{eq:layer-functional}
 \mathcal W(t)=\frac{\sigma}{\bt}\mathcal H_u
 +\frac{\sigma}{\bt^2}\mathcal H_\theta
 +\frac{\eta\sigma^2}{2\bt^2}(\bt \norm{\sqrt\rho\dot u}_{L^2}^2+\norm{\sqrt\rho\dot\theta}_{L^2}^2)
 +L\bt^{-\alpha}\Ecal(t).
\end{equation}
We first prove the pressure-oscillation estimate needed here, without using
any estimate from the subsequent post-layer argument.  Since
\[
 \rho\theta-m\bt=\theta(\rho-m)+m(\theta-\bt),
\]
the pointwise temperature bootstrap, \eqref{eq:entropy-coercivity}, and the
Dirichlet Poincar\'e inequality imply
\begin{align*}
 R\norm{\rho\theta-m\bt}_{L^2}
 \leq R\norm{\theta}_{L^\infty}\norm{\rho-m}_{L^2}
       +Rm\norm{\theta-\bt}_{L^2}
 \leq C\bt\bigl(\bt^{-1}\Ecal\bigr)^{1/2}
       +C\norm{\theta_x}_{L^2}
 \leq C(\bt\Ecal)^{1/2}+C\norm{\theta_x}_{L^2}.
\end{align*}
Here and below in this calculation,
$C=C(M_{\rm in},m_*,m^*,R)$ is independent of $\bt$, $\sigma$, and any
positive lower bound for the density.  Consequently, Cauchy--Schwarz gives
\begin{align*}
 \frac{R\sigma}{\bt}
 \left|\int_0^1(\rho\theta-m\bt)u_x\dd x\right|
 &\leq C\sigma\bt^{-1}
 \left((\bt\Ecal)^{1/2}+\norm{\theta_x}_{L^2}\right)
 \norm{u_x}_{L^2}\\
 &=C\left(\sigma^{1/2}\bt^{(\alpha-1)/2}\norm{u_x}_{L^2}\right)
 \left(\sigma^{1/2}\bt^{-\alpha/2}\Ecal^{1/2}
       +\sigma^{1/2}\bt^{-(\alpha+1)/2}\norm{\theta_x}_{L^2}\right).
\end{align*}
Applying Young's inequality separately to the two products, for every
$\varepsilon>0$ we obtain
\begin{align}\label{eq:layer-cross}
 \frac{R\sigma}{\bt}\left|\int_0^1(\rho\theta-m\bt)u_x\dd x\right|
 &\leq\varepsilon\sigma\bt^{\alpha-1}\norm{u_x}_{L^2}^2
 +C_\varepsilon\sigma\bt^{-\alpha}\Ecal
 +C_\varepsilon\sigma\bt^{-\alpha-1}\norm{\theta_x}_{L^2}^2.
\end{align}
Choosing first $\varepsilon$ and then $L$ gives
\begin{equation}\label{eq:layer-equivalence}
 c\mathcal B_\sigma+c\bt^{-\alpha}\Ecal
 \leq\mathcal W
 \leq C\mathcal B_\sigma+C\bt^{-\alpha}\Ecal.
\end{equation}

The initial value of this functional is bounded without any smallness of
$u_0$ or $\theta_0-\bt$.  Indeed,
$\norm{\sqrt{\rho_0}\dot u(0)}_{L^2}=\norm{g_1}_2$ and
$\norm{\sqrt{\rho_0}\dot\theta(0)}_{L^2}=\norm{g_2}_2$, whereas
\eqref{eq:data-dynamic},
\eqref{eq:natural-data}, and $\tau=\bt^{1-\beta}$ give
\begin{align*}
 &\tau\bt^{\alpha-1}\norm{u_x(0)}_{L^2}^2\leq K_0\bt^{\alpha-\beta},
 \qquad
 \tau\bt^{\beta-2}\norm{\theta_x(0)}_{L^2}^2\leq K_0\bt^{-1},\\
 &\tau^2\bt^{-1}\norm{\sqrt{\rho_0}\dot u(0)}_{L^2}^2
   \leq K_0\bt^{2\alpha+1-2\beta},
 \qquad
 \tau^2\bt^{-2}\norm{\sqrt{\rho_0}\dot\theta(0)}_{L^2}^2\leq K_0.
\end{align*}
Since $\Ecal(0)\leq C\bt$, it follows that
\begin{equation}\label{eq:layer-initial-value}
 \mathcal W(0)\leq C(K_0).
\end{equation}

We next derive the differential inequality.  Put $K:=2K_{\rm in}$.  Under
\eqref{eq:layer-bootstrap}, \eqref{eq:layer-B} gives
\begin{equation}\label{eq:layer-sizes}
 \norm{u_x}_{L^2}^2\leq K\sigma^{-1}\bt^{1-\alpha},\quad
 \norm{\theta_x}_{L^2}^2\leq K\sigma^{-1}\bt^{2-\beta},\quad
 \norm{\sqrt\rho\dot u}_{L^2}^2\leq K\sigma^{-2}\bt,\quad
 \norm{\sqrt\rho\dot\theta}_{L^2}^2\leq K\sigma^{-2}\bt^2.
\end{equation}
In particular,
\begin{equation}\label{eq:layer-smallness}
 \sigma^{-1}\bt^{-\beta}\leq\bt^{-1},\qquad
 \frac{\norm{\theta_x}_{L^2}}{\bt}\leq C_K\bt^{-1/2},\qquad
 \bt^{(\alpha-\beta-1)/2}\norm{u_x}_{L^2}\leq C_K\bt^{-1/2}.
\end{equation}
The flux anchor also yields the time-weighted Lipschitz bound
\begin{equation}\label{eq:layer-ux}
 \sigma\norm{u_x}_{L^\infty}
 \leq C\sigma\bt^{1-\alpha}+C\sigma\bt^{-\alpha}\norm{\sqrt\rho\dot u}_{L^2}
 \leq C_K\bt^{1-\alpha}.
\end{equation}

Multiply \eqref{eq:first-u} by $\sigma\bt^{-1}$ and
\eqref{eq:first-theta} by $\sigma\bt^{-2}$.  Their nonlinear groups are
bounded as follows:
\begin{align}
 \sigma\bt^{-1}\norm{\sqrt\rho\dot\theta}_{L^2}\norm{u_x}_{L^2}
 &\leq\varepsilon\sigma\bt^{-2}\norm{\sqrt\rho\dot\theta}_{L^2}^2
       +C_\varepsilon\norm{u_x}_{L^2}^2,                 \label{eq:layer-I1}\\
 \sigma \norm{u_x}_{L^2}^2&\leq\norm{u_x}_{L^2}^2,                     \label{eq:layer-I2}\\
 \sigma\bt^{\alpha-1}\norm{u_x}_\infty \norm{u_x}_{L^2}^2
 &\leq C_K\norm{u_x}_{L^2}^2,                           \label{eq:layer-I3}\\
 \sigma\bt^{\alpha-2}\norm{\dot\theta_x}_{L^2}\norm{u_x}_{L^2}^2
 &\leq\varepsilon\sigma^2\bt^{\beta-2}\norm{\dot\theta_x}_{L^2}^2
       +C_{\varepsilon,K}\bt^{\alpha-2}\norm{u_x}_{L^2}^2,          \label{eq:layer-I4}\\
 \sigma\bt^{\beta-2}\norm{u_x}_\infty \norm{\theta_x}_{L^2}^2
 &\leq C_K\bt^{\beta-\alpha-1}\norm{\theta_x}_{L^2}^2,             \label{eq:layer-J3}\\
 \sigma\bt^{\beta-3}\norm{\dot\theta_x}_{L^2}\norm{\theta_x}_{L^2}^2
 &\leq\varepsilon\sigma^2\bt^{\beta-2}\norm{\dot\theta_x}_{L^2}^2
       +C_{\varepsilon,K}\bt^{\beta-3}\norm{\theta_x}_{L^2}^2.     \label{eq:layer-J4}
\end{align}
Here \eqref{eq:layer-I4}--\eqref{eq:layer-J4} use, respectively,
\[
 \bt^{2\alpha-\beta-2}\norm{u_x}_{L^2}^4
 \leq C_K\bt^{\alpha-2}\norm{u_x}_{L^2}^2,
 \qquad
 \bt^{\beta-4}\norm{\theta_x}_{L^2}^4
 \leq C_K\bt^{\beta-3}\norm{\theta_x}_{L^2}^2,
\]
which follow from \eqref{eq:layer-sizes} and the first inequality in
\eqref{eq:layer-smallness}.  For clarity, the two thermal terms having the
same algebraic form as \eqref{eq:layer-I1} and \eqref{eq:layer-I4} are,
explicitly,
\begin{align*}
 &\sigma\bt^{-2}
 \left|R\int_0^1\rho\theta u_x\dot\theta\dd x\right|
 \leq C\sigma\bt^{-1}\norm{u_x}_{L^2}
             \norm{\sqrt\rho\dot\theta}_{L^2}\leq\varepsilon\sigma\bt^{-2}
             \norm{\sqrt\rho\dot\theta}_{L^2}^2
          +C_\varepsilon\norm{u_x}_{L^2}^2,\\
 &\sigma\bt^{-2}
 \left|\mu\int_0^1\theta^\alpha u_x^2\dot\theta\dd x\right|
 \leq C\sigma\bt^{\alpha-2}\norm{\dot\theta_x}_{L^2}
                    \norm{u_x}_{L^2}^2\leq\varepsilon\sigma^2\bt^{\beta-2}
                    \norm{\dot\theta_x}_{L^2}^2
          +C_{\varepsilon,K}\bt^{\alpha-2}\norm{u_x}_{L^2}^2.
\end{align*}
Thus the six displayed estimates together with these two instances cover,
term by term, all four terms in each of \eqref{eq:first-u} and
\eqref{eq:first-theta}.

The decisive step is to use the exact material identities before applying
Young's inequality.  Multiply \eqref{eq:exact-material-u-energy} by $\bt$
and add \eqref{eq:exact-material-temperature-energy}.  The two leading
mixed terms satisfy the exact cancellation
\begin{align}\label{eq:layer-cancellation}
 \bt R\int_0^1\rho\dot\theta\dot u_x\dd x
 -R\int_0^1\rho\theta\dot u_x\dot\theta\dd x
 =R\int_0^1\rho(\bt-\theta)\dot\theta\dot u_x\dd x.
\end{align}
Consequently,
\begin{align}\label{eq:layer-material-exact}
 &\frac12(\bt \norm{\sqrt\rho\dot u}_{L^2}^2+\norm{\sqrt\rho\dot\theta}_{L^2}^2)'
 +\mu\bt\int_0^1\theta^\alpha|\dot u_x|^2\dd x
 +\kappa\int_0^1\theta^\beta|\dot\theta_x|^2\dd x=R\int_0^1\rho(\bt-\theta)\dot\theta\dot u_x\dd x
 +\mathcal R_m,
\end{align}
where
\begin{align*}
 \mathcal R_m={}&
 -R\bt\int_0^1\rho\theta u_x\dot u_x\dd x
 -\mu\alpha\bt\int_0^1\theta^{\alpha-1}\dot\theta u_x\dot u_x\dd x
 +\mu\bt\int_0^1\theta^\alpha u_x^2\dot u_x\dd x
 -R\int_0^1\rho u_x|\dot\theta|^2\dd x\\
 &+R\int_0^1\rho\theta u_x^2\dot\theta\dd x
 -\mu\int_0^1\theta^\alpha u_x^3\dot\theta\dd x
 +\mu\alpha\int_0^1\theta^{\alpha-1}|\dot\theta|^2u_x^2\dd x
 +2\mu\int_0^1\theta^\alpha u_x\dot u_x\dot\theta\dd x\\
 &-\kappa\beta\int_0^1\theta^{\beta-1}\dot\theta\theta_x\dot\theta_x\dd x
 +\kappa\int_0^1\theta^\beta u_x\theta_x\dot\theta_x\dd x .
\end{align*}
All undisplayed integrals in this formula are over $\I$.

Multiply \eqref{eq:layer-material-exact} by
$\eta\sigma^2\bt^{-2}$.  Since
$\norm{\theta-\bt}_\infty\leq \norm{\theta_x}_{L^2}\leq C_K\bt^{1/2}$,
the remaining mixed term satisfies
\begin{align}\label{eq:layer-main-mixed}
 &\eta\sigma^2\bt^{-2}
 \left|\int_0^1\rho(\bt-\theta)\dot\theta\dot u_x\dd x\right|\leq\varepsilon\eta\sigma^2\bt^{\alpha-1}\norm{\dot u_x}_{L^2}^2
 +C_{\varepsilon,K}\bt^{-\alpha}
       \sigma\bt^{-2}\norm{\sqrt\rho\dot\theta}_{L^2}^2.
\end{align}
The ten terms in $\mathcal R_m$, in their displayed order, obey
\begin{align}
 \eta\sigma^2\norm{u_x}_{L^2}\norm{\dot u_x}_{L^2}
 &\leq\varepsilon\eta\sigma^2\bt^{\alpha-1}\norm{\dot u_x}_{L^2}^2
 +C_\varepsilon\sigma^2\bt^{1-\alpha}\norm{u_x}_{L^2}^2,              \label{eq:layer-M1}\\
 \eta\sigma^2\bt^{\alpha-2}\norm{\dot\theta_x}_{L^2}\norm{u_x}_{L^2}\norm{\dot u_x}_{L^2}
 &\leq C_K\bt^{-1/2}\eta
  (\sigma^2\bt^{\alpha-1}\norm{\dot u_x}_{L^2}^2
   +\sigma^2\bt^{\beta-2}\norm{\dot\theta_x}_{L^2}^2),                       \label{eq:layer-M2}\\
 \eta\sigma^2\bt^{\alpha-1}\norm{u_x}_\infty \norm{u_x}_{L^2}\norm{\dot u_x}_{L^2}
 &\leq\varepsilon\eta\sigma^2\bt^{\alpha-1}\norm{\dot u_x}_{L^2}^2
 +C_\varepsilon\eta(\sigma\norm{u_x}_\infty)^2
       \bt^{\alpha-1}\norm{u_x}_{L^2}^2,              \label{eq:layer-M3}\\
 \eta\sigma^2\bt^{-2}\norm{u_x}_\infty \norm{\sqrt\rho\dot\theta}_{L^2}^2
 &\leq C_K\bt^{1-\alpha}\eta\sigma\bt^{-2}\norm{\sqrt\rho\dot\theta}_{L^2}^2,       \label{eq:layer-N1}\\
 \eta\sigma^2\bt^{-1}\norm{\sqrt\rho\dot\theta}_{L^2}\norm{u_x}_\infty \norm{u_x}_{L^2}
 &\leq\varepsilon\sigma\bt^{-2}\norm{\sqrt\rho\dot\theta}_{L^2}^2
 +C_{\varepsilon,K}(\sigma\norm{u_x}_\infty)^2
       \norm{u_x}_{L^2}^2,                              \label{eq:layer-N2}\\
 \eta\sigma^2\bt^{\alpha-2}\norm{\dot\theta_x}_{L^2}\norm{u_x}_\infty \norm{u_x}_{L^2}^2
 &\leq\varepsilon\eta\sigma^2\bt^{\beta-2}\norm{\dot\theta_x}_{L^2}^2
 +C_{\varepsilon,K}(\sigma\norm{u_x}_\infty)^2
       \bt^{\alpha-2}\norm{u_x}_{L^2}^2,              \label{eq:layer-N3}\\
 \eta\sigma^2\bt^{\alpha-3}\norm{\dot\theta_x}_{L^2}^2\norm{u_x}_{L^2}^2
 &\leq C_K\bt^{-1}\eta\sigma^2\bt^{\beta-2}\norm{\dot\theta_x}_{L^2}^2,  \label{eq:layer-N4}\\
 \eta\sigma^2\bt^{\alpha-2}\norm{\dot\theta_x}_{L^2}\norm{u_x}_{L^2}\norm{\dot u_x}_{L^2}
 &\leq C_K\bt^{-1/2}\eta
  (\sigma^2\bt^{\alpha-1}\norm{\dot u_x}_{L^2}^2
   +\sigma^2\bt^{\beta-2}\norm{\dot\theta_x}_{L^2}^2),                       \label{eq:layer-N5}\\
 \eta\sigma^2\bt^{\beta-3}\norm{\theta_x}_{L^2}\norm{\dot\theta_x}_{L^2}^2
 &\leq C_K\bt^{-1/2}\eta\sigma^2\bt^{\beta-2}\norm{\dot\theta_x}_{L^2}^2,\label{eq:layer-N6}\\
 \eta\sigma^2\bt^{\beta-2}\norm{u_x}_\infty \norm{\theta_x}_{L^2}\norm{\dot\theta_x}_{L^2}
 &\leq\varepsilon\eta\sigma^2\bt^{\beta-2}\norm{\dot\theta_x}_{L^2}^2
 +C_\varepsilon\eta(\sigma\norm{u_x}_\infty)^2
       \bt^{\beta-2}\norm{\theta_x}_{L^2}^2.          \label{eq:layer-N7}
\end{align}
Only \eqref{eq:layer-sizes}--\eqref{eq:layer-ux}, the boundary
Poincar\'e estimates in \eqref{eq:dot-poincare}, and $\rho\leq2M_{\rm in}$
are used here.  Thus no density lower bound is hidden in the table.

It remains to account for the derivative of the time weight.  On the set
where $\sigma'=1$,
\begin{align}
 \frac{\sigma'}{\bt}|\mathcal H_u|
 &\leq C\bigl(\bt^{\alpha-1}\norm{u_x}_{L^2}^2
                 +\bt^{\beta-2}\norm{\theta_x}_{L^2}^2\bigr)
        +C\bt^{-\alpha}\Ecal,
 \label{eq:layer-weight-u}\\
 \frac{\sigma'}{\bt^2}|\mathcal H_\theta|
 &\leq C\bt^{\beta-2}\norm{\theta_x}_{L^2}^2.       \label{eq:layer-weight-theta}
\end{align}
The derivative of the material weight is
\[
 \eta\sigma\sigma'(\bt^{-1}\norm{\sqrt\rho\dot u}_{L^2}^2+\bt^{-2}\norm{\sqrt\rho\dot\theta}_{L^2}^2),
\]
which is absorbed by the two first-identity dissipations after $\eta$ is
chosen sufficiently small.  The sources
\eqref{eq:layer-weight-u}--\eqref{eq:layer-weight-theta} are retained on the
right and are integrable by \eqref{eq:layer-basic}; no pointwise absorption
of them is used.  Combining
\eqref{eq:layer-I1}--\eqref{eq:layer-weight-theta}, choose first $\eta$,
then every Young parameter $\varepsilon$ smaller than a fixed multiple of
$\eta$, next $L$, and finally $\bt$.  This gives
{\small
\begin{equation}\label{eq:layer-differential}
 \mathcal W'\!+c\sigma\bigl[\bt^{-1}\norm{\sqrt\rho\dot u}_{L^2}^2+\bt^{-2}\norm{\sqrt\rho\dot\theta}_{L^2}^2
 +\eta\sigma\bigl(\bt^{\alpha-1}\norm{\dot u_x}_{L^2}^2+\bt^{\beta-2}\norm{\dot\theta_x}_{L^2}^2\bigr)\bigr]
 \leq C\bigl(\bt^{\alpha-1}\norm{u_x}_{L^2}^2+\bt^{\beta-2}\norm{\theta_x}_{L^2}^2+\bt^{-\alpha}\Ecal\bigr).
\end{equation}
}
Integration, \eqref{eq:layer-basic}, \eqref{eq:layer-equivalence}, and
\eqref{eq:layer-initial-value} prove the energy and dissipation parts of
\eqref{eq:layer-improvement}--\eqref{eq:layer-dissipation}, with a strict
improvement after $K_{\rm in}$ is fixed sufficiently large.

For the temperature, \eqref{eq:layer-B} and $\sigma\geq\tau$ give
\[
 \norm{\theta_x(t)}_{L^2}^2\leq C\sigma^{-1}\bt^{2-\beta}\leq C\bt.
\]
Since $\theta-\bt$ has zero boundary trace,
$\norm{\theta-\bt}_\infty\leq \norm{\theta_x}_{L^2}\leq C\bt^{1/2}$.  Increasing
$\Theta_{\rm in}$ improves the temperature bootstrap.

For the density, we first derive the logarithmic inequality needed for the
Zlotnik argument, rather than invoking its later post-layer formulation.
Along a positive-density particle path $X(t)$, put
$y(t)=\log\rho(X(t),t)$.  The continuity equation and
\eqref{eq:flux} give
\begin{equation*}
 y'=\frac{\dot\rho}{\rho}=-u_x=-\frac{F+R\rho\theta}{\mu\theta^\alpha}
 =-\frac{R}{\mu}\rho\theta^{1-\alpha}+\frac{1}{\mu}\theta^{-\alpha}\Pi
 -\frac{1}{\mu}\theta^{-\alpha}(F+\Pi).
\end{equation*}
We now estimate the three terms separately.  Since
$\bt/4\leq\theta\leq2\bt$ and $\alpha>2$,
\[
 -\frac{R}{\mu}\rho\theta^{1-\alpha}
 \leq-c_0\bt^{1-\alpha}\rho
 =-c_0\bt^{1-\alpha}e^y.
\]
Moreover, the definition \eqref{eq:Pi}, the same pointwise temperature
bounds, and $\int_0^1\rho\dd x=m$ imply
\begin{equation*}
 0\leq\Pi(t)=R\frac{\displaystyle\int_0^1\rho\theta^{1-\alpha}\dd x}
          {\displaystyle\int_0^1\theta^{-\alpha}\dd x}
 \leq C\frac{m\bt^{1-\alpha}}{\bt^{-\alpha}}\leq C m\bt.
\end{equation*}
and hence
\[
 \frac{1}{\mu}\theta^{-\alpha}\Pi(t)
 \leq C_0m\bt^{1-\alpha}.
\]
Finally, the flux-anchor estimate \eqref{eq:F-anchor} gives
\[
 \frac{1}{\mu}\theta^{-\alpha}|F+\Pi|
 \leq C_0m^{1/2}\bt^{-\alpha}
       \norm{\sqrt\rho\dot u(t)}_{L^2}.
\]
Combining these estimates proves, at this first point of use,
\begin{equation*}
 y'\leq-c_0\bt^{1-\alpha}e^y
 +C_0m\bt^{1-\alpha}
 +C_0m^{1/2}\bt^{-\alpha}
       \norm{\sqrt\rho\dot u(t)}_{L^2}.
\end{equation*}
The constants $c_0,C_0>0$ depend only on the displayed fixed parameters and
the pointwise temperature ratios; in particular, they are independent of
$\bt$ and of any positive lower bound for the density.

We next estimate the increment in the last term.  By
\eqref{eq:layer-B},
\begin{equation}\label{eq:layer-P-pointwise}
 \norm{\sqrt\rho\dot u(t)}_{L^2}\leq C_{K_{\rm in}}\bt^{1/2}\sigma(t)^{-1}.
\end{equation}
Hence, for $0\leq t_1<t_2\leq S$,
\begin{align}\label{eq:layer-P-increment}
 \bt^{-\alpha}\int_{t_1}^{t_2}\norm{\sqrt\rho\dot u(t)}_{L^2}\dd t
 &\leq C_{K_{\rm in}}\bt^{1/2-\alpha}
       \int_{t_1}^{t_2}\frac{\dd t}{\sigma(t)}\leq C_{K_{\rm in}}\bt^{1/2-\alpha}
       \bigl((\beta-1)\log\bt+1\bigr)=:N_0(\bt),
\end{align}
where we split the integral at $1-\tau$ if that point belongs to
$[t_1,t_2]$.  Thus $N_0(\bt)\to0$ because $\alpha>2$.  Combining the
logarithmic density inequality just proved with this increment estimate, in
Lemma~\ref{lem:zlotnik} we take
\[
 g_{\bt}(y)=-c_0\bt^{1-\alpha}e^y+C_0m\bt^{1-\alpha},
 \qquad N_1=0,
\]
and choose the fixed threshold $\zeta=\log(C_Zm)$.  After increasing
$\Theta_{\rm in}$ so that $N_0(\bt)\leq1$, the lemma gives
\[
 \rho(x,t)\leq e\max\{\norm{\rho_0}_\infty+1,C_Zm\},
\]
uniformly in the positive approximation.  By \eqref{eq:Min-choice}, this
strictly improves the density bootstrap and involves no circular choice.
Finally,
$\sigma(1)=1$ and \eqref{eq:layer-B} imply
\eqref{eq:layer-entry}.  The proof is complete.
\end{proof}

\subsection{High-temperature closure}

The pressure cross term in \eqref{eq:first-u} is centered by observing that
$\int_0^1u_x\dd x=0$:
\[
 R\int_0^1\rho\theta u_x\dd x
 =R\int_0^1(\rho\theta-m\bt)u_x\dd x.
\]
The pressure-oscillation estimate was proved in detail above, immediately
before \eqref{eq:layer-cross}, using only
\eqref{eq:entropy-coercivity}, the pointwise temperature bounds, and the
Dirichlet Poincar\'e inequality.  We record the same estimate here in the
form used for the post-layer argument:
\begin{equation}\label{eq:pressure-oscillation}
 R\norm{\rho\theta-m\bt}_{L^2}
 \leq C(\bt\Ecal)^{1/2}+C\norm{\theta-\bt}_{L^2}
 \leq C(\bt\Ecal)^{1/2}+C\norm{\theta_x}_{L^2}.
\end{equation}
Thus, under the pointwise bootstrap, for every
$\varepsilon>0$,
\begin{align}\label{eq:cross-control}
 R\bt^{-2}\left|\int_0^1(\rho\theta-m\bt)u_x\dd x\right|
 &\leq\varepsilon\bt^{\alpha-2}\norm{u_x}_{L^2}^2
 +C_\varepsilon\bt^{-\alpha-1}\Ecal
 +C_\varepsilon\bt^{-\alpha-2}\norm{\theta_x}_{L^2}^2 .
\end{align}
Choose fixed numbers $0<\eta\ll1\ll L$ and set
\begin{align}\label{eq:Y}
 \Ycal(t)={}&\bt^{-2}\Bigl[
 \frac\mu2\int_0^1\theta^\alpha u_x^2\dd x
 -R\int_0^1(\rho\theta-m\bt)u_x\dd x
 +\frac\kappa2\int_0^1\theta^\beta\theta_x^2\dd x
 \nonumber\\
 &\quad+\frac\eta2\bigl(\norm{\sqrt\rho\dot u}_{L^2}^2
                  +\norm{\sqrt\rho\dot\theta}_{L^2}^2\bigr)\Bigr]
 +L\bt^{-\alpha-1}\Ecal(t).
\end{align}
For $\eta$ sufficiently small and $L$ sufficiently large,
\begin{equation}\label{eq:Y-equivalence}
 c\Acal(t)+c\bt^{-\alpha-1}\Ecal(t)
 \leq\Ycal(t)
 \leq C\Acal(t)+C\bt^{-\alpha-1}\Ecal(t).
\end{equation}

We now isolate the nonlinear part of the post-layer argument.  The next
lemma is not a short-time existence statement and does not repeat the
initial-layer estimate of Proposition~\ref{prop:initial-layer}.  Its purpose
is to convert the first material identities and the material-derivative
inequalities into one differential inequality for $\Ycal$.  The terms that
depend on the energy bootstrap number $K$ acquire the small factor
$\bt^{-\gamma}$, whereas the remaining low-order dissipation has a
$K$-independent coefficient.  This separation is what permits the later
choice of the temperature threshold without a circular choice of $K$.

\begin{lemma}[Nonlinear bookkeeping]\label{lem:bookkeeping}
Let \eqref{eq:exponents} hold and fix $M,K\geq1$.  Put
\[
 \Dcal_0=\bt^{\alpha-1}\norm{u_x}_{L^2}^2+\bt^{\beta-2}\norm{\theta_x}_{L^2}^2,\qquad
 \Dcal_h=\bt^{-2}(\norm{\sqrt\rho\dot u}_{L^2}^2+\norm{\sqrt\rho\dot\theta}_{L^2}^2)
          +\bt^{\alpha-2}\norm{\dot u_x}_{L^2}^2+\bt^{\beta-2}\norm{\dot\theta_x}_{L^2}^2 .
\]
There are $C_0=C_0(M,R,\mu,\kappa,\alpha,\beta)$, an integer $N$, and
\[
 \gamma=\min\left\{\alpha-2,\frac{\beta}{2}-1,1\right\}>0
\]
such that the following pointwise-in-time assertion holds.  Let $J$ be any
time interval on which a smooth solution satisfies
\[
 0\leq\rho\leq M,\qquad
 \frac14\bt\leq\theta\leq2\bt,\qquad
 \Acal(t)\leq K,\qquad \bt\geq1.
\]
Then, for almost every $t\in J$,
\begin{align}\label{eq:closure-differential}
 \Ycal'(t)+c\Dcal_h(t)
 \leq C_0\Dcal_0(t)
 +C_0(1+K)^N\bt^{-\gamma}
       \bigl(\Dcal_0(t)+\Dcal_h(t)\bigr).
\end{align}
In particular, the coefficient $C_0$ of the principal low-order
dissipation is independent of the bootstrap number $K$.  No smallness
condition is imposed on the length of $J$; below we take $J=[1,T]$ with
arbitrary $T\geq1$.
\end{lemma}

\begin{proof}
From \eqref{eq:A}, the bootstrap gives
\begin{equation}\label{eq:scaled-size}
 \norm{u_x}_{L^2}\leq K^{1/2}\bt^{1-\alpha/2},\quad
 \norm{\theta_x}_{L^2}\leq K^{1/2}\bt^{1-\beta/2},\quad
 \norm{\sqrt\rho\dot u}_{L^2}+\norm{\sqrt\rho\dot\theta}_{L^2}\leq2K^{1/2}\bt.
\end{equation}
Lemmas~\ref{lem:flux-anchor} and \ref{lem:thermal-anchor} then imply
\begin{equation}\label{eq:scaled-pointwise}
 \norm{u_x}_{L^\infty}\leq C_{M,K}\bt^{1-\alpha},\qquad
 \norm{\theta_x}_{L^\infty}\leq C_{M,K}\bt^{2-\beta},\qquad
 \norm{\theta_{xx}}_{L^2}\leq C_{M,K}\bt^{1-\beta}.
\end{equation}
Here $C_{M,K}$ depends, besides $M$ and $K$, only on
$R,\mu,\kappa,\alpha,\beta$, and is independent of $\inf\rho$.
We now list every nonlinear group from
\eqref{eq:first-u}--\eqref{eq:material-theta}.  Constants depending only on
$M$ and the physical parameters are suppressed.  For any fixed
$\varepsilon>0$, the four terms in
\eqref{eq:first-u}, after multiplication by $\bt^{-2}$, satisfy
\begin{align}
 \bt^{-2}\norm{\sqrt\rho\dot\theta}_{L^2}\norm{u_x}_{L^2}
 &\leq\varepsilon\bt^{-2}\norm{\sqrt\rho\dot\theta}_{L^2}^2
       +C_\varepsilon\bt^{-2}\norm{u_x}_{L^2}^2,
 \label{eq:book-I1}\\
 \bt^{-1}\norm{u_x}_{L^2}^2
 &\quad\text{is retained without further estimation},
 \label{eq:book-I2}\\
 \bt^{\alpha-2}\norm{u_x}_{L^\infty} \norm{u_x}_{L^2}^2
 &\leq\varepsilon\bt^{-2}\norm{\sqrt\rho\dot u}_{L^2}^2
       +C_\varepsilon\bigl(\bt^{-1}+K\bt^{-\alpha}\bigr)
          \norm{u_x}_{L^2}^2,\label{eq:book-I3}\\
 \bt^{\alpha-3}\norm{\dot\theta_x}_{L^2}\norm{u_x}_{L^2}^2
 &\leq\varepsilon\bt^{\beta-2}\norm{\dot\theta_x}_{L^2}^2
       +C_\varepsilon K\bt^{\alpha-\beta-2}\norm{u_x}_{L^2}^2.
\label{eq:book-I4}
\end{align}
Here \eqref{eq:book-I3} follows by inserting the two terms of
\eqref{eq:ux-anchor} before applying Young's inequality.  The four terms in
\eqref{eq:first-theta} obey
\begin{align}
 \bt^{\alpha-2}\norm{\dot\theta_x}_{L^2}\norm{u_x}_{L^2}^2
 &\leq\varepsilon\bt^{\beta-2}\norm{\dot\theta_x}_{L^2}^2
       +C_\varepsilon K\bt^{\alpha-\beta}\norm{u_x}_{L^2}^2,
\label{eq:book-J1}\\
 \bt^{-1}\norm{\sqrt\rho\dot\theta}_{L^2}\norm{u_x}_{L^2}
 &\leq\varepsilon\bt^{-2}\norm{\sqrt\rho\dot\theta}_{L^2}^2
       +C_\varepsilon\norm{u_x}_{L^2}^2,
\label{eq:book-J2}\\
 \bt^{\beta-2}\norm{u_x}_{L^\infty} \norm{\theta_x}_{L^2}^2
 &\leq\varepsilon\bt^{-2}\norm{\sqrt\rho\dot u}_{L^2}^2
       +C_\varepsilon\bigl(\bt^{\beta-\alpha-1}
          +K\bt^{\beta-2\alpha}\bigr)\norm{\theta_x}_{L^2}^2,
\label{eq:book-J3}\\
 \bt^{\beta-3}\norm{\dot\theta_x}_{L^2}\norm{\theta_x}_{L^2}^2
 &\leq\varepsilon\bt^{\beta-2}\norm{\dot\theta_x}_{L^2}^2
       +C_\varepsilon K\bt^{-2}\norm{\theta_x}_{L^2}^2.
\label{eq:book-J4}
\end{align}
The two pressure factors in these displays are harmless constants depending
on $M$.

After multiplication by $\eta\bt^{-2}$, the four terms on the right of
\eqref{eq:material-u} are bounded by
\begin{align}
 \bt^{-\alpha-2}\norm{\sqrt\rho\dot\theta}_{L^2}^2,\qquad
 \bt^{-\alpha}\norm{u_x}_{L^2}^2,\qquad
 \bt^{-\alpha}\norm{u_x}_{L^2}^2
   +K\bt^{-2\alpha}\norm{\sqrt\rho\dot u}_{L^2}^2,\qquad
 K\bt^{-2}\norm{\dot\theta_x}_{L^2}^2 .
\label{eq:book-M}
\end{align}
Finally, the eight scalar groups on the right of
\eqref{eq:material-theta}, in their displayed order, satisfy
\begin{align}
 \bt^{-2}\norm{u_x}_{L^\infty} \norm{\sqrt\rho\dot\theta}_{L^2}^2
 &\leq C(1+K^{1/2})\bt^{-\alpha-1}
       \norm{\sqrt\rho\dot\theta}_{L^2}^2,\label{eq:book-N1}\\
 \bt^{-1}\norm{\sqrt\rho\dot\theta}_{L^2}\norm{\dot u_x}_{L^2}
 &\leq\varepsilon\bt^{\alpha-2}\norm{\dot u_x}_{L^2}^2
       +C_\varepsilon\bt^{-\alpha}\norm{\sqrt\rho\dot\theta}_{L^2}^2,
\label{eq:book-N2a}\\
 \bt^{-1}\norm{\sqrt\rho\dot\theta}_{L^2}\norm{u_x}_{L^\infty} \norm{u_x}_{L^2}
 &\leq\varepsilon\bt^{-2}\norm{\sqrt\rho\dot\theta}_{L^2}^2
       +C_\varepsilon\bt^{2-2\alpha}\norm{u_x}_{L^2}^2
       +C_\varepsilon K\bt^{2-3\alpha}
          \norm{\sqrt\rho\dot u}_{L^2}^2,\label{eq:book-N2b}\\
 \bt^{\alpha-2}\norm{\dot\theta_x}_{L^2}\norm{u_x}_{L^\infty} \norm{u_x}_{L^2}^2
 &\leq\varepsilon\bt^{\beta-2}\norm{\dot\theta_x}_{L^2}^2
       +C_\varepsilon K\bt^{2-\alpha-\beta}\norm{u_x}_{L^2}^2
       +C_\varepsilon K^2\bt^{2-2\alpha-\beta}
          \norm{\sqrt\rho\dot u}_{L^2}^2,
\label{eq:book-N3}\\
 \bt^{\alpha-3}\norm{\dot\theta_x}_{L^2}^2\norm{u_x}_{L^2}^2
 &\leq K\bt^{-1}\norm{\dot\theta_x}_{L^2}^2,\label{eq:book-N4}\\
 \bt^{\alpha-2}\norm{\dot\theta_x}_{L^2}\norm{u_x}_{L^2}\norm{\dot u_x}_{L^2}
 &\leq C K^{1/2}\bigl(
       \bt^{\alpha-\beta/2-1}\norm{\dot u_x}_{L^2}^2
       +\bt^{\beta/2-1}\norm{\dot\theta_x}_{L^2}^2\bigr),
\label{eq:book-N5}\\
 \bt^{\beta-3}\norm{\theta_x}_{L^2}\norm{\dot\theta_x}_{L^2}^2
 &\leq C K^{1/2}\bt^{\beta/2-2}\norm{\dot\theta_x}_{L^2}^2,
 \label{eq:book-N6}\\
 \bt^{\beta-2}\norm{u_x}_{L^\infty} \norm{\theta_x}_{L^2}\norm{\dot\theta_x}_{L^2}
 &\leq\varepsilon\bt^{\beta-2}\norm{\dot\theta_x}_{L^2}^2
       +C_\varepsilon\bt^{\beta-2\alpha}\norm{\theta_x}_{L^2}^2
       +C_\varepsilon K\bt^{-2\alpha}\norm{\sqrt\rho\dot u}_{L^2}^2.
\label{eq:book-N7}
\end{align}
These are exhaustive: \eqref{eq:book-I1}--\eqref{eq:book-I4} correspond to
the four terms in \eqref{eq:first-u},
\eqref{eq:book-J1}--\eqref{eq:book-J4} to the four terms in
\eqref{eq:first-theta}, \eqref{eq:book-M} to
\eqref{eq:material-u}, and \eqref{eq:book-N1}--\eqref{eq:book-N7} to
\eqref{eq:material-theta}.  To illustrate the normalization, the pressure
 work in \eqref{eq:first-u} gives
\[
 R\bt^{-2}\left|\int_0^1\rho\theta u_x^2\dd x\right|
 \leq C\bt^{-1}\norm{u_x}_{L^2}^2,
\]
whereas the mixed material term gives
\begin{align*}
 \bt^{\alpha-2}\norm{\dot\theta_x}_{L^2}\norm{u_x}_{L^2}\norm{\dot u_x}_{L^2}
 &\leq C K^{1/2}\bt^{1-\beta/2}
 \left(\bt^{\alpha-2}\norm{\dot u_x}_{L^2}^2+\bt^{\beta-2}\norm{\dot\theta_x}_{L^2}^2\right).
\end{align*}
Thus the constants in the table are independent of $\inf\rho$.

The choices are made in the following order.  First fix $\eta>0$ in
\eqref{eq:Y} sufficiently small.  Next choose every Young parameter
$\varepsilon$ smaller than a fixed multiple of $\eta$, so that only
$\Dcal_h$ is absorbed.  Then choose $L$ in \eqref{eq:Y} to secure the
equivalence \eqref{eq:Y-equivalence}.  Every remaining power containing
$K$ is bounded by
$C_0(1+K)^N\bt^{-\gamma}$, whereas the $K$-independent low-order terms have
coefficient at most $C_0$.  Adding
$L\bt^{-\alpha-1}$ times \eqref{eq:entropy-identity} now gives
\eqref{eq:closure-differential}.  Only $\Dcal_h$ is absorbed; the
$\Dcal_0$ terms remain on the right and are integrated with the basic
entropy estimate.  The lower bound on $\bt$ is chosen only afterward in
Proposition~\ref{prop:uniform}.
\end{proof}

Proposition~\ref{prop:initial-layer} and
Lemma~\ref{lem:bookkeeping} have complementary roles.  The former controls
the possibly large compatibility traces on $[0,1]$ and supplies the entry
bound \eqref{eq:layer-entry} at $t=1$.  The latter is the post-layer
differential estimate, valid on an interval of arbitrary length.  The next
proposition combines that entry bound, the bookkeeping inequality, and the
basic relative-entropy dissipation to improve all bootstrap assumptions on
$[1,T]$, with constants independent of $T$.

\begin{proposition}[Post-layer vacuum-compatible estimate]\label{prop:uniform}
There exist $M_1,K_1,\Theta_1$, depending only on
$R,\mu,\kappa,\alpha,\beta$, $K_0$, $m_*$, and $m^*$, with the following
property.  Let a smooth positive-density solution exist on $[0,T]$, with
$T\geq1$, and suppose that the entry estimates furnished by
Proposition~\ref{prop:initial-layer} hold at $t=1$.  Suppose also, as
bootstrap assumptions on $[1,T]$, that
\begin{equation}\label{eq:bootstrap}
 \rho\leq2M_1,\qquad \frac14\bt\leq\theta\leq2\bt,\qquad
 \Acal\leq2K_1 .
\end{equation}
If $\bt\geq\Theta_1$, then on $[1,T]$
\begin{equation}\label{eq:uniform-prop}
 \rho\leq M_1,\qquad \frac12\bt\leq\theta\leq\frac32\bt,\qquad
 \sup_{1\leq t\leq T}\Acal(t)+\int_1^T\Dcal(t)\dd t\leq K_1.
\end{equation}
All constants are independent of $\inf\rho$ and of $T$.
\end{proposition}

\begin{proof}
Choose the constants in a triangular order.  First fix
\begin{equation}\label{eq:M1-choice}
 M_1=2e\max\{M_{\rm in}+2,C_Z(m^*+1),1\},
\end{equation}
where $C_Z$ is the fixed pressure constant in the threshold supplied by
Lemma~\ref{lem:zlotnik}.  It is important that $C_Z$ is independent of a
density bootstrap upper bound: the trajectory inequality used below follows
directly from \eqref{eq:flux}, \eqref{eq:F-anchor}, and the bound
$|\Pi|\leq C_{\alpha,R}m\bt$ proved in
Lemma~\ref{lem:flux-anchor}, none of which requires $\rho\leq M$.
Thus \eqref{eq:M1-choice} does not involve the $M$-dependent constants
$C_{\rm F}$ or $C_{\rm G}$.  After $M_1$ has been fixed, apply
Lemma~\ref{lem:bookkeeping} with density bootstrap number $2M_1$.
This fixes $C_{\rm F}(2M_1)$, $C_{\rm G}(2M_1)$, and hence $C_0$ and $N$,
before $K_1$ is chosen.  The order is therefore triangular and contains no
logical circle.

The entry bound \eqref{eq:layer-entry}, the equivalence
\eqref{eq:Y-equivalence}, and $\Ecal(1)\leq\Ecal(0)\leq C\bt$ imply
$\Ycal(1)\leq C_{\rm ent}$, where $C_{\rm ent}$ is fixed by the data class
and, in particular, dominates the equivalence constant times
$1+K_{\rm in}$.  By \eqref{eq:entropy-identity}, there is a fixed
$C_{\rm basic}$ such that, for every $T\geq1$,
\begin{equation}\label{eq:basic-integrated}
 \int_1^T\left(\bt^{\alpha-1}\norm{u_x}_{L^2}^2+
                    \bt^{\beta-2}\norm{\theta_x}_{L^2}^2\right)\dd t\leq C_{\rm basic}.
\end{equation}
Let $C_{\rm eq}$ dominate the fixed constants in
\eqref{eq:layer-entry}, \eqref{eq:Y-equivalence}, and in the integration of
\eqref{eq:closure-differential}, and set
\[
 C_*=2C_{\rm eq}
       \bigl(1+C_{\rm ent}+C_0C_{\rm basic}+K_{\rm in}\bigr).
\]
This number is independent of $T$ and, crucially, does not involve the
bootstrap number $K_1$.  Set
\[
 K_1=2\max\{K_0,C_*\}.
\]
Finally choose $\Theta_1$ so large that
\begin{equation}\label{eq:theta-choice}
 C_0(1+2K_1)^N\Theta_1^{-\gamma}
 \leq\min\{c/2,1\}.
\end{equation}
Only the high-order part of the last term of
\eqref{eq:closure-differential} is absorbed.  Its low-order part is
combined with \eqref{eq:basic-integrated}.  Integration, followed by
\eqref{eq:Y-equivalence}, gives
\begin{equation}\label{eq:A-D-pre}
 \sup_{1\leq t\leq T}\Acal(t)+\int_1^T\Dcal(t)\dd t
 \leq C_*\leq\frac12K_1.
\end{equation}
This improves the energy bootstrap without a circular choice of constants.

It remains to improve the density and temperature bootstrap.  The
logarithmic density inequality was derived in detail above in the proof of
Proposition~\ref{prop:initial-layer}, before
\eqref{eq:layer-P-pointwise}.  The same derivation uses only the continuity
equation, \eqref{eq:flux}, \eqref{eq:Pi}, \eqref{eq:F-anchor}, and the
pointwise temperature bounds, and therefore applies on the present
post-layer interval.  We record it in the form used here:
\begin{equation}\label{eq:logrho}
 y'\leq-c_0\bt^{1-\alpha}e^y
 +C_0m\bt^{1-\alpha}
 +C_0m^{1/2}\bt^{-\alpha}\norm{\sqrt\rho\dot u(t)}_{L^2}.
\end{equation}
Moreover, \eqref{eq:A-D-pre} gives
$\int_1^T \norm{\sqrt\rho\dot u}_{L^2}^2\dd t\leq C_*\bt^2$.  For
$1\leq t_1<t_2\leq T$ and $s=t_2-t_1$,
\begin{equation*}
 C\bt^{-\alpha}\int_{t_1}^{t_2}\norm{\sqrt\rho\dot u}_{L^2}\dd t
 \leq C C_*^{1/2}\bt^{1-\alpha}s^{1/2}\leq C C_*\bt^{1-\alpha}
   +\frac{c_0}{4}\bt^{1-\alpha}s .
\end{equation*}
Increase $\Theta_1$, after $M_1$ and $K_1$ have been fixed, so that
$C C_*\Theta_1^{1-\alpha}\leq1$.  Then
Lemma~\ref{lem:zlotnik}, applied to \eqref{eq:logrho}, yields
\[
 \rho(x,t)\leq e\max\{M_{\rm in}+1,C_Zm\}\leq\frac12M_1,
\]
uniformly for the positive approximations.
This improves the first condition in \eqref{eq:bootstrap}.

Finally, \eqref{eq:A-D-pre} and the equal boundary values of $\theta$ give
\[
 \norm{\theta(t)-\bt}_{L^\infty}\leq\norm{\theta_x(t)}_{L^2}\leq C_*^{1/2}\bt^{1-\beta/2}.
\]
Increase $\Theta_1$ once more so that this is at most $\bt/2$.  This improves the
temperature bootstrap and proves \eqref{eq:uniform-prop}.
\end{proof}

\begin{remark}\label{rem:no-rhox}
Neither Lemma~\ref{lem:bookkeeping} nor Proposition~\ref{prop:uniform}
contains $\norm{\rho_x}_{L^2}$.  This is the decisive difference from an
argument that writes $\rho=m+(\rho-m)$ inside a density-gradient energy and
then attempts to absorb the oscillation by $m$.  Such an absorption already
forces a positive density and cannot be used for genuine vacuum.
\end{remark}

\section{Finite-time higher-order estimates}
\label{sec:higher}

Uniform control of an unweighted density gradient is neither needed nor
true.  What is required for continuation is a bound on every finite time
interval.

The shifted weight never vanishes: $\sigma_{\bt}\geq\bt^{1-\beta}$.
For a current smooth positive-density solution, combine
\eqref{eq:layer-basic}, \eqref{eq:layer-improvement}, and
\eqref{eq:layer-dissipation} on $[0,\min\{1,T\}]$ with
\eqref{eq:uniform-prop} on $[1,T]$ when $T>1$.  It follows that, for each
fixed finite $T$,
\begin{equation}\label{eq:finite-layer-base}
 \sup_{0\leq t\leq T}(\norm{u_x}_{L^2}+\norm{\theta_x}_{L^2}+\norm{\sqrt\rho\dot u}_{L^2}+\norm{\sqrt\rho\dot\theta}_{L^2})
 +\int_0^T(\norm{\dot u_x}_{L^2}^2+\norm{\dot\theta_x}_{L^2}^2)\dd t\leq C_{T,\bt},
\end{equation}
where the constant may depend on $\bt$ but is independent of a positive
lower bound for the density.  This is exactly the type of finite-time input
used below.  In particular, the initial layer causes no loss in the vacuum
compactness or continuation argument.

The first continuation step is to propagate two spatial derivatives of the
density on each fixed finite interval.  This estimate is not claimed to be
uniform as $T\to\infty$; its role is to provide precisely the density
regularity needed by the elliptic recovery and the continuation criterion.

\begin{lemma}[Finite-time propagation of density derivatives]
\label{lem:density-high}
Let a smooth solution satisfy \eqref{eq:pointwise-main} and
\eqref{eq:finite-layer-base} on $[0,T]$.  Then
\begin{equation}\label{eq:density-H2}
 \sup_{0\leq t\leq T}\norm{\rho(t)}_{H^2}\leq C_T,
\end{equation}
where $C_T$ does not depend on a positive lower bound for $\rho$.
\end{lemma}

\begin{proof}
Differentiate $\dot\rho=-\rho u_x$ and eliminate $u_{xx}$ with the momentum
equation.  This gives the exact identity
\begin{align}\label{eq:rhox-transport}
 \Dt\rho_x+
 \left(2u_x+\frac R\mu\rho\theta^{1-\alpha}\right)\rho_x
={}&-\frac1\mu\rho^2\theta^{-\alpha}\dot u
 -\frac R\mu\rho^2\theta^{-\alpha}\theta_x
 +\alpha\rho\theta^{-1}\theta_xu_x .
\end{align}
For any $f\in H^1$ and $u|_{\partial\I}=0$,
\[
 \int_0^1(\Dt f)f\dd x
 =\frac12\frac{\dd}{\dd t}\norm{f}_{L^2}^2
  -\frac12\int_0^1u_xf^2\dd x.
\]
Thus, multiplication of \eqref{eq:rhox-transport} by $\rho_x$ gives the
positive, but degenerate, term
$\frac R\mu\int_0^1\rho\theta^{1-\alpha}\rho_x^2\dd x$.
The three terms on the
right are estimated without dividing by $\rho$:
\begin{align*}
 \left|\int_0^1\rho^2\theta^{-\alpha}\dot u\rho_x\dd x\right|
 &\leq C\norm{\dot u}_{L^2}\norm{\rho_x}_{L^2}
 \leq C \norm{\dot u_x}_{L^2}^2+C\norm{\rho_x}_{L^2}^2,\\
 \left|\int_0^1\rho^2\theta^{-\alpha}\theta_x\rho_x\dd x\right|
 &\leq C \norm{\theta_x}_{L^2}^2+C\norm{\rho_x}_{L^2}^2,\\
 \left|\int_0^1\rho\theta^{-1}\theta_xu_x\rho_x\dd x\right|
 &\leq C\norm{\theta_x}_{L^\infty} \norm{u_x}_{L^2}\norm{\rho_x}_{L^2}
 \leq C\norm{\theta_x}_{L^\infty}^2\norm{u_x}_{L^2}^2+C\norm{\rho_x}_{L^2}^2.
\end{align*}
Together with \eqref{eq:dot-poincare} and the uniform pointwise bounds,
these estimates imply
\begin{align}\label{eq:rhox-gronwall}
 \frac{\dd}{\dd t}\norm{\rho_x}_{L^2}^2
 +c\bt^{1-\alpha}\int_0^1\rho\rho_x^2\dd x
 \leq C(1+\norm{u_x}_{L^\infty})\norm{\rho_x}_{L^2}^2
 +C\bigl(\norm{\dot u_x}_{L^2}^2+\norm{\theta_x}_{L^2}^2+\norm{\theta_x}_{L^\infty}^2\norm{u_x}_{L^2}^2\bigr).
\end{align}
The last three terms are integrable on every finite interval by
\eqref{eq:finite-layer-base}.  Gronwall's inequality yields
\[
 \sup_{t\leq T}\norm{\rho_x(t)}_{L^2}\leq C_T.
\]
Notice that the positive dissipation in \eqref{eq:rhox-gronwall} is the
degenerate quantity $\int_0^1\rho\rho_x^2\dd x$, not
$\norm{\rho_x}_{L^2}^2$.

Put
\[
 A=2u_x+\frac R\mu\rho\theta^{1-\alpha},\qquad
 B=-\frac1\mu\rho^2\theta^{-\alpha}\dot u
 -\frac R\mu\rho^2\theta^{-\alpha}\theta_x
 +\alpha\rho\theta^{-1}\theta_xu_x .
\]
Differentiating \eqref{eq:rhox-transport} once more gives
\begin{equation}\label{eq:rhoxx-transport}
 \Dt\rho_{xx}+(A+u_x)\rho_{xx}=B_x-A_x\rho_x .
\end{equation}
No second derivative of $\rho$ occurs on the right.  Indeed,
\begin{align}
 A_x={}&2u_{xx}+\frac R\mu\theta^{1-\alpha}\rho_x
 +\frac R\mu(1-\alpha)\rho\theta^{-\alpha}\theta_x,
\label{eq:Ax}\\
 B_x={}&-\frac1\mu\left(
 2\rho\rho_x\theta^{-\alpha}\dot u
 -\alpha\rho^2\theta^{-\alpha-1}\theta_x\dot u
 +\rho^2\theta^{-\alpha}\dot u_x\right)\nonumber\\
&-\frac R\mu\left(
 2\rho\rho_x\theta^{-\alpha}\theta_x
 -\alpha\rho^2\theta^{-\alpha-1}\theta_x^2
 +\rho^2\theta^{-\alpha}\theta_{xx}\right)\nonumber\\
&+\alpha\left(
 \rho_x\theta^{-1}\theta_xu_x
 -\rho\theta^{-2}\theta_x^2u_x
 +\rho\theta^{-1}\theta_{xx}u_x
 +\rho\theta^{-1}\theta_xu_{xx}\right).
\label{eq:Bx}
\end{align}
The momentum equation further gives
\begin{equation}\label{eq:uxx-elliptic}
 \mu\theta^\alpha u_{xx}
 =\rho\dot u+R\theta\rho_x+R\rho\theta_x
  -\mu\alpha\theta^{\alpha-1}\theta_xu_x,
\end{equation}
so the already obtained $\rho_x$ estimate, \eqref{eq:dot-poincare}, and
\eqref{eq:finite-layer-base} give
\begin{equation}\label{eq:uxx-finite}
 \norm{u_{xx}}_{L^2}\leq C_T(1+\norm{\dot u_x}_{L^2}+\norm{\sqrt\rho\dot u}_{L^2}+\norm{\theta_x}_{L^2}).
\end{equation}
Put $Y(t)=1+\norm{\rho_{xx}(t)}_{L^2}^2$.  Using the one-dimensional inequality
\[
 \norm{\rho_x}_{L^\infty}^2
 \leq C\norm{\rho_x}_{L^2}
       \bigl(\norm{\rho_x}_{L^2}+\norm{\rho_{xx}}_{L^2}\bigr),
\]
the first-derivative bound yields $\norm{\rho_x}_{L^\infty}^2\leq C_TY$.
The products in \eqref{eq:Ax}--\eqref{eq:Bx} fall into the following
classes:
\begin{align*}
 \norm{\rho_x\dot u}_{L^2}^2
 &\leq\norm{\rho_x}_{L^2}^2\norm{\dot u}_{L^\infty}^2
 \leq C_T\norm{\dot u_x}_{L^2}^2,\\
 \norm{u_{xx}\rho_x}_{L^2}^2
 &\leq\norm{u_{xx}}_{L^2}^2\norm{\rho_x}_{L^\infty}^2
 \leq C_T\norm{u_{xx}}_{L^2}^2Y,\\
 \norm{\rho_x^2}_{L^2}^2
 &\leq\norm{\rho_x}_{L^\infty}^2\norm{\rho_x}_{L^2}^2
 \leq C_TY,\\
 \norm{\rho_x\theta_x}_{L^2}^2
 &\leq C_T\norm{\theta_x}_{L^\infty}^2,\\
 \norm{\theta_xu_{xx}}_{L^2}^2
 &\leq\norm{\theta_x}_{L^\infty}^2\norm{u_{xx}}_{L^2}^2,\\
 \norm{\theta_x^2}_{L^2}^2
 &\leq\norm{\theta_x}_{L^\infty}^2\norm{\theta_x}_{L^2}^2,\\
 \norm{\dot u_x}_{L^2}^2+\norm{\theta_{xx}}_{L^2}^2
 &=\norm{\dot u_x}_{L^2}^2+\norm{\theta_{xx}}_{L^2}^2.
\end{align*}
Every remaining factor in \eqref{eq:Ax}--\eqref{eq:Bx} is a bounded
coefficient multiplying one of these seven classes.  Consequently,
\begin{equation}\label{eq:AB-estimate}
 \norm{B_x-A_x\rho_x}_{L^2}^2
 \leq H_T(t)\bigl(1+\norm{\rho_{xx}}_{L^2}^2\bigr),
 \qquad H_T\in L^1(0,T).
\end{equation}
One admissible choice is
\[
 H_T(t)=C_T\left[1+\norm{\dot u_x}_{L^2}^2+\norm{\dot\theta_x}_{L^2}^2+\norm{u_{xx}}_{L^2}^2
 +\norm{\theta_{xx}}_{L^2}^2+\norm{u_x}_{L^\infty}^2
 +\norm{\theta_x}_{L^\infty}^4\right].
\]
Its integrability follows from \eqref{eq:finite-layer-base},
\eqref{eq:thetaxx}, and \eqref{eq:uxx-finite}.  Multiplication of
\eqref{eq:rhoxx-transport} by $\rho_{xx}$ and use of
\eqref{eq:AB-estimate} yield
\[
 \frac{\dd}{\dd t}\bigl(1+\norm{\rho_{xx}}_{L^2}^2\bigr)
 \leq C\bigl(1+\norm{u_x}_{L^\infty}+H_T(t)\bigr)
       \bigl(1+\norm{\rho_{xx}}_{L^2}^2\bigr).
\]
Another application of Gronwall proves \eqref{eq:density-H2}.
\end{proof}

Once the density derivatives have been propagated, the momentum and
temperature equations can be used as elliptic equations for $u$ and
$\theta$.  The next lemma recovers the spatial derivatives required for
continuation, with estimates that remain independent of the positive lower
bound of the density.

\begin{lemma}[Elliptic recovery]\label{lem:elliptic-recovery}
Under the assumptions of Lemma~\ref{lem:density-high},
\begin{align}\label{eq:elliptic-finite}
 &\sup_{0\leq t\leq T}
 \left(\norm{u(t)}_{H^2}+\norm{\theta(t)-\bt}_{H^2}\right)
 \leq C_T,\\
 &\int_0^T
 \left(\norm{u}_{H^3}^2+\norm{\theta-\bt}_{H^3}^2\right)\dd t
 \leq C_T.\nonumber
\end{align}
\end{lemma}

\begin{proof}
We first recover the second derivatives.  Formula
\eqref{eq:uxx-elliptic}, the bounds for $\rho$ and $\theta$, and
$\norm{\rho\dot u}_{L^2}\leq M^{1/2}
\norm{\sqrt\rho\dot u}_{L^2}$ give
\begin{align*}
 \norm{u_{xx}}_{L^2}
 \leq C_T\bigl(&\norm{\sqrt\rho\dot u}_{L^2}
 +\norm{\rho_x}_{L^2}+\norm{\theta_x}_{L^2}
 +\norm{\theta_x}_{L^\infty}\norm{u_x}_{L^2}\bigr).
\end{align*}
Lemma~\ref{lem:thermal-anchor}, Lemma~\ref{lem:density-high}, and
\eqref{eq:finite-layer-base} make the right-hand side uniformly bounded on
$[0,T]$.  Since $u=0$ at the endpoints, Poincar\'e's inequality now yields
$\sup_{t\leq T}\norm{u(t)}_{H^2}\leq C_T$.

Similarly, the temperature equation gives
\begin{equation}\label{eq:thetaxx-elliptic}
 \kappa\theta^\beta\theta_{xx}
 =\rho\dot\theta+R\rho\theta u_x-\mu\theta^\alpha u_x^2
  -\kappa\beta\theta^{\beta-1}\theta_x^2,
\end{equation}
and therefore
\begin{align*}
 \norm{\theta_{xx}}_{L^2}
 \leq C_T\bigl(&\norm{\sqrt\rho\dot\theta}_{L^2}
 +\norm{u_x}_{L^2}
 +\norm{u_x}_{L^\infty}\norm{u_x}_{L^2}
 +\norm{\theta_x}_{L^\infty}\norm{\theta_x}_{L^2}\bigr).
\end{align*}
The flux anchor \eqref{eq:ux-anchor}, the thermal anchor
\eqref{eq:thetax-infty}, and \eqref{eq:finite-layer-base} bound this
right-hand side uniformly.  Applying the Dirichlet Poincar\'e inequality to
$\theta-\bt$ proves
$\sup_{t\leq T}\norm{\theta(t)-\bt}_{H^2}\leq C_T$.

Differentiating the two elliptic identities gives
\begin{align}
 \mu\theta^\alpha u_{xxx}={}&
 \rho_x\dot u+\rho\dot u_x+R\theta\rho_{xx}
 +2R\rho_x\theta_x+R\rho\theta_{xx}
 -(\mu\theta^\alpha)_{xx}u_x
 -2(\mu\theta^\alpha)_xu_{xx},\label{eq:uxxx}\\
 \kappa\theta^\beta\theta_{xxx}={}&
 \rho_x\dot\theta+\rho\dot\theta_x+(R\rho\theta)_xu_x
 +R\rho\theta u_{xx}
 -(\mu\theta^\alpha)_xu_x^2-2\mu\theta^\alpha u_xu_{xx}
 -(\kappa\theta^\beta)_{xx}\theta_x
 -2(\kappa\theta^\beta)_x\theta_{xx}.
\label{eq:thetaxxx}
\end{align}
We estimate every class of terms in these identities.  First,
\begin{align*}
 \norm{\rho_x\dot u}_{L^2}
 &\leq\norm{\rho_x}_{L^2}\norm{\dot u}_{L^\infty}
 \leq C_T\norm{\dot u_x}_{L^2},
 &\norm{\rho\dot u_x}_{L^2}&\leq M\norm{\dot u_x}_{L^2},\\
 \norm{\rho_x\dot\theta}_{L^2}
 &\leq C_T\norm{\dot\theta_x}_{L^2},
 &\norm{\rho\dot\theta_x}_{L^2}&\leq M\norm{\dot\theta_x}_{L^2},
\end{align*}
where \eqref{eq:dot-poincare} is used in the first and third estimates.
The density terms satisfy
\begin{align*}
 \norm{\theta\rho_{xx}}_{L^2}
 +\norm{\rho_x\theta_x}_{L^2}
 +\norm{\rho\theta_{xx}}_{L^2}
 \leq C_T,
\end{align*}
because $\rho\in L^\infty(0,T;H^2)$ and
$\theta-\bt\in L^\infty(0,T;H^2)$.

For the variable coefficients, direct differentiation gives
\begin{align*}
 (\theta^\alpha)_x=\alpha\theta^{\alpha-1}\theta_x, \qquad (\theta^\alpha)_{xx}
 =\alpha\theta^{\alpha-1}\theta_{xx}
   +\alpha(\alpha-1)\theta^{\alpha-2}\theta_x^2,
\end{align*}
and the same formulas hold with $\alpha$ replaced by $\beta$.
Consequently, using $H^1(0,1)\hookrightarrow L^\infty(0,1)$,
\begin{align*}
 \norm{(\theta^\alpha)_{xx}u_x}_{L^2}
 &\leq C_T\bigl(\norm{\theta_{xx}}_{L^2}
    +\norm{\theta_x}_{L^\infty}\norm{\theta_x}_{L^2}\bigr)
       \norm{u_x}_{L^\infty}\leq C_T,\\
 \norm{(\theta^\alpha)_xu_{xx}}_{L^2}
 &\leq C_T\norm{\theta_x}_{L^\infty}\norm{u_{xx}}_{L^2}\leq C_T,\\
 \norm{(\theta^\beta)_{xx}\theta_x}_{L^2}
 +\norm{(\theta^\beta)_x\theta_{xx}}_{L^2}
 &\leq C_T\bigl(\norm{\theta_x}_{L^\infty}
       \norm{\theta_{xx}}_{L^2}
       +\norm{\theta_x}_{L^\infty}^2
        \norm{\theta_x}_{L^2}\bigr)\leq C_T.
\end{align*}
The remaining products in \eqref{eq:thetaxxx} obey
\begin{align*}
 \norm{(\rho\theta)_xu_x}_{L^2}
 +\norm{\rho\theta u_{xx}}_{L^2}
 &\leq C_T,\\
 \norm{(\theta^\alpha)_xu_x^2}_{L^2}
 +\norm{\theta^\alpha u_xu_{xx}}_{L^2}
 &\leq C_T,
\end{align*}
using the already established $H^2$ bounds for $u$ and $\theta$.
After division by the uniformly positive coefficients
$\mu\theta^\alpha$ and $\kappa\theta^\beta$, we therefore obtain
\begin{align*}
 \norm{u_{xxx}}_{L^2}
 +\norm{\theta_{xxx}}_{L^2}
 \leq C_T\bigl(1+\norm{\dot u_x}_{L^2}
                    +\norm{\dot\theta_x}_{L^2}\bigr).
\end{align*}
Squaring and integrating over $(0,T)$, then using
\eqref{eq:finite-layer-base}, proves the time-integrated $H^3$ estimate.
Every constant used above depends on the density only through its upper
bound and its $H^2$ norm; no density lower bound, $F_{xx}$ estimate, or
auxiliary-flux estimate is hidden in the argument.
\end{proof}

\begin{corollary}[Finite-time continuation bound]\label{cor:continuation}
For every $T<\infty$, a smooth positive-density solution satisfying
the pointwise bounds \eqref{eq:pointwise-main} and
\eqref{eq:finite-layer-base} obeys \eqref{eq:finite-high}, with a constant
independent of the density lower bound.
\end{corollary}

\begin{proof}
Lemmas~\ref{lem:density-high} and \ref{lem:elliptic-recovery} give
\begin{align*}
 &\sup_{0\leq t\leq T}
 \bigl(\norm{\rho(t)}_{H^2}+\norm{u(t)}_{H^2}
       +\norm{\theta(t)-\bt}_{H^2}\bigr)\leq C_T,\\
 &\int_0^T\bigl(\norm{u}_{H^3}^2
       +\norm{\theta-\bt}_{H^3}^2\bigr)\dd t\leq C_T.
\end{align*}
Moreover, \eqref{eq:dot-poincare} and
\eqref{eq:finite-layer-base} imply
\begin{align*}
 \int_0^T\bigl(\norm{\dot u}_{H^1}^2
                 +\norm{\dot\theta}_{H^1}^2\bigr)\dd t
 \leq C\int_0^T\bigl(\norm{\dot u_x}_{L^2}^2
                 +\norm{\dot\theta_x}_{L^2}^2\bigr)\dd t\leq C_T.
\end{align*}
These three bounds are precisely \eqref{eq:finite-high}.

For completeness, the time derivatives required by the strong-solution
class follow from
\[
 \rho_t=-u\rho_x-\rho u_x,\qquad
 u_t=\dot u-uu_x,\qquad
 \theta_t=\dot\theta-u\theta_x
\]
and the one-dimensional $H^1$ product inequality.  Indeed,
\begin{align*}
 \norm{\rho_t}_{H^1}
 &\leq C\norm{u}_{H^2}\norm{\rho}_{H^2},\\
 \norm{u_t}_{H^1}
 &\leq\norm{\dot u}_{H^1}+C\norm{u}_{H^2}^2,\\
 \norm{\theta_t}_{H^1}
 &\leq\norm{\dot\theta}_{H^1}
      +C\norm{u}_{H^2}\norm{\theta}_{H^2}.
\end{align*}
Hence $\rho_t\in L^\infty(0,T;H^1)$ and
$u_t,\theta_t\in L^2(0,T;H^1)$.  All constants come from the preceding
vacuum-uniform estimates and are therefore independent of the positive
lower bound for $\rho$.
\end{proof}

\section{Construction and continuation}
\label{sec:construction}

We first record the linear estimates used in the positive-density
construction.  They are stated in the precise Sobolev scale needed below so
that no H\"older-space local theorem is invoked implicitly.

\begin{lemma}[Transport and parabolic estimates]\label{lem:linear-local}
Let $0<\tau\leq1$.  Suppose
\[
 v\in L^\infty(0,\tau;H^2\cap H_0^1)\cap L^2(0,\tau;H^3),
 \qquad v_t\in L^2(0,\tau;H^1).
\]
Then the solution of
\[
 r_t+vr_x+rv_x=0,\qquad r|_{t=0}=r_0\in H^2,
\]
satisfies
\begin{align}
 \norm{r(t)}_{H^2}^2
 &\leq \norm{r_0}_{H^2}^2
 \exp\left(C\int_0^t\bigl(1+\norm{v(s)}_{H^3}\bigr)\dd s\right),
 \label{eq:linear-transport-H2}\\
 r(X_v(a,t),t)
 &=r_0(a)\exp\left(-\int_0^t v_x(X_v(a,s),s)\dd s\right),
 \label{eq:linear-transport-positive}
\end{align}
where $(X_v)_t=v(X_v,t)$.

Let, in addition, $q,A\in L^\infty(0,\tau;H^2)$,
$q_t,A_t\in L^2(0,\tau;H^1)$ and
\[
 0<q_*\leq q\leq q^*,\qquad 0<A_*\leq A\leq A^*.
\]
For the Dirichlet problem
\begin{equation}\label{eq:linear-parabolic}
 qw_t+qvw_x-(Aw_x)_x=f,\qquad
 w|_{\partial\I}=0,\qquad w|_{t=0}=w_0,
\end{equation}
assume $w_0\in H^2\cap H_0^1$, $f\in L^2(0,\tau;H^1)$,
$f_t\in L^2(0,\tau;H^{-1})$, and that the equation at $t=0$
defines $w_t(0)\in L^2$.  Then \eqref{eq:linear-parabolic} has a unique
solution satisfying
\begin{align}
 &w\in C([0,\tau];H^2\cap H_0^1)\cap L^2(0,\tau;H^3),\nonumber\\
 &w_t\in C([0,\tau];L^2)\cap L^2(0,\tau;H^1),
 \label{eq:linear-parabolic-regularity}
\end{align}
and
\begin{align}
 &\sup_{t\leq\tau}\left(\norm{w(t)}_{H^2}^2
                         +\norm{w_t(t)}_{L^2}^2\right)
 +\int_0^\tau\left(\norm{w}_{H^3}^2
                         +\norm{w_t}_{H^1}^2\right)\dd t
 \nonumber\\
 &\qquad\leq C_{\rm lin}\left[
  \norm{w_0}_{H^2}^2+\norm{w_t(0)}_{L^2}^2
  +\norm{f}_{L^2(0,\tau;H^1)}^2+\norm{f_t}_{L^2(0,\tau;H^{-1})}^2\right].
 \label{eq:linear-parabolic-estimate}
\end{align}
Here $C_{\rm lin}$ depends only on the displayed positive bounds and on the
coefficient norms, and is locally bounded as a function of those quantities.
\end{lemma}

\begin{proof}
We first prove the transport estimate.  Multiplying the equation by $r$
and integrating the transport term by parts gives
\[
 \frac12\frac{\dd}{\dd t}\norm{r}_{L^2}^2
 +\frac12\int_0^1v_xr^2\dd x=0.
\]
The first and second differentiated equations are
\begin{align*}
 r_{xt}+vr_{xx}+2v_xr_x+v_{xx}r&=0,\\
 r_{xxt}+vr_{xxx}+3v_xr_{xx}+3v_{xx}r_x+v_{xxx}r&=0.
\end{align*}
Testing them by $r_x$ and $r_{xx}$, respectively, and integrating the
terms containing $v r_{xx}$ and $v r_{xxx}$ by parts, we obtain
\begin{align*}
 \frac12\frac{\dd}{\dd t}\norm{r_x}_{L^2}^2
 \leq C\norm{v}_{H^2}\norm{r}_{H^1}^2, \qquad
 \frac12\frac{\dd}{\dd t}\norm{r_{xx}}_{L^2}^2
 \leq C\norm{v}_{H^3}\norm{r}_{H^2}^2.
\end{align*}
Here, for example,
\begin{align*}
 \left|\int_0^1v_{xx}r_xr_{xx}\dd x\right|
 &\leq \norm{v_{xx}}_{L^2}\norm{r_x}_{L^\infty}
             \norm{r_{xx}}_{L^2},\\
 \left|\int_0^1v_{xxx}rr_{xx}\dd x\right|
 &\leq \norm{v_{xxx}}_{L^2}\norm r_{L^\infty}
             \norm{r_{xx}}_{L^2},
\end{align*}
and $H^1(\I)\hookrightarrow L^\infty(\I)$.  Adding the three
inequalities yields
\[
 \frac{\dd}{\dd t}\norm r_{H^2}^2
 \leq C\bigl(1+\norm v_{H^3}\bigr)\norm r_{H^2}^2,
\]
so Gronwall's inequality proves \eqref{eq:linear-transport-H2}.  Along the
flow $X_v$,
\[
 \frac{\dd}{\dd t}r(X_v(a,t),t)
 =-v_x(X_v(a,t),t)r(X_v(a,t),t),
\]
and solving this scalar ODE gives \eqref{eq:linear-transport-positive}.

We next give the parabolic estimates.  They may first be performed on a
Galerkin approximation.  Testing \eqref{eq:linear-parabolic} by $w_t$ and
using $w_t=0$ at the boundary gives the exact identity
\begin{align*}
 \int_0^1q|w_t|^2\dd x
 +\frac12\frac{\dd}{\dd t}\int_0^1A|w_x|^2\dd x
 =\int_0^1(f-qvw_x)w_t\dd x
  +\frac12\int_0^1A_t|w_x|^2\dd x.
\end{align*}
For any $\varepsilon>0$, the right-hand side is bounded by
\begin{align*}
 \varepsilon\norm{w_t}_{L^2}^2
 +C_\varepsilon\bigl(\norm f_{L^2}^2
       +\norm v_{L^\infty}^2\norm{w_x}_{L^2}^2\bigr)
 +\frac12\norm{A_t}_{L^\infty}\norm{w_x}_{L^2}^2.
\end{align*}
Since $A_t\in L^2H^1\hookrightarrow L^1L^\infty$ on the finite time
interval, choosing $\varepsilon<q_*/2$ and applying Gronwall's inequality
gives
\begin{equation}\label{eq:linear-first-energy}
 \sup_{t\leq\tau}\norm{w_x(t)}_{L^2}^2
 +\int_0^\tau\norm{w_t}_{L^2}^2\dd t
 \leq C_{\rm lin}\left(\norm{w_0}_{H^1}^2
                    +\norm f_{L^2(0,\tau;L^2)}^2\right).
\end{equation}

To expose the estimate for $w_t$, differentiate
\eqref{eq:linear-parabolic} in time.  The resulting equation is
\begin{align*}
 qw_{tt}+qv w_{tx}-(Aw_{tx})_x
 ={}& f_t-q_tw_t-(q_tv+qv_t)w_x+(A_tw_x)_x.
\end{align*}
This identity is justified at the Galerkin level, or equivalently by time
difference quotients.  Pair it with $w_t\in H_0^1$.  The principal terms
satisfy
\begin{align*}
 \int_0^1q w_{tt}w_t\dd x
 &=\frac12\frac{\dd}{\dd t}\int_0^1q|w_t|^2\dd x
   -\frac12\int_0^1q_t|w_t|^2\dd x,\\
 -\int_0^1(Aw_{tx})_xw_t\dd x
 &=\int_0^1A|w_{tx}|^2\dd x,
\end{align*}
whereas
\begin{align*}
 \left|\langle f_t,w_t\rangle\right|
 &\leq\varepsilon\norm{w_{tx}}_{L^2}^2
      +C_\varepsilon\norm{f_t}_{H^{-1}}^2,\\
 \left|\int_0^1A_tw_xw_{tx}\dd x\right|
 &\leq\varepsilon\norm{w_{tx}}_{L^2}^2
      +C_\varepsilon\norm{A_t}_{L^\infty}^2\norm{w_x}_{L^2}^2.
\end{align*}
The convection and coefficient-derivative terms are estimated in exactly
the same way, using
$q_t,v_t,A_t\in H^1\hookrightarrow L^\infty$ and the coefficient bounds;
more explicitly, their sum is at most
\begin{align*}
 \varepsilon\norm{w_{tx}}_{L^2}^2
 +C_\varepsilon\bigl(1+\norm{q_t}_{H^1}^2
     +\norm{v_t}_{H^1}^2+\norm{A_t}_{H^1}^2\bigr)
       \bigl(\norm{w_t}_{L^2}^2+\norm{w_x}_{L^2}^2\bigr).
\end{align*}
After choosing $\varepsilon$ so that these terms are absorbed by
$A_*\norm{w_{tx}}_{L^2}^2$, integration and Gronwall's inequality, together
with \eqref{eq:linear-first-energy}, yield
\begin{align*}
 \sup_{t\leq\tau}\norm{w_t(t)}_{L^2}^2
 +\int_0^\tau\norm{w_{tx}}_{L^2}^2\dd t
 \leq C_{\rm lin}\left[
  \norm{w_t(0)}_{L^2}^2+\norm{f}_{L^2(0,\tau;H^1)}^2
  +\norm{f_t}_{L^2(0,\tau;H^{-1})}^2+\norm{w_0}_{H^2}^2\right].
\end{align*}

It remains to recover the spatial derivatives without differentiating the
Dirichlet condition.  Put $h=qw_t+qvw_x-f$.  From
\[
 Aw_{xx}=h-A_xw_x
\]
and the estimates above, together with the standard implication
$f\in L^2H^1$, $f_t\in L^2H^{-1}\Rightarrow f\in C L^2$, we obtain
$w\in L^\infty H^2$.  Differentiating the same elliptic identity once gives
\[
 Aw_{xxx}=h_x-A_{xx}w_x-2A_xw_{xx}.
\]
Every term on the right belongs to $L^2(0,\tau;L^2)$: in particular,
$h_x$ is controlled by $w_{tx}$, $w_x$, $w_{xx}$, $f_x$, and the displayed
coefficient norms.  Division by $A\geq A_*$ therefore gives
$w\in L^2H^3$ and completes \eqref{eq:linear-parabolic-estimate}.

The estimates are uniform in the Galerkin dimension, so weak compactness
and the strong convergence of lower-order terms produce a solution.
The time-difference-quotient bounds give
$w_t\in C([0,\tau];L^2)$; the elliptic identity then gives
$w\in C([0,\tau];H^2)$.  Finally, applying
\eqref{eq:linear-first-energy} to the difference of two solutions with zero
initial data gives uniqueness.  This proves
\eqref{eq:linear-parabolic-regularity}--\eqref{eq:linear-parabolic-estimate}.
\end{proof}

With the transport and parabolic estimates available, we can now construct
the nonlinear solution for strictly positive approximate densities.  The
next lemma also identifies the continuation criterion that will later be
verified uniformly in the approximation parameter by the a priori
estimates.

\begin{lemma}[Positive-density local solution and continuation]
\label{lem:positive-local}
Suppose the data satisfy \eqref{eq:data-dynamic},
$\rho_0\in H^2$, $\rho_0\geq\delta>0$, and $\theta_0\geq c_0>0$, and suppose
there exist $g_1,g_2\in L^2(0,1)$ for which \eqref{eq:compatibility} holds.
Then a unique strong solution exists on a nontrivial interval.  It can be
continued past $T<\infty$ provided
\[
 \inf_{[0,1]\times[0,T]}\rho>0,\qquad
 \inf_{[0,1]\times[0,T]}\theta>0,
\]
and the norms in \eqref{eq:finite-high} are finite.
\end{lemma}

\begin{proof}
We use a Picard construction based on Lemma~\ref{lem:linear-local}; its
structure is the positive-density iteration introduced in the classical
compressible-flow theory~\cite{Tani1977}, but the estimates below are in the
Sobolev spaces of Definition~\ref{def:strong}.  Set
$\vartheta=\theta-\bt$.  Starting from iterates having the same initial and
Dirichlet data, define
$\rho^{n+1}$ by
\[
 \rho^{n+1}_t+u^n\rho^{n+1}_x+\rho^{n+1}u_x^n=0
\]
and solve the two linear parabolic Dirichlet problems
\begin{align*}
 \rho^{n+1}(u^{n+1}_t+u^n u^{n+1}_x)
 -\partial_x\!\left(\mu(\theta^n)^\alpha u^{n+1}_x\right)
 &=-\partial_x(R\rho^{n+1}\theta^n),\\
 \rho^{n+1}(\theta^{n+1}_t+u^n\theta^{n+1}_x)
 -\partial_x\!\left(\kappa(\theta^n)^\beta\theta^{n+1}_x\right)
 &=-R\rho^{n+1}\theta^n u_x^n
   +\mu(\theta^n)^\alpha(u_x^n)^2 .
\end{align*}
The powers of $\theta^n$ in these two equations are frozen coefficients.
The characteristic formula is
\begin{equation}\label{eq:local-density-formula}
 \rho^{n+1}(X^n(a,t),t)=\rho_0(a)
 \exp\left(-\int_0^t u_x^n(X^n(a,s),s)\dd s\right).
\end{equation}
Thus every iterate satisfies $\rho^{n+1}\geq\delta/2$ after $\tau$ is
chosen so that
$\int_0^\tau\norm{u_x^n}_{L^\infty}\dd t\leq\log2$.
Use the norm
\begin{align*}
 \norm{(\rho,u,\vartheta)}_{\mathcal X_\tau}^2={}&
 \sup_{t\leq\tau}\norm{(\rho,u,\theta-\bt)(t)}_{H^2}^2
 +\sup_{t\leq\tau}\left[
   \norm{\rho_t(t)}_{H^1}^2+\norm{u_t(t)}_{L^2}^2
   +\norm{\vartheta_t(t)}_{L^2}^2\right]\\
 &+\int_0^\tau\left[
  \norm{(u,\vartheta)}_{H^3}^2
 +\norm{(u_t,\vartheta_t)}_{H^1}^2\right]\dd t .
\end{align*}
Let $B_L\subset\mathcal X_\tau$ be the closed ball determined by this norm,
the two Dirichlet conditions, and
$\rho\geq\delta/2$, $\theta\geq c_0/2$.  On $B_L$ the density multiplier and
the two diffusion coefficients are uniformly positive.  The product laws
in $H^1$ and $H^2$ show that the two parabolic right-hand sides belong to
$L^2H^1$ and their time derivatives to $L^2H^{-1}$.  Applying
Lemma~\ref{lem:linear-local} to the transport equation and then to the two
Dirichlet equations gives
\begin{equation}\label{eq:local-iteration}
 \norm{(\rho^{n+1},u^{n+1},\theta^{n+1}-\bt)}_{\mathcal X_\tau}^2
 \leq C_{\delta,c_0}(\text{data})
 +C_{\delta,c_0}\tau^{1/2}(1+L)^3 .
\end{equation}
Here the parabolic estimate is applied with homogeneous boundary data for
$u$ and $\vartheta$; it does not rely on discarding boundary terms after two
spatial differentiations.
Furthermore,
\[
 \sup_{t\leq\tau}\norm{\theta^{n+1}(t)-\theta_0}_{L^\infty}
 \leq \tau^{1/2}\norm{\theta_t^{n+1}}_{L^2(0,\tau;H^1)}.
\]
After $L$ is fixed, reducing $\tau$ makes this quantity at most $c_0/2$;
hence the temperature lower bound is also invariant under the iteration.
The compatibility conditions give the $L^2$ traces
$\sqrt{\rho_0}\dot u(0)=g_1$ and
$\sqrt{\rho_0}\dot\theta(0)=g_2$; since $\rho_0\geq\delta$, these are the
initial time-derivative quantities in
\eqref{eq:linear-parabolic-estimate}.  Choosing first
$L=2C_{\delta,c_0}(\text{data})$ and then $\tau$ so that the second term in
\eqref{eq:local-iteration} is at most $L/2$ makes the iteration invariant.
For completeness, subtract the equations for two consecutive iterates and
put
\begin{align*}
 d_n:={}&\sup_{0\leq t\leq\tau}
 \bigl(\norm{Z_\rho^n(t)}_{L^2}^2+\norm{Z_u^n(t)}_{L^2}^2
       +\norm{Z_\theta^n(t)}_{L^2}^2\bigr)
 +\int_0^\tau\bigl(\norm{Z_{u,x}^n}_{L^2}^2
                         +\norm{Z_{\theta,x}^n}_{L^2}^2\bigr)\dd t,
\end{align*}
where $Z^n$ is the difference of the $n$th and $(n-1)$st triples.
The density difference satisfies the explicit equation
\begin{equation}\label{eq:local-density-difference}
 Z_{\rho,t}^{n+1}+u^nZ_{\rho,x}^{n+1}+u_x^nZ_\rho^{n+1}
 =-Z_u^n\rho_x^n-\rho^nZ_{u,x}^n,
 \qquad Z_\rho^{n+1}(0)=0.
\end{equation}
Testing by $Z_\rho^{n+1}$, integrating the transport term by parts, and
using $H^1\hookrightarrow L^\infty$ gives, for every $\varepsilon>0$,
\begin{align}
 \frac{\dd}{\dd t}\norm{Z_\rho^{n+1}}_{L^2}^2
 \leq{}& C_L\norm{Z_\rho^{n+1}}_{L^2}^2
 +\varepsilon\norm{Z_{u,x}^n}_{L^2}^2
 +C_{L,\varepsilon}\norm{Z_u^n}_{L^\infty}^2.
 \label{eq:local-density-difference-energy}
\end{align}

The momentum difference has the same positive principal part as the
$(n+1)$st equation and can be written as
\begin{align*}
 &\rho^{n+1}\bigl(Z_{u,t}^{n+1}+u^nZ_{u,x}^{n+1}\bigr)
 -\partial_x\bigl(\mu(\theta^n)^\alpha Z_{u,x}^{n+1}\bigr)\\
 &=-Z_\rho^{n+1}(u_t^n+u^{n-1}u_x^n)
   -\rho^{n+1}Z_u^n u_x^n
   +\partial_x\!\left(\mu\bigl[(\theta^n)^\alpha
                       -(\theta^{n-1})^\alpha\bigr]u_x^n\right)-\partial_x\!\left(RZ_\rho^{n+1}\theta^n
                         +R\rho^nZ_\theta^n\right).
\end{align*}
The temperature difference is
\begin{align*}
 &\rho^{n+1}\bigl(Z_{\theta,t}^{n+1}+u^nZ_{\theta,x}^{n+1}\bigr)
 -\partial_x\bigl(\kappa(\theta^n)^\beta Z_{\theta,x}^{n+1}\bigr)\\
 &=-Z_\rho^{n+1}(\theta_t^n+u^{n-1}\theta_x^n)
   -\rho^{n+1}Z_u^n\theta_x^n
   +\partial_x\!\left(\kappa\bigl[(\theta^n)^\beta
                       -(\theta^{n-1})^\beta\bigr]\theta_x^n\right)
   +Z_Q^n,
\end{align*}
where
\begin{align*}
 Z_Q^n=-R\bigl(\rho^{n+1}\theta^nu_x^n
                  -\rho^n\theta^{n-1}u_x^{n-1}\bigr)+\mu\bigl((\theta^n)^\alpha(u_x^n)^2
             -(\theta^{n-1})^\alpha(u_x^{n-1})^2\bigr).
\end{align*}
Each difference in $Z_Q^n$ is expanded by changing one factor at a time;
for instance,
\[
 (\theta^n)^\alpha(u_x^n)^2-(\theta^{n-1})^\alpha(u_x^{n-1})^2
 =\bigl[(\theta^n)^\alpha-(\theta^{n-1})^\alpha\bigr](u_x^n)^2
 +(\theta^{n-1})^\alpha Z_{u,x}^n(u_x^n+u_x^{n-1}).
\]
Because $\theta^n\geq c_0/2$, the mean-value theorem gives
\begin{equation}\label{eq:local-power-difference}
 \norm{(\theta^n)^\gamma-(\theta^{n-1})^\gamma}_{L^p}
 \leq C_{c_0,L,\gamma}\norm{Z_\theta^n}_{L^p},
 \qquad \gamma\in\{\alpha,\beta\},\quad 2\leq p\leq\infty.
\end{equation}

Testing the last two difference equations by $Z_u^{n+1}$ and
$Z_\theta^{n+1}$, respectively, the transport terms combine with the
continuity equation for $\rho^{n+1}$ and hence introduce no derivative of
the density.  The two principal diffusion terms give
\begin{align*}
 \mu(c_0/2)^\alpha\norm{Z_{u,x}^{n+1}}_{L^2}^2,
 \qquad
 \kappa(c_0/2)^\beta\norm{Z_{\theta,x}^{n+1}}_{L^2}^2.
\end{align*}
All divergence-form source terms are integrated by parts and paired with
the corresponding derivative of the test function.  For example,
\begin{align*}
 \left|\int_0^1R(Z_\rho^{n+1}\theta^n
                  +\rho^nZ_\theta^n)Z_{u,x}^{n+1}\dd x\right|
 \leq \varepsilon\norm{Z_{u,x}^{n+1}}_{L^2}^2
   +C_{L,\varepsilon}\bigl(\norm{Z_\rho^{n+1}}_{L^2}^2
                           +\norm{Z_\theta^n}_{L^2}^2\bigr),\\
 \left|\int_0^1\mu\bigl[(\theta^n)^\alpha
                   -(\theta^{n-1})^\alpha\bigr]u_x^n
                   Z_{u,x}^{n+1}\dd x\right|
\leq \varepsilon\norm{Z_{u,x}^{n+1}}_{L^2}^2
   +C_{L,\varepsilon}\norm{u_x^n}_{L^\infty}^2
                           \norm{Z_\theta^n}_{L^2}^2.
\end{align*}
The nondivergence terms are handled by the same Young inequality and by
the uniform $\mathcal X_\tau$ bound.  After $\varepsilon$ is chosen to
absorb all new-gradient terms, the two parabolic energies and
\eqref{eq:local-density-difference-energy} imply
\begin{align*}
 d_{n+1}
 \leq C_{\delta,c_0,L}\int_0^\tau
 \bigl(1+\norm{(u^n,\theta^n)}_{H^3}^2
          +\norm{(u_t^n,\theta_t^n)}_{H^1}^2\bigr)
 \bigl[\norm{Z_\rho^{n+1}}_{L^2}^2
       +\norm{Z^n}_{H^1\times H^1}^2\bigr]\dd t.
\end{align*}
The term involving $Z_\rho^{n+1}$ is absorbed by Gronwall's inequality.
For the remaining terms, the one-dimensional interpolation estimate and
Cauchy--Schwarz give
\begin{align*}
 \int_0^\tau\norm z_{L^\infty}^2\dd t
 &\leq C\int_0^\tau\norm z_{L^2}\norm z_{H^1}\dd t
 \leq C\tau^{1/2}
  \left(\sup_{t\leq\tau}\norm z_{L^2}^2
       +\int_0^\tau\norm{z_x}_{L^2}^2\dd t\right).
\end{align*}
Combining this with the $L^2$-in-time high norms in the
$\mathcal X_\tau$ ball (and weakening $\tau^{1/2}$ to $\tau^{1/4}$ when
$0<\tau\leq1$) gives
\begin{equation}\label{eq:local-contraction}
 d_{n+1}\leq C_{\delta,c_0,L}\tau^{1/4}d_n.
\end{equation}
Reducing $\tau$ makes the coefficient in \eqref{eq:local-contraction}
strictly less than one.  The fixed point is
the asserted local strong solution, while the uniform $\mathcal X_\tau$
bound and weak lower semicontinuity recover the full regularity stated in
the lemma.  The smooth functions
$\theta\mapsto\theta^\alpha$ and $\theta\mapsto\theta^\beta$ are locally
Lipschitz in $H^2$ on $\theta\geq c_0/2$.

At a positive endpoint $T$, the lower bounds and the norms in
\eqref{eq:finite-high} control the $H^2$ traces of the three unknowns and all
coefficient norms in Lemma~\ref{lem:linear-local}.  The compatibility
residuals are supplied by the equations:
\begin{align*}
 g_1(T)&=\rho(T)^{-1/2}
 \left[\partial_x\bigl(\mu\theta(T)^\alpha u_x(T)\bigr)
       -\partial_x\bigl(R\rho(T)\theta(T)\bigr)\right]
       =\sqrt{\rho(T)}\,\dot u(T),\\
 g_2(T)&=\rho(T)^{-1/2}
 \left[\partial_x\bigl(\kappa\theta(T)^\beta\theta_x(T)\bigr)
       -R\rho(T)\theta(T)u_x(T)+\mu\theta(T)^\alpha u_x(T)^2\right]
       =\sqrt{\rho(T)}\,\dot\theta(T).
\end{align*}
The bracketed expressions belong to $L^2$ by the $H^2$ endpoint bounds,
and multiplication by $\rho(T)^{-1/2}\in H^2\hookrightarrow L^\infty$ is
legitimate because $\inf\rho(T)>0$.  Thus $g_1(T),g_2(T)\in L^2$, so the
compatibility hypothesis is recovered at time $T$.  The local lifespan
depends only on these endpoint norms and the two positive lower bounds.
The construction can therefore be restarted at $T$, proving the
continuation criterion.  This lemma is used only for strictly positive
approximations; no local vacuum theorem is invoked.
\end{proof}

\begin{proof}[Proof of Theorem~\ref{thm:main}]
For $0<\delta<1$ set
\[
 \rho_0^\delta=\rho_0+\delta,\qquad
 u_0^\delta=u_0,\qquad \theta_0^\delta=\theta_0.
\]
Its conserved mean is $m^\delta=m+\delta$.  In every estimate for the
$\delta$-problem, the relative entropy, $\Pi$, and the centered pressure use
$m^\delta$ in place of $m$.  Since
\[
 m^\delta\in[m_*,m^*+1],\qquad
 \norm{\rho_0^\delta}_{H^2}\leq K_0+1,
\]
and these elementary bounds will be absorbed into fixed enlarged
data-class constants below.
The corresponding compatibility residuals are defined by division by
$\sqrt{\rho_0^\delta}$.  More explicitly,
\begin{equation*}
 g_1^\delta=\frac{\sqrt{\rho_0}g_1-R\delta\theta_{0x}}{\sqrt{\rho_0+\delta}},
 \qquad g_2^\delta=\frac{\sqrt{\rho_0}g_2-R\delta\theta_0u_{0x}}{\sqrt{\rho_0+\delta}}.
\end{equation*}
Since $\delta/\sqrt{\rho_0+\delta}\leq\sqrt\delta$,
\begin{equation}\label{eq:gdelta}
 \norm{g_1^\delta}_{L^2}+\norm{g_2^\delta}_{L^2}
 \leq\norm{g_1}_{L^2}+\norm{g_2}_{L^2}
 +C\sqrt\delta\bigl(\norm{\theta_{0x}}_{L^2}
                    +\bt\norm{u_{0x}}_{L^2}\bigr).
\end{equation}
Combining \eqref{eq:gdelta} with \eqref{eq:data-dynamic},
\eqref{eq:natural-data}, and $\bt\geq1$, we obtain
\begin{align}\label{eq:natural-data-delta}
 &\bt^{-2\alpha}\norm{g_1^\delta}_{L^2}^2
  +\bt^{-2\beta}\norm{g_2^\delta}_{L^2}^2\leq C\left(\bt^{-2\alpha}\norm{g_1}_{L^2}^2
                  +\bt^{-2\beta}\norm{g_2}_{L^2}^2\right)+C\delta\left(
   \bt^{-2\alpha}\norm{\theta_{0x}}_{L^2}^2
  +\bt^{2-2\beta}\norm{u_{0x}}_{L^2}^2\right)
 \leq C K_0.
\end{align}
Consequently, there is a constant
\[
 K_0^\sharp=C_\sharp(1+K_0),\qquad
 C_\sharp=C_\sharp(R,\mu,\kappa,\alpha,\beta)\geq1,
\]
such that
\begin{equation}\label{eq:data-delta}
 \norm{\rho_0^\delta}_{H^2}\leq K_0^\sharp,\qquad
 \norm{u_0}_{H^2}^2+\norm{\theta_0-\bt}_{H^2}^2
 +\bt^{-2\alpha}\norm{g_1^\delta}_{L^2}^2
 +\bt^{-2\beta}\norm{g_2^\delta}_{L^2}^2\leq K_0^\sharp
 \quad(0<\delta<1).
\end{equation}
Set $m^{*,\sharp}=m^*+1$.  In the statement of the theorem, choose
$\Theta_0$ to be the larger of the thresholds supplied by
Propositions~\ref{prop:initial-layer} and \ref{prop:uniform}, with
$K_0,m^*$ replaced by $K_0^\sharp,m^{*,\sharp}$.  This choice still depends
only on the quantities listed in Theorem~\ref{thm:main}.  Thus all initial
bounds needed below are uniform in $\delta$; in particular,
\[
 \norm{\sqrt{\rho^\delta}\dot u^\delta(0)}_{L^2}=\norm{g_1^\delta}_{L^2},\qquad
 \norm{\sqrt{\rho^\delta}\dot\theta^\delta(0)}_{L^2}=\norm{g_2^\delta}_{L^2}.
\]

Lemma~\ref{lem:positive-local} gives a maximal positive-density solution.
The identities in Sections~3--4, written there first for smooth solutions,
are valid for this strong solution by Steklov averaging in time and standard
spatial regularization; the compatibility conditions supply the initial
traces.  Equivalently, one may perform the estimates on the regularized
iterates above and pass to the strong limit.
Choose the bootstrap constants in Propositions~\ref{prop:initial-layer} and
\ref{prop:uniform} with $K_0,m^*$ replaced by
$K_0^\sharp,m^{*,\sharp}$.
For the positive-density solution, Lemma~\ref{lem:linear-local} gives
$u_t,\theta_t\in C([0,\tau];L^2)$ on each local interval.  Consequently
$\sqrt\rho\dot u$, $\sqrt\rho\dot\theta$, and the weighted functional
$\mathcal B_\sigma$ are continuous in time.  Thus the initial-layer
bootstrap holds on a short interval.  Proposition~\ref{prop:initial-layer}
improves it.  Since $\sigma\geq\bt^{1-\beta}$, its weighted bounds imply
\eqref{eq:finite-layer-base} on every current subinterval of $[0,1]$; the
proofs of Lemmas~\ref{lem:density-high}--\ref{lem:elliptic-recovery} then
supply all continuation norms, with constants allowed to depend on $\bt$
but not on $\delta$.  The positive solution therefore reaches $t=1$ and
satisfies \eqref{eq:layer-entry}.  Starting at $t=1$,
Proposition~\ref{prop:uniform} improves the post-layer bootstrap, so the
same continuation argument extends the solution globally.  More explicitly,
on every finite interval the flow formula gives
\[
 \rho^\delta(X^\delta(a,t),t)
 =\rho_0^\delta(a)
 \exp\left(-\int_0^t u_x^\delta(X^\delta(a,s),s)\dd s\right)>0.
\]
Consequently,
\[
 \inf_{x,\,0\leq t\leq T}\rho^\delta(x,t)
 \geq\delta\exp\left(
  -\int_0^T\norm{u_x^\delta(s)}_{L^\infty}\dd s\right)>0.
\]
The temperature has the lower bounds furnished successively by
Propositions~\ref{prop:initial-layer} and \ref{prop:uniform}, and
Corollary~\ref{cor:continuation} supplies all remaining continuation norms.
Therefore the $\delta$-solution is global.

Fix $T<\infty$.  Propositions~\ref{prop:initial-layer} and
\ref{prop:uniform}, together with Corollary~\ref{cor:continuation}, are
uniform in $\delta$.  The
equations additionally give
\begin{align}
 \norm{\rho_t^\delta}_{H^1}
 &\leq C\norm{u^\delta}_{H^2}\norm{\rho^\delta}_{H^2},
 \label{eq:compact-rhot}\\
 \norm{u_t^\delta}_{H^1}
 &\leq\norm{\dot u^\delta}_{H^1}
      +C\norm{u^\delta}_{H^2}^2,
 \label{eq:compact-ut}\\
 \norm{\theta_t^\delta}_{H^1}
 &\leq\norm{\dot\theta^\delta}_{H^1}
      +C\norm{u^\delta}_{H^2}\norm{\theta^\delta}_{H^2}.
 \label{eq:compact-thetat}
\end{align}
Hence $\rho_t^\delta$ is bounded in $L^\infty(0,T;H^1)$ and
$u_t^\delta,\theta_t^\delta$ are bounded in $L^2(0,T;H^1)$.
The Aubin--Lions--Simon compactness theorem, applied with
$H^2\Subset H^1\hookrightarrow L^2$, yields, along a subsequence,
\begin{align*}
 \rho^\delta&\longrightarrow\rho
     &&\text{strongly in }C([0,T];H^1),\\
 u^\delta,\theta^\delta&\longrightarrow u,\theta
     &&\text{strongly in }C([0,T];H^1)
       \text{ and weakly in }L^2(0,T;H^3).
\end{align*}
In one dimension the strong convergences are uniform in space-time.  Thus
\[
 (\theta^\delta)^\alpha\longrightarrow\theta^\alpha,
 \qquad (\theta^\delta)^\beta\longrightarrow\theta^\beta,
 \qquad \rho^\delta\theta^\delta\longrightarrow\rho\theta
\]
uniformly.  Since $u_x^\delta\to u_x$ and
$\theta_x^\delta\to\theta_x$ strongly in
$C([0,T];L^2)$, it follows that
\begin{align*}
 (\theta^\delta)^\alpha u_x^\delta
 &\longrightarrow\theta^\alpha u_x,
 & (\theta^\delta)^\beta\theta_x^\delta
 &\longrightarrow\theta^\beta\theta_x
 &&\text{strongly in }L^2((0,T)\times\I),\\
 (u_x^\delta)^2&\longrightarrow u_x^2,
 & (\theta^\delta)^\alpha(u_x^\delta)^2
 &\longrightarrow\theta^\alpha u_x^2
 &&\text{strongly in }L^1((0,T)\times\I).
\end{align*}
The inertial and pressure products also converge strongly; for example,
\begin{align*}
 \rho^\delta u^\delta&\longrightarrow\rho u,&
 \rho^\delta(u^\delta)^2&\longrightarrow\rho u^2,&
 \rho^\delta\theta^\delta&\longrightarrow\rho\theta
 \qquad\text{in }C([0,T];L^2).
\end{align*}
Alternatively, in material form the weak convergence
$\dot u^\delta\rightharpoonup\dot u$ in $L^2H^1$, together with the
uniform convergence of $\rho^\delta$, gives
$\rho^\delta\dot u^\delta\rightharpoonup\rho\dot u$ in $L^2L^2$; the
same statement holds for the temperature material derivative.  Hence the
continuity equation passes to the limit through the first three products,
the momentum equation through the pressure and viscous-flux limits, and
the temperature equation through the heat-flux and viscous-heating limits.
Testing first with smooth functions compactly supported in time also shows
that the initial traces pass by the strong $C H^1$ convergence.  The zero
Dirichlet traces are closed under weak convergence in $H_0^1$.

The uniform pointwise estimates pass to the limit because the convergence
is uniform, while the energy and higher-order bounds pass by weak or
weak-star lower semicontinuity.  A diagonal extraction over
$T=1,2,\ldots$ therefore gives a global solution satisfying
\eqref{eq:pointwise-main}--\eqref{eq:finite-high}.

We finally recover the time continuity required by
Definition~\ref{def:strong}.  From
$u_t=\dot u-uu_x$ and $\theta_t=\dot\theta-u\theta_x$,
\eqref{eq:dot-poincare} and the finite-time bounds give
\[
 u_t,\theta_t\in L^2(0,T;H^1),\qquad
 u,\theta-\bt\in L^2(0,T;H^3).
\]
The Lions--Magenes lemma for
$H^3\subset H^2\subset H^1$ therefore yields
$u,\theta-\bt\in C([0,T];H^2)$.  The $H^2$ transport energy associated
with
\[
 \rho_{xxt}+u\rho_{xxx}+3u_x\rho_{xx}
 +3u_{xx}\rho_x+\rho u_{xxx}=0
\]
is absolutely continuous.  Multiplication by $\rho_{xx}$ and integration
by parts give the exact identity
\begin{equation}\label{eq:rho-top-continuity}
 \frac12\frac{\dd}{\dd t}\norm{\rho_{xx}}_{L^2}^2+\frac52\int_0^1u_x\rho_{xx}^2\dd x
 +3\int_0^1u_{xx}\rho_x\rho_{xx}\dd x+\int_0^1\rho u_{xxx}\rho_{xx}\dd x=0.
\end{equation}
Every integral term in \eqref{eq:rho-top-continuity} belongs to $L^1(0,T)$ by the
finite-time bounds.  Consequently,
\[
 \frac{\dd}{\dd t}\norm{\rho_{xx}(t)}_{L^2}^2\in L^1(0,T).
\]
The estimates already imply $\rho\in C_w([0,T];H^2)$; weak continuity
together with continuity of the $H^2$ norm gives
$\rho\in C([0,T];H^2)$.  The continuity equation then gives
$\rho_t\in C([0,T];H^1)$.  Weak-star compactness in
$L^\infty(0,T;L^2)$ identifies the limits of
$\sqrt{\rho^\delta}\dot u^\delta$ and
$\sqrt{\rho^\delta}\dot\theta^\delta$ with the corresponding material
derivatives of the limit: indeed, \eqref{eq:dot-poincare} gives, after
subsequence extraction,
\[
 \dot u^\delta\rightharpoonup\dot u,\qquad
 \dot\theta^\delta\rightharpoonup\dot\theta
 \quad\hbox{in }L^2(0,T;H_0^1),
\]
while $\sqrt{\rho^\delta}\to\sqrt\rho$ uniformly on
$[0,1]\times[0,T]$.  Thus the limit is a global strong solution, and all
asserted estimates follow by lower semicontinuity.

\end{proof}

\section{Exponential relaxation and persistence of vacuum}
\label{sec:persistence}

The global dissipation in Theorem~\ref{thm:main} does not by itself yield a
decay rate.  We first recover the missing density coercivity from the
effective flux, without differentiating the density, and then combine the
relative entropy with the differentiated energy $\Ycal$.

\subsection{Vacuum-compatible coercivity}

Set
\begin{equation}\label{eq:decay-dissipations}
 \Dcal_0=\bt^{\alpha-1}\norm{u_x}_{L^2}^2+\bt^{\beta-2}\norm{\theta_x}_{L^2}^2,
 \qquad
 \Dcal_h=\bt^{-2}(\norm{\sqrt\rho\dot u}_{L^2}^2+\norm{\sqrt\rho\dot\theta}_{L^2}^2)
          +\bt^{\alpha-2}\norm{\dot u_x}_{L^2}^2+\bt^{\beta-2}\norm{\dot\theta_x}_{L^2}^2.
\end{equation}

The decay argument first needs an algebraic replacement for the missing
unweighted density-gradient dissipation.  The next lemma recovers
$\rho-m$ directly from the effective flux and converts the instantaneous
dissipations into coercive control of the full energy.  This is the input
that turns the post-layer differential inequality into a strict Lyapunov
inequality.

\begin{lemma}[Instantaneous density coercivity]
\label{lem:decay-coercivity}
Suppose that
$0\leq\rho\leq M$ and $\bt/2\leq\theta\leq3\bt/2$.  Then
\begin{align}
 |\Pi-Rm\bt|&\leq C \norm{\theta_x}_{L^2},\label{eq:Pi-equilibrium}\\
 \norm{\rho-m}_{L^2}
 &\leq C\left(
   \bt^{\alpha-1}\norm{u_x}_{L^2}+\bt^{-1}\norm{\sqrt\rho\dot u}_{L^2}+\bt^{-1}\norm{\theta_x}_{L^2}\right),
 \label{eq:density-coercivity}
\end{align}
and
\begin{equation}\label{eq:instantaneous-coercivity}
 \Acal(t)+\bt^{-1}\Ecal(t)
 \leq C\bt^{\alpha-1}\bigl(\Dcal_0(t)+\Dcal_h(t)\bigr).
\end{equation}
The constant is independent of a positive lower bound for $\rho$.
\end{lemma}

\begin{proof}
Put $q=\theta/\bt$.  Since $1/2\leq q\leq3/2$, definition
\eqref{eq:Pi} gives
\begin{align*}
 \frac{\Pi}{R\bt}-m
 &=\frac{\displaystyle
    \int_0^1\rho(q^{1-\alpha}-1)\dd x
    +m\int_0^1(1-q^{-\alpha})\dd x}
   {\displaystyle\int_0^1q^{-\alpha}\dd x}.
\end{align*}
The mean-value theorem and the bounds for $q$ therefore imply
\[
 |\Pi-Rm\bt|
 \leq C\bt\norm{q-1}_{L^\infty}
 \leq C\norm{\theta-\bt}_{L^\infty}
 \leq C\norm{\theta_x}_{L^2},
\]
which proves \eqref{eq:Pi-equilibrium}.  From
$R\rho\theta=\mu\theta^\alpha u_x-F$, \eqref{eq:F-anchor}, and
\eqref{eq:Pi-equilibrium},
\begin{equation}\label{eq:pressure-by-dissipation}
 \norm{R\rho\theta-Rm\bt}_{L^2}
 \leq C\left(\bt^\alpha \norm{u_x}_{L^2}+\norm{\sqrt\rho\dot u}_{L^2}+\norm{\theta_x}_{L^2}\right).
\end{equation}
On the other hand,
$R\rho\theta-Rm\bt=R\bt(\rho-m)+R\rho(\theta-\bt)$.
The density upper bound and
the Dirichlet Poincar\'e inequality give
\[
 R\bt\norm{\rho-m}_{L^2}
 \leq\norm{R\rho\theta-Rm\bt}_{L^2}+C\norm{\theta-\bt}_{L^2}
 \leq\norm{R\rho\theta-Rm\bt}_{L^2}+C\norm{\theta_x}_{L^2}.
\]
Together with \eqref{eq:pressure-by-dissipation}, this proves
\eqref{eq:density-coercivity}.

For $0\leq\rho\leq M$, $m\in[m_*,m^*]$, and
$\bt/2\leq\theta\leq3\bt/2$, the functions in
\eqref{eq:relative-functions} also satisfy
\[
 h(\rho\mid m)\leq C|\rho-m|^2,
 \qquad
 \Psi(\theta\mid\bt)\leq C\bt^{-1}|\theta-\bt|^2.
\]
Hence, by the boundary Poincar\'e inequality,
\begin{equation}\label{eq:entropy-upper-decay}
 \Ecal(t)\leq C\left(
  \bt\norm{\rho-m}_{L^2}^2+\norm{u_x}_{L^2}^2+\bt^{-1}\norm{\theta_x}_{L^2}^2\right).
\end{equation}
Insert \eqref{eq:density-coercivity} into
\eqref{eq:entropy-upper-decay}.  Since $\bt\geq1$ and $\alpha>2$, this
yields
\[
 \bt^{-1}\Ecal(t)
 \leq C\bt^{\alpha-1}\bigl(\Dcal_0(t)+\Dcal_h(t)\bigr).
\]
The definition \eqref{eq:A} gives the same upper bound for $\Acal(t)$.
This proves \eqref{eq:instantaneous-coercivity}.
\end{proof}

\subsection{Enhanced Lyapunov functional}

We now combine the bookkeeping inequality from
Lemma~\ref{lem:bookkeeping}, the relative-entropy identity, and the
instantaneous coercivity from Lemma~\ref{lem:decay-coercivity}.  The resulting
functional controls both the differentiated energy and the density part of
the entropy, and hence yields a genuine exponential rate rather than only
time-integrability of the dissipation.

\begin{proposition}[Strict Lyapunov inequality]
\label{prop:strict-lyapunov}
Let a smooth positive-density solution satisfy the bounds furnished by
Proposition~\ref{prop:uniform} on $[1,\infty)$.  There is a fixed $k_*>0$,
independent of $\bt$ and of the density lower bound, such that, for $t\geq1$,
\begin{equation}\label{eq:enhanced-Lyapunov}
 \Zcal=\Ycal+k_*\bt^{-1}\Ecal
\end{equation}
satisfies
\begin{align}
 &c\bigl(\Acal+\bt^{-1}\Ecal\bigr)
 \leq\Zcal
 \leq C\bigl(\Acal+\bt^{-1}\Ecal\bigr),
 \label{eq:Z-equivalence}\\
& \Zcal'(t)+c\bigl(\Dcal_0(t)+\Dcal_h(t)\bigr)\leq0,
 \label{eq:Z-differential}\\
& \Zcal(t)\leq\Zcal(1)
  \exp\!\left(-c\bt^{1-\alpha}(t-1)\right).
 \label{eq:Z-decay}
\end{align}
Moreover,
\begin{equation}\label{eq:Z-tail}
 \int_t^\infty(\Dcal_0+\Dcal_h)(s)\dd s
 \leq C\Zcal(1)
  \exp\!\left(-c\bt^{1-\alpha}(t-1)\right).
\end{equation}
\end{proposition}

\begin{proof}
On $[1,\infty)$, Proposition~\ref{prop:uniform} gives
\[
 0\leq\rho\leq M_1,\qquad
 \frac12\bt\leq\theta\leq\frac32\bt,
 \qquad \Acal\leq K_1.
\]
Apply Lemma~\ref{lem:bookkeeping} with $M=M_1$ and $K=K_1$.
By \eqref{eq:theta-choice}, the coefficient of the last term in
\eqref{eq:closure-differential} is at most $\min\{c/2,1\}$.
Its $\Dcal_h$ part is absorbed into the left-hand side and its
$\Dcal_0$ part remains on the right.  Hence there are fixed
$c_1,C_1>0$, independent of $\bt$ and $\inf\rho$, such that
\begin{equation}\label{eq:Y-decay-start}
 \Ycal'(t)+c_1\Dcal_h(t)\leq C_1\Dcal_0(t).
\end{equation}
The entropy identity can be written as
\[
 \Ecal'+\mu\bt\int_0^1\theta^{\alpha-1}u_x^2\dd x
 +\kappa\bt\int_0^1\theta^{\beta-2}\theta_x^2\dd x=0.
\]
Since $\theta\geq\bt/2$, $\alpha>2$, and $\beta>2$, we have
$\theta^{\alpha-1}\geq2^{1-\alpha}\bt^{\alpha-1}$ and
$\theta^{\beta-2}\geq2^{2-\beta}\bt^{\beta-2}$.  Therefore
\begin{equation}\label{eq:E-decay-start}
 \Ecal'(t)+c_2\bt\Dcal_0(t)\leq0,
 \qquad c_2:=\min\{\mu2^{1-\alpha},\kappa2^{2-\beta}\}>0.
\end{equation}

Choose $k_*=2C_1/c_2$.  Multiplying \eqref{eq:E-decay-start} by
$k_*\bt^{-1}$ and adding it to \eqref{eq:Y-decay-start} gives
\[
 \Zcal'+c_1\Dcal_h+(k_*c_2-C_1)\Dcal_0\leq0.
\]
Since $k_*c_2-C_1=C_1$, with
$c_3:=\min\{c_1,C_1\}$ this becomes
\begin{equation*}
 \Zcal'+c_3(\Dcal_0+\Dcal_h)\leq0,
\end{equation*}
which is \eqref{eq:Z-differential}.

Let $c_Y,C_Y>0$ be the constants in \eqref{eq:Y-equivalence}.  Since
$\bt^{-\alpha-1}\leq\bt^{-1}$, we have
\begin{align*}
 \Zcal
 &\geq c_Y\Acal+c_Y\bt^{-\alpha-1}\Ecal
       +k_*\bt^{-1}\Ecal
 \geq\min\{c_Y,k_*\}(\Acal+\bt^{-1}\Ecal),\\
 \Zcal
 &\leq C_Y\Acal+C_Y\bt^{-\alpha-1}\Ecal
       +k_*\bt^{-1}\Ecal
 \leq\max\{C_Y,C_Y+k_*\}(\Acal+\bt^{-1}\Ecal).
\end{align*}
This proves \eqref{eq:Z-equivalence}, and in particular $\Zcal\geq0$.

The upper equivalence and Lemma~\ref{lem:decay-coercivity} give a fixed
$C_L>0$ such that
\[
 \Zcal(t)\leq C_L\bt^{\alpha-1}
   \bigl(\Dcal_0(t)+\Dcal_h(t)\bigr).
\]
Thus
\[
 \Dcal_0+\Dcal_h\geq C_L^{-1}\bt^{1-\alpha}\Zcal,
\]
and \eqref{eq:Z-differential} implies
\[
 \Zcal'+\lambda_0\bt^{1-\alpha}\Zcal\leq0,
 \qquad \lambda_0:=c_3C_L^{-1}>0.
\]
Multiplying by
$\exp(\lambda_0\bt^{1-\alpha}(t-1))$ and integrating from $1$ to $t$
proves \eqref{eq:Z-decay}.

Finally, integration of \eqref{eq:Z-differential} over $[t,T]$ yields
\[
 c_3\int_t^T(\Dcal_0+\Dcal_h)(s)\dd s
 \leq\Zcal(t)-\Zcal(T)\leq\Zcal(t).
\]
Letting $T\to\infty$ by monotone convergence and then using
\eqref{eq:Z-decay} proves \eqref{eq:Z-tail}.
\end{proof}

\begin{proof}[Proof of the relaxation assertions in
Theorem~\ref{thm:relaxation}]
Apply Proposition~\ref{prop:strict-lyapunov} to the positive-density
$\delta$-solutions constructed in Section~\ref{sec:construction}, using
their conserved means $m^\delta=m+\delta$.  The enlarged data-class bounds
\eqref{eq:data-delta}, Proposition~\ref{prop:initial-layer}, the equivalence
\eqref{eq:Z-equivalence}, and $\Ecal^\delta(1)\leq C\bt$ show that
\begin{align*}
 \Zcal^\delta(1)
 &=\Ycal^\delta(1)+k_*\bt^{-1}\Ecal^\delta(1)
 \leq C\Acal^\delta(1)
      +C\bt^{-\alpha-1}\Ecal^\delta(1)
      +k_*\bt^{-1}\Ecal^\delta(1)\leq C
\end{align*}
uniformly in $\delta$ and $\bt$.  Consequently,
\begin{equation}\label{eq:delta-decay}
 \Acal^\delta(t)+\bt^{-1}\Ecal^\delta(t)
 \leq C\exp\!\left(-c\bt^{1-\alpha}(t-1)\right),
 \qquad t\geq1,
\end{equation}
and the corresponding version of \eqref{eq:Z-tail} holds with the same
constants.

Fix $T>1$.  The compactness established in the proof of
Theorem~\ref{thm:main} gives strong convergence of
$\rho^\delta,u^\delta,\theta^\delta$ in $C([0,T];H^1)$ and weak convergence
of the material derivatives in the spaces stated in
Definition~\ref{def:strong}.  The state-variable terms in
\eqref{eq:delta-decay} pass to the vacuum limit by strong convergence.  In
fact, $m^\delta\to m$ and hence, at every fixed time,
\begin{align*}
 \bt^{\alpha-2}\norm{u_x^\delta}_{L^2}^2\longrightarrow
       \bt^{\alpha-2}\norm{u_x}_{L^2}^2,\qquad
 \bt^{\beta-2}\norm{\theta_x^\delta}_{L^2}^2\longrightarrow
       \bt^{\beta-2}\norm{\theta_x}_{L^2}^2,\qquad
 \bt^{-1}\Ecal^\delta\longrightarrow\bt^{-1}\Ecal.
\end{align*}
The material derivatives satisfy, after extraction,
\begin{align*}
 \sqrt{\rho^\delta}\,\dot u^\delta
 \stackrel{*}{\rightharpoonup}\sqrt\rho\,\dot u,\qquad
 \sqrt{\rho^\delta}\,\dot\theta^\delta
 \stackrel{*}{\rightharpoonup}\sqrt\rho\,\dot\theta
 \quad\text{in }L^\infty(1,T;L^2).
\end{align*}
Weak-star lower semicontinuity of their scalar $L^2$ norms therefore passes
\eqref{eq:delta-decay} to the limit for almost every $t$.

For the tail estimate, first fix $1\leq t<T$.  Weak lower semicontinuity
for each nonnegative quadratic dissipation gives
\begin{align*}
 \int_t^T(\Dcal_0+\Dcal_h)(s)\dd s
 &\leq\liminf_{\delta\to0}
       \int_t^T(\Dcal_0^\delta+\Dcal_h^\delta)(s)\dd s
 \leq C\exp\!\left(-c\bt^{1-\alpha}(t-1)\right).
\end{align*}
Letting $T\to\infty$ and using monotone convergence proves
\eqref{eq:decay-tail}.

Finally, \eqref{eq:entropy-coercivity}, the definition of $\Acal$, and the
Dirichlet Poincar\'e inequalities yield
\begin{align*}
 \norm{\rho-m}_{L^2}^2
 &\leq C\bt^{-1}\Ecal,\\
 \bt^{\alpha-2}\norm u_{H^1}^2
 &\leq C\bt^{\alpha-2}\norm{u_x}_{L^2}^2\leq C\Acal,\\
 \bt^{\beta-2}\norm{\theta-\bt}_{H^1}^2
 &\leq C\bt^{\beta-2}\norm{\theta_x}_{L^2}^2\leq C\Acal.
\end{align*}
The two material-derivative terms in \eqref{eq:low-order-decay} occur
explicitly in $\Acal$.  Inserting the limiting form of
\eqref{eq:delta-decay} proves \eqref{eq:low-order-decay}.  The terms
involving only $\rho,u$, and $\theta$ satisfy the estimate at every time by
their $C([0,T];H^1)$ continuity; the material-derivative terms are
understood at their Lebesgue times.

It remains to prove \eqref{eq:density-Lp-decay}.  The uniform density upper
bound and \eqref{eq:low-order-decay} imply, for $2\leq p<\infty$,
\[
 \norm{\rho-m}_{L^p}
 \leq\norm{\rho-m}_{L^\infty}^{1-2/p}
      \norm{\rho-m}_{L^2}^{2/p}
 \leq C_p\exp\!\left(
    -\frac{\lambda_{\rm d}}{p}(t-1)\right).
\]
For $1\leq p\leq2$, the finite measure of $\I$ gives the same conclusion
with exponent $\lambda_{\rm d}/2$.  This proves all the relaxation
assertions of Theorem~\ref{thm:relaxation}.
\end{proof}

\subsection{Transported vacuum and proof of layer sharpening}

The preceding estimates describe decay in integral norms but do not by
themselves identify the evolution of an initial vacuum point.  The next
proposition supplies this geometric information through the particle flow;
it is then combined with the density-gradient estimate to obtain the sharp
vacuum-layer bounds.

\begin{proposition}[Vacuum is transported]\label{prop:vacuum}
Let $(\rho,u,\theta)$ be a solution furnished by
Theorem~\ref{thm:main} and let
$X_t(a,t)=u(X(a,t),t)$, $X(a,0)=a$.  Then
\begin{equation}\label{eq:flow-formula}
 X_a(a,t)=\exp\left(\int_0^t u_x(X(a,s),s)\dd s\right)>0,
\qquad
 \rho(X(a,t),t)X_a(a,t)=\rho_0(a).
\end{equation}
In particular, the zero set is transported by the bi-Lipschitz flow on
every finite time interval.
\end{proposition}

\begin{proof}
Fix an arbitrary $T<\infty$.  We first justify the flow map and its
invertibility on $[0,T]$.  By Definition~\ref{def:strong} and the
one-dimensional embedding $H^1(0,1)\hookrightarrow L^\infty(0,1)$,
\begin{equation*}
 L_T:=\int_0^T\norm{u_x(s)}_{L^\infty}\dd s
 \leq C\int_0^T\norm{u(s)}_{H^2}\dd s<\infty.
\end{equation*}
Thus $u(\cdot,t)$ is spatially Lipschitz with an integrable Lipschitz
constant.  The nonautonomous ODE defining $X$ consequently has a unique
solution for every $a\in[0,1]$.  Moreover, because the boundary condition
gives $u(0,t)=u(1,t)=0$, uniqueness implies
\begin{equation*}
 X(0,t)=0,\qquad X(1,t)=1,
 \qquad 0\leq t\leq T,
\end{equation*}
so every trajectory remains in $[0,1]$.

The same spatial regularity permits differentiation of the flow with
respect to its label.  The derivative $J(a,t):=X_a(a,t)$ satisfies the
variational equation
\begin{equation*}
 J_t(a,t)=u_x(X(a,t),t)J(a,t),
 \qquad J(a,0)=1.
\end{equation*}
Solving this scalar equation gives
\begin{equation*}
 X_a(a,t)=\exp\!\left(\int_0^t
 u_x(X(a,s),s)\dd s\right).
\end{equation*}
In particular,
\begin{equation}\label{eq:flow-Jacobian-bounds}
 e^{-L_T}\leq X_a(a,t)\leq e^{L_T}
 \qquad (a\in[0,1],\ 0\leq t\leq T).
\end{equation}
Hence $a\mapsto X(a,t)$ is strictly increasing.  Since it fixes the two
endpoints, it maps $[0,1]$ bijectively onto itself.  Integrating
\eqref{eq:flow-Jacobian-bounds} between two labels gives
\begin{equation*}
 e^{-L_T}|a-b|\leq|X(a,t)-X(b,t)|
 \leq e^{L_T}|a-b|,
\end{equation*}
and the inverse map has the same type of Lipschitz bound.  Thus the flow is
bi-Lipschitz on every finite time interval.

It remains to identify the density transported by this flow.  The strong
solution regularity
$\rho\in C([0,T];H^2)$ and
$\rho_t\in C([0,T];H^1)$ justifies the chain rule along each trajectory.
Using the continuity equation, we obtain
\begin{align*}
 \frac{\dd}{\dd t}\rho(X(a,t),t)
 &=\rho_t(X(a,t),t)+X_t(a,t)\rho_x(X(a,t),t)\\
 &=\rho_t+u\rho_x
 =-\rho u_x
 \qquad\text{at }(X(a,t),t).
\end{align*}
Combining this identity with the variational equation for $X_a$ yields
\begin{align*}
 \frac{\dd}{\dd t}\bigl(\rho(X(a,t),t)X_a(a,t)\bigr)
 &=(-\rho u_x)X_a+\rho(u_xX_a)=0.
\end{align*}
Since $X(a,0)=a$ and $X_a(a,0)=1$, integration in time gives
\begin{equation*}
 \rho(X(a,t),t)X_a(a,t)=\rho_0(a),
\end{equation*}
which proves the second identity in \eqref{eq:flow-formula}.  Equivalently,
\begin{equation*}
 \rho(X(a,t),t)=\rho_0(a)
 \exp\!\left(-\int_0^t u_x(X(a,s),s)\dd s\right).
\end{equation*}

Finally, let
$\mathcal V_0=\{a\in[0,1]:\rho_0(a)=0\}$ and
$\mathcal V(t)=\{x\in[0,1]:\rho(x,t)=0\}$.  Because $X_a>0$ and
$X(\cdot,t)$ is onto, the preceding identity shows both inclusions
\begin{equation*}
 \mathcal V(t)=X(\mathcal V_0,t).
\end{equation*}
Thus an initial vacuum point remains a vacuum point along its particle
path, while a particle starting with positive density cannot become vacuum
at any finite time.  This proves the asserted transport of the zero set.
\end{proof}

\begin{proof}[Proof of the vacuum-layer bounds]
Set $x_{\rm v}(t)=X(a_0,t)$.  Proposition~\ref{prop:vacuum} gives
$\rho(x_{\rm v}(t),t)=0$, which proves the first inequality in
\eqref{eq:gradient-growth}.  Put $f=\rho-m$.  Since
$\int_0^1f\dd x=0$, integration by parts gives the exact representation
\[
 f(x_{\rm v})=\int_0^{x_{\rm v}}s f_x(s)\dd s-\int_{x_{\rm v}}^1(1-s)f_x(s)\dd s.
\]
Hence, by the Cauchy--Schwarz inequality,
\[
 m=|f(x_{\rm v})|\leq\left(\frac{x_{\rm v}^3+(1-x_{\rm v})^3}{3}\right)^{1/2}
 \norm{\rho_x(t)}_{L^2}\leq\frac1{\sqrt{3}}\norm{\rho_x(t)}_{L^2}.
\]
This proves the time-independent lower bound in
\eqref{eq:gradient-growth}.

Since $f$ is continuous and has zero mean, there is $z_t\in[0,1]$ such
that $f(z_t,t)=0$.  Therefore
\begin{equation*}
 m^2=\left|f^2(x_{\rm v}(t),t)-f^2(z_t,t)\right|
 \leq2\norm{\rho(t)-m}_{L^2}\norm{\rho_x(t)}_{L^2}.
\end{equation*}
The density part of \eqref{eq:low-order-decay} now yields
\[
 \norm{\rho_x(t)}_{L^2}\geq \frac{m^2}{2C_{\rm d}^{1/2}}e^{\lambda_{\rm d}(t-1)/2}.
\]
By the definition of $c_{\rm v}$, this is precisely the exponential lower
bound in \eqref{eq:gradient-growth}.  Combining the two pointwise lower
bounds, squaring, and integrating proves
\eqref{eq:gradient-growth-integrated} and completes the proof of
Theorem~\ref{thm:relaxation}.
\end{proof}

\section*{Data availability}
Data sharing is not applicable to this article because no datasets were
generated or analyzed in the course of this study.

\section*{Declaration of competing interest}
The authors declare that they have no known competing financial interests or
personal relationships that could have appeared to influence the work
reported in this paper.

\section*{CRediT authorship contribution statement}
All authors contributed equally to the conceptualization, analysis, and
writing of the manuscript.  All authors read and approved the final version.

\section*{Acknowledgments}
This work was supported in part by the National Natural Science Foundation of
China under Grants 12371221, 12301277, 11831011, and 12161141004; the Science
Foundation of Zhejiang Sci-Tech University under Grant 25062122-Y; and the
Natural Science Foundation of Zhejiang Province under Grant ZCLQN26A0102.
The work was also partially supported by the Fundamental Research Funds for
the Central Universities and the Shanghai Frontiers Science Center of Modern
Analysis.

\bibliographystyle{unsrtnat}
\bibliography{references}

\end{document}